\documentclass[11pt,a4paper]{article}
\usepackage{amsmath}
\usepackage{amssymb}
\usepackage{graphicx}
\usepackage{amsfonts}
\usepackage{hyperref}
\usepackage{float}
\usepackage{pictex}
\newtheorem{theorem}{Theorem}[section]

\newtheorem{corollary}[theorem]{Corollary}
\newtheorem{definition}[theorem]{Definition}
\newtheorem{example}[theorem]{Example}
\newtheorem{lemma}[theorem]{Lemma}
\newtheorem{proposition}[theorem]{Proposition}
\newtheorem{remark}[theorem]{Remark}

\newenvironment{proof}[1][Proof]{\noindent\textbf{#1.} }{\rule{0.5em}{0.5em}}
\newenvironment{acknowledgement}[1][Acknowledgement]{\textbf{#1.} }{ }

\def\RM{\rm}

\def\limfunc#1{\mathop{\mathrm{#1}}}
\def\func#1{\mathop{\mathrm{#1}}\nolimits}

\def\tbigcup{\mathop{\textstyle \bigcup }}
\def\tbigcap{\mathop{\textstyle \bigcap }}

\def\enddoc{\end{document}}

\def\Qlb#1{#1}

\def\Qcb#1{#1} 

\def\FRAME#1#2#3#4#5#6#7#8
{
 \begin{figure}[ptbh]
 \begin{center}
 \includegraphics[height=#3]{#7}
 \caption{#5}
 \label{#6}
 \end{center}
 \end{figure}
}

\begin{document}

\title{Isotropic Markov semigroups on ultra-metric spaces\thanks{%
Version of May 09, 2014. Mathematics Subject Classification: 05C05, 47S10,
60J25, 81Q10}}
\author{Alexander Bendikov\thanks{%
Supported by the Polish Government Scientific Research Fund, Grant
2012/05/B/ST 1/00613} \and Alexander Grigor'yan\thanks{%
Supported by SFB 701 of German Research Council} \and Christophe Pittet%
\thanks{%
Supported by the CNRS, France} \and Wolfgang Woess\thanks{%
Supported by Austrian Science Fund projects FWF W1230-N13 and FWF P24028-N18}%
}
\date{Dedicated to the memories of\\
V.S. Vladimirov (1923--2012) and M.H. Taibleson (1929--2004)}
\maketitle

\begin{abstract}
Let $(X,d)$ be a separable ultra-metric space with compact balls. Given a
reference measure $\mu $ on $X$ and a distance distribution function $\sigma 
$ on $[0\,,\,\infty )$, we construct a symmetric Markov semigroup $%
\{P^{t}\}_{t\geq 0}$ acting in $L^{2}(X,\mu )$. Let $\{\mathcal{X}_{t}\}$ be
the corresponding Markov process. We obtain upper and lower bounds of its
transition density and its Green function, give a transience criterion,
estimate its moments and describe the Markov generator $\mathcal{L}$ and its
spectrum which is pure point. In the particular case when $X=\mathbb{Q}%
_{p}^{n}$, where $\mathbb{Q}_{p}$ is the field of $p$-adic numbers, our
construction recovers the Taibleson Laplacian (spectral multiplier), and we
can also apply our theory to the study of the Vladimirov Laplacian. Even in
this well established setting, several of our results are new. We also
elaborate the relation between our processes and Kigami's jump processes on
the boundary of a tree which are induced by a random walk. In conclusion, we
provide examples illustrating the interplay between the fractional
derivatives and random walks.
\end{abstract}

\tableofcontents

\section{Introduction}

\label{Intro}

\setcounter{equation}{0}In the past three decades there has been an
increasing interest in various constructions of Markov chains on
ultra-metric spaces, such as the Cantor set or the field of $p$-adic
numbers. In this paper we introduce and study a class of symmetric Markov
semigroups and their generators on ultra-metric spaces. Our construction is
very transparent, and it leads to a number of new results as well as to a
better understanding of previously known results.

Let $\left( X,d\right) $ be a metric space. The metric $d$ is called an 
\emph{ultra-metric} if it satisfies the ultra-metric inequality%
\begin{equation}
d(x,y)\leq \max \{d(x,z),d(z,y)\},  \label{um}
\end{equation}%
that is obviously stronger than the usual triangle inequality. In this case $%
\left( X,d\right) $ is called an ultra-metric space.

We will always assume in addition that the ultra-metric space $\left(
X,d\right) $ in question is separable, and that every closed ball 
\begin{equation}
B_{r}(x)=\{y\in X:d(x,y)\leq r\}  \label{Br}
\end{equation}
is compact. The latter implies that $\left( X,d\right) $ is complete.

The ultra-metric property (\ref{um}) implies that the balls in an
ultra-metric space $\left( X,d\right) $ look very differently from familiar
Euclidean balls. In particular, any two ultra-metric balls of the same
radius are either disjoint or identical. Consequently, the collection of all
distinct balls of the same radius $r$ forms a partition of $X$.

One of the best known examples of an ultra-metric space is the field $%
\mathbb{Q}_{p}$ of $p$-adic numbers endowed with the $p$-adic norm $\Vert
x\Vert _{p}$ and the $p$-adic ultra-metric $d(x,y)=\Vert x-y\Vert _{p}\,$.
Moreover, for any integer $n\geq 1$, the $p$-adic $n$-space $\mathbb{Q}%
_{p}^{n}=\mathbb{Q}_{p}\times ...\times \mathbb{Q}_{p}$ is also an
ultra-metric space with the ultra-metric $d_{n}(x,y)$ defined as 
\begin{equation*}
d_{n}(x,y)=\max \{d(x_{1},y_{1}),...,d(x_{n},y_{n})\}.
\end{equation*}

If the group of isometries of an ultra-metric space $(X,d)$ acts
transitively on $X,$ \noindent then $(X,d)$ is in fact a locally compact
Abelian group, which in particular is the case for $\mathbb{Q}_{p}^{n}.$

In literature one distinguishes the following two subclasses of ultra-metric
spaces:

\begin{itemize}
\item $(X,d)$ is discrete and infinite.

\item $(X,d)$ is perfect (that is, $X$ contains no isolated point).
\end{itemize}

Various constructions of Markov processes on non-compact perfect locally
compact Abelian groups have been developed by Evans~\cite{Evans1989}, Haran~%
\cite{Haran1990}, \cite{Haran1993}, Ismagilov~\cite{Ismagilov1991}, Kochubei 
\cite{Kochubei1991}, \cite{Kochubei2001}, Albeverio and Karwowski \cite{AK0}%
, \cite{AK1}, Albeverio and Zhao~\cite{AlbevZhao2000}, Del Muto and Fig\'{a}%
-Talamanca~\cite{Figa-Tal1}, \cite{Figa-Tal2}, Rodrigues-Vega and
Zuniga-Galindo \cite{ZuGa1}, \cite{ZuGa2}. They studied $X$-valued
infinitely divisible random variables and processes by using tools of
Fourier analysis; for general references, see Hewitt and Ross~\cite%
{HewittRoss}, Taibleson~\cite{TAI75} and Kochubei \cite{Kochubei2001}.
Indeed, Taibleson's spectral multipliers on $\mathbb{Q}_{p}^{n}$ are early
forerunners of the Laplacians that we are considering here.

Pearson and Bellissard~\cite{PearsonBellissard} and Kigami~\cite{Ki}, \cite%
{Kigami2} considered random walks on the Cantor set, resp. the Cantor set
minus one point. In \cite{Ki}, \cite{Kigami2}, a main focus is on the
interplay between random walks on trees and jump process on their
boundaries. In this context, we also mention Aldous and Evans~\cite%
{AldousEvans} and Chen, Fukushima and Ying \cite{ChFuYi}. We shall come back
to Kigami's work in the last three sections of this paper.

An entirely different approach was developed by Vladimirov, Volovich and
Zelenov \cite{Vladimirov}, \cite{Vladimirov94}. They were concerned with $p$%
-adic analysis (Bruhat distributions, Fourier transform etc.) related to the
concept of $p$-adic Quantum Mechanics, and introduced a class of
pseudo-differential operators on $\mathbb{Q}_{p}$ and on $\mathbb{Q}_{p}^{n}$%
. In particular, they considered the $p$-adic Laplacian defined on $\mathbb{Q%
}_{p}^{3}$ and studied the corresponding $p$-adic Schr\"{o}dinger equation.
Among other results, they explicitly computed (as series expansions) certain
heat kernels as well as the Green function of the $p$-adic Laplacian. In
connection with the theory of pseudo-differential operators on general
totally disconnected groups we mention here the pioneering work of
Saloff-Coste~\cite{Saloff-Coste1986}.

Discrete ultra-metric spaces $(X,d)$ were treated by Bendikov, Grigor'yan
and Pittet~\cite{BGP1}, the direct forerunner of the present work. Among the
examples of such spaces we mention the class of locally finite groups: a
countable group $G$ is locally finite if any of its finite subsets generates
a finite subgroup. Every locally finite group $G$ is the union of an
increasing sequence of finite subgroups $\{G_{n}\}$. An ultra-metric $d$ in $%
G$ can be defined as follows: $d(x,y)$ is the minimal value of $n$ such that 
$x$ and $y$ belong to a common coset of $G_{n}\,$.

Since locally finite groups are not finitely generated, the basic notions of
geometric group theory such as the word metric, volume growth, isoperimetric
inequalities, etc. (cf. e.g. \cite{CoulhonGP}, \cite{Gromov}, \cite%
{Saloff-Coste2001}, \cite{SALOFF-PITTET1999}, \cite{SALOFF-PITTET2000}, \cite%
{SALOFF-PITTET2002}, , \cite{VaropoulosSaloff-CosteCoulhon}, \cite{Woess2000}%
), do not apply in this setting. The notion of an ultra-metric can be used
instead of the word metric in this setting (see \cite{B27}, \cite{BGP1}, 
\cite{BBP2010}).

Selecting a set of generators for each subgroup $G_{n}\,$of a locally finite
group $G$, one defines thereby a random walk, that is, a Markov kernel on $%
G_{n}$. Taking a convex combination of the Markov kernels across all $G_{n}$%
, one obtains a Markov kernel on $G$ that determines a random walk on $G\,$.
Such random walks have been studied by Darling and Erd\"{o}s~ \cite%
{DarlingErdos}, Kesten and Spitzer~\cite{KestenSpitzer}, Flatto and Pitt~%
\cite{FlattoPitt}, Fereig and Molchanov~\cite{Molchanov1978}, Kasymdzhanova~%
\cite{Kasym1981}, Cartwright~\cite{Car1988}, Lawler~\cite{Lawler1995},
Brofferio and Woess~ \cite{BrofWoess}, see also Bendikov and Saloff-Coste 
\cite{BenSal-some}. In particular, \cite{Lawler1995} has a remarkable
general criterion of recurrence of such random walks. Further results on
Markov processes on ultra-metric spaces can be found in \cite{DMM96}, \cite%
{DMM09}, \cite{Evans93}, \cite{Evans01}, \cite{MRM07}, \cite{RTV86}.

Many of the results in the above-mentioned literature are subsumed by our
approach via ultra-metrics. We develop tools to analyse a class of very
natural Markov processes on ultra-metric spaces without assuming any group
structure. In particular, the nature of our argument allows us to bring into
consideration an arbitrary Radon measure $\mu $ on $X$ (instead of the Haar
measure in the case of groups), that is used as a speed measure for a Markov
process.

So, given an ultra-metric space $\left( X,d\right) $, fix a Radon measure $%
\mu $ on $X$ with full support and define the family $\left\{ \mathrm{Q}%
_{r}\right\} _{r>0}$ of averaging operators acting on non-negative or
bounded Borel functions $f:X\rightarrow \mathbb{R}$ by 
\begin{equation}
\mathrm{Q}_{r}f(x)=\frac{1}{\mu \left( B_{r}(x)\right) }\int_{B_{r}(x)}f\,d%
\mu .  \label{Averager}
\end{equation}%
Note that $0<\mu \left( B_{r}\left( x\right) \right) <\infty $ for all $x\in
X$ and $r>0$. The operator $\mathrm{Q}_{r}$ has the kernel 
\begin{equation}
K_{r}(x,y)=\frac{1}{\mu \left( B_{r}(x)\right) }\,\mathbf{1}_{B_{r}(x)}(y).
\label{K_B kernel}
\end{equation}%
It is symmetric in $x,y$ because $B_{r}(x)=B_{r}(y)\ $for\ any$\;y\in
B_{r}(x).$ Clearly, $\mathrm{Q}_{r}$ is a Markov operator on the space $%
\mathcal{B}_{b}\left( X\right) $ of bounded Borel functions on $X$, that is, 
$\mathrm{Q}_{r}f\geq 0$ if $f\geq 0$ and $\mathrm{Q}_{r}1=1$. Hence, $%
\mathrm{Q}_{r}$ extends to a bounded self-adjoint operator in $L^{2}\left(
X,\mu \right) .$

Let us choose a function $\sigma $ that satisfies the following assumptions:%
\begin{equation}
\left. 
\begin{array}{l}
\sigma :[0,\infty ]\rightarrow \left[ 0,1\right] \ \text{\emph{is\ a\
strictly\ monotone\ increasing\ left-continuous\ function}} \\ 
\text{\emph{such that} }\sigma \left( 0+\right) =0\ \text{\emph{and}}\
\sigma \left( \infty \right) =1.%
\end{array}%
\right.  \label{sigma}
\end{equation}%
Then the operator 
\begin{equation}
Pf=\int_{0}^{\infty }\mathrm{Q}_{r}f\,d\sigma (r)  \label{convexcomb}
\end{equation}%
is also a Markov operator in $\mathcal{B}_{b}\left( X\right) $ as well as a
bounded self-adjoint operator in $L^{2}\left( X,\mu \right) $.

The operator $P$ determines a discrete time Markov chain $\left\{ \mathcal{X}%
_{n}\right\} _{n\in \mathbb{N}}$ on $X$ with the following transition rule: $%
\mathcal{X}_{n+1}$ is $\mu $-uniformly distributed in $B_{r}(\mathcal{X}%
_{n}) $ where the radius $r$ is chosen at random according to the
probability distribution $\sigma $. For that reason we refer to $\sigma $ as
the \emph{distance distribution function}.

Note that $P$ is determined by the triple $\left( d,\mu ,\sigma \right) .$
We refer to $P$ as an \emph{isotropic Markov} operator associated with $%
\left( d,\mu ,\sigma \right) $. The isotropic Markov operator $P$ has some
unique features arising from the ultra-metric property. First of all, there
is the simple identity 
\begin{equation}  \label{QrQs}
\mathrm{Q}_{r} \, \mathrm{Q}_{s} = \mathrm{Q}_{s} \, \mathrm{Q}_{r} =\mathrm{%
Q}_{\max\{r,s\}}.
\end{equation}
Indeed, for any ball $B$ of radius $r$, \emph{any} point $x\in B$ is a
center of $B$. Since the value $\mathrm{Q}_{r}f\left( x\right) $ is the
average of $f$ in $B$, we see that $\mathrm{Q}_{r}f\left( x\right) $ does
not depend on $x\in B$; that is, $\mathrm{Q}_{r}f=\func{const}$ on $B$. Now,
if $s \le r$ then the application of $\mathrm{Q}_s$ to $\mathrm{Q}_rf$ does
not change this constant, whence we obtain $\mathrm{Q}_{r}\mathrm{Q}_{s}f=%
\mathrm{Q}_{r}f.$ On the other hand, if $s > r$ then any ball $B$ of radius $%
s$ is the disjoint union of finitely many balls $B_{j}$ of radius $r$. Then $%
\mathrm{Q}_{r}f=\func{const}_{j}$ on $B_{j}\,$, and by a one-line
computation, $\mathrm{Q}_{r}\mathrm{Q}_{s}f=\mathrm{Q}_{s}f.$

Since by (\ref{QrQs}) $\mathrm{Q}_{r}^{2}=\mathrm{Q}_{r}$, we obtain that $%
\mathrm{Q}_{r}$ is an \emph{orthoprojector}\footnote{%
Let us mention for comparison, that the analogous averaging operator in $%
\mathbb{R}^{n}$ is also bounded and self-adjoint, but it has a non-empty
negative part of the spectrum. In particular, it is not an orthoprojector.}
in $L^{2}.$ In particular, $\func{spec}\mathrm{Q}_{r}\subset \left[ 0,1%
\right] .$

It follows from (\ref{convexcomb}) that the spectral projectors in the
spectral decomposition of $P$ are the averaging operators $\mathrm{Q}_{r}$,
up to a change of variables (cf. (\ref{Ela}). The fact that the spectral
projectors are themselves Markov operators brings up a new insight, new
technical possibilities, and a new type of results, that have no analogue in
other commonly used settings.

In particular, the Markov operator $P$ is non-negative definite, which
allows us to define the powers $P^{t}$ for all $t\geq 0$. Then $\left\{
P^{t}\right\} _{t\geq 0}$ is a symmetric strongly continuous Markov
semigroup. It follows from (\ref{convexcomb}) that $P^{t}$ admits for $t>0$
the following representation: 
\begin{equation}  \label{convexcomb-Pt}
P^{t}f(x)=\int_{0}^{\infty }\mathrm{Q}_{r}f(x)\,d\sigma ^{t}(r)\,.
\end{equation}
In a more elementary way, one can also \emph{define} $P^{t}$ by (\ref%
{convexcomb-Pt}) and use formula (\ref{QrQs}) to derive that $P^{s}P^{t} =
P^{s+t}$.

The semigroup $\left\{ P^{t}\right\} _{t\geq 0}$ determines a continuous
time Markov process $\left\{ \mathcal{X}_{t}\right\} _{t\geq 0}.$ Since the $%
n$-step transition operator of the discrete time Markov chain $\left\{ 
\mathcal{X}_{n}\right\} _{n\in \mathbb{N}}$ is $P^{n}$, we see that the
discrete time Markov chain coincides with the restriction of the continuous
time Markov process $\{\mathcal{X}_{t}\}$ to integer values of $t$. This
allows us to concentrate on the study of the continuous time process $%
\left\{ \mathcal{X}_{t}\right\} _{t\geq 0}$ only.

We refer to the Markov semigroup $\left\{ P^{t}\right\} _{t\geq 0}$ defined
by (\ref{Averager})-(\ref{convexcomb-Pt}) as an \emph{isotropic semigroup},
and to the process $\left\{ \mathcal{X}_{t}\right\} _{t\geq 0}$ as an \emph{%
isotropic jump process, }associated with the triple $\left( d,\mu ,\sigma
\right) $.

Let us briefly describe the content of the present paper that is devoted to
the study of isotropic semigroups.

In Section \ref{heat} we construct the isotropic semigroup as mentioned
above and provide explicit formulas for its heat kernel $p\left(
t,x,y\right) $ (=the transition density of the process $\left\{ \mathcal{X}%
_{t}\right\} $). As indicated above, our approach is based upon the
observation that the building blocks of the operator $P$, namely, the
averaging operators $\mathrm{Q}_{r}$ of (\ref{Averager}), are orthogonal
projectors in $L^{2}(X,\mu ),$ which enables us to engage at an early stage
the methods of spectral theory and functional calculus.

We establish some basic properties of the heat kernel, for example, its
continuity away from the diagonal, and prove upper and lower bounds in terms
of $t$ and $d\left( x,y\right) $.

For example, in $\mathbb{Q}_{p}$ with the $p$-adic ultra-metric $\left\Vert
x-y\right\Vert _{p}$ and the Haar measure $\mu $, the most natural choice of
the distance distribution function is 
\begin{equation}
\sigma \left( r\right) =\exp \left( -\left( \frac{p}{r}\right) ^{\alpha
}\right)  \label{sa}
\end{equation}%
for $\alpha >0$, when the associated heat kernel admit the estimate%
\begin{equation}
p_{t}\left( x,y\right) \simeq \frac{\,t}{(t^{1/\alpha }+\left\Vert
x-y\right\Vert _{p})^{1+\alpha }}  \label{ptp}
\end{equation}%
for all $t>0$ and $x,y\in \mathbb{Q}_{p}$. Note that the estimate (\ref{ptp}%
) is similar to the heat kernel bound for a symmetric stable process in $%
\mathbb{R}$ of index $\alpha $.

We also obtain explicit expression for the Green function of the isotropic
semigroup and provide a transience criterion in terms of the volume growth.
Unlike the previously known transience criteria (cf. \cite{Lawler1995}),
ours does not assume any group structure.

In Section \ref{generator} we are concerned with the spectral properties of
the \emph{isotropic Laplacian} $\mathcal{L}$ that is the (positive definite)
generator of the isotropic semigroup. We provide a full description of the
spectrum, in particular, show that the spectrum is pure point, and list
explicitly all the eigenfunctions by means of ultra-metric balls. Also, we
show that the spectra of the extensions of $\mathcal{L}$ in the spaces $%
L^{p},$ $1\leq p<\infty $, do not depend on $p$.

A striking property of the isotropic Laplacian $\mathcal{L}$ is that, for
any increasing bijection $\psi :[0,\infty )\rightarrow \lbrack 0,\infty )$,
the operator $\psi \left( \mathcal{L}\right) $ is also an isotropic
Laplacian (for another distance distribution function). In particular, $%
\mathcal{L}^{\alpha }$ is an isotropic Laplacian for any $\alpha >0$. Recall
for comparison that, for a general symmetric Markov generator $\mathcal{L}$,
the operator $\mathcal{L}^{\alpha }$ generates a Markov semigroup only for $%
0<\alpha \leq 1$.

In Section \ref{moments} we obtain two sided estimates of moments of the
Markov process $\left\{ \mathcal{X}_{t}\right\} $, which is a pure jump
process.

In the case when $X$ is a locally compact group, our results apply with an
arbitrary Radon measure $\mu $ instead of the Haar measure. Some of the
aforementioned questions are particularly sensitive to the choice of the
measure $\mu $, for example, the heat kernel and Green function estimates.
On the other hand, the spectrum of the Laplacian and escape rate bounds do
not depend on $\mu .$ These quantities depend strongly on the choice of the
ultra-metric $d$, whereas the eigenfunctions depend both on $d$ and $\mu $.

In Section \ref{SecQp} we compare our isotropic Laplacian with other
previously known \textquotedblleft differential\textquotedblright\ operators
in $\mathbb{Q}_{p}$ and $\mathbb{Q}_{p}^{n}$. The notion of fractional
derivative $\mathfrak{D}^{\alpha }$ on functions on $\mathbb{Q}_{p}$ was
introduced by Vladimirov~\cite{Vladimirov} by means of Fourier transform in $%
\mathbb{Q}_{p}$, which coincides with the operator of Taibleson \cite{TAI75}%
, introduced in the quite different context of Riesz multipliers on $\mathbb{%
Q}_{p}^{n}$. We show that $\mathfrak{D}^{\alpha }$ coincides with our
isotropic Laplacian $\mathcal{L}_{\alpha }$ associated with the distance
distribution function (\ref{sa}). In particular, this implies that the heat
kernel of $\mathfrak{D}^{\alpha }$ satisfies the estimate (\ref{ptp}). Note
that previously only an upper bound for the heat kernel of $\mathfrak{D}%
^{\alpha }$ was known (cf. Kochubei \cite[Ch.4.1, Lemma 4.1]{Kochubei2001}).
We also give a simple proof for a previously known explicit formula for the
fundamental solution of $\mathfrak{D}^{\alpha }$.

Using functional calculus of the operator $\mathfrak{D}^{1}$, we give a full
description of the class of all rotation invariant Markov generators on $%
\mathbb{Q}_{p}$, which includes but is not restricted to the isotropic
Laplacians. As a consequence, we obtain that the class of all rotation
invariant Markov processes in $\mathbb{Q}_{p}$ coincides with the class of
Markov processes constructed by Albeverio and Karwowski \cite{AK1} by use of
much more involved technical tools.

Next we consider \textquotedblleft differential\textquotedblright\ operators
on $\mathbb{Q}_{p}^{n}$. The $p$-adic Laplace operator of Vladimirov on $%
\mathbb{Q}_{p}^{n}$ is defined as a direct sum of the operators $\mathfrak{D}%
^{\alpha }$ acting separately on each coordinate. Although this operator is
not an isotropic Laplacian, it can be studied within our setting, which
gives simple direct proofs of many results of \cite{Vladimirov94}, without
using Fourier Analysis and the theory of Bruhat distributions.

Another multidimensional generalization of $\mathfrak{D}^{\alpha }$ is the
Taibleson operator $\mathfrak{T}^{\alpha }$ in $\mathbb{Q}_{p}^{n}$ that is
defined by means of Fourier transform in $\mathbb{Q}_{p}^{n}.$ We show that
the operator $\mathfrak{T}^{\alpha }$ is an isotropic Laplacian, which
allows to obtain very detailed analytic results.

In Section \ref{trees} we use the fact that every locally compact
ultra-metric space arises as the boundary of a locally finite tree. Using
that we relate random walks\footnote{%
Discrete time random walks of nearest neighbour type on a tree are very well
understood -- see the book by Woess~\cite[Ch. 9]{W-Markov}} on the tree with
isotropic jump processes on its boundary. In recent work, Kigami~\cite{Ki}
starts with a transient nearest neighbour random walk on a tree and
constructs a naturally associated jump process on the boundary of the tree:
given the Dirichlet form of the random walk on the vertex set of the tree,
the boundary process is induced by the Dirichlet form that reproduces the
power (\textquotedblleft energy\textquotedblright ) of a harmonic function
on the tree via its boundary values. This is analogous to the well-known
Douglas integral \cite{Doug} on the unit disk. Using this approach, \cite{Ki}
undertakes a detailed analysis of the process on the boundary.

Restricting attention at first to the compact case, in Section \ref{duality}
we answer the obvious question how the approaches of Kigami and of the
present paper are related. The relation is basically one-to-one: 
%
every boundary process induced by a random walk is an isotropic process of
our setting. Conversely, we show that up to a unique linear time change,
every isotropic jump process on the boundary of a tree arises from a
uniquely determined random walk as the process of \cite{Ki}. In addition, we
explain how the boundary processes on a tree transforms into an isotropic
jump process on the non-compact ultra-metric space given by a punctured
boundary of the tree. This should be compared with the very recent work \cite%
{Kigami2}.

Finally, in Section \ref{SecFrac} we elaborate the specific examples of the $%
p$-adic fractional derivative on the (compact) group of $p$-adic integers
and the corresponding random walk on the associated rooted tree, as well as
the random walk corresponding to the fractional derivative on the whole of $%
\mathbb{Q}_{p}$.

\begin{acknowledgement}
This work was begun and finished at Bielefeld University under support of
SFB 701 of the German Research Council. The authors thank S.~Albeverio,
J.~Bellissard, P.~Diaconis, W.~Herfort, A.N.~Kochubei, S.A.~Molchanov,
L.~Saloff-Coste, I.V.~Volovich and E.I.~Zelenov for fruitful discussions and
valuable comments.
\end{acknowledgement}

\section{Heat semigroup and heat kernel}

\label{heat}\setcounter{equation}{0}Throughout this paper, $\left(
X,d\right) $ is an ultra-metric space which is separable, and such that all $%
d$-balls $B_{r}\left( x\right) $ are compact.

\subsection{Averaging operator}

Recall that for any $r>0$, 
\begin{equation*}
\mathrm{Q}_{r}f(x)=\frac{1}{\mu \left( B_{r}(x)\right) }\int_{B_{r}(x)}f\,d%
\mu
\end{equation*}%
is an orthoprojector in $L^{2}\equiv L^{2}\left( X,\mu \right) $ (cf. (\ref%
{Averager})), and the image of $\mathrm{Q}_{r}$ is the subspace $\mathcal{V}%
_{r}$ of $L^{2}$ that consists of all functions taking constant values on
each ball radius $r.$

Clearly, the family $\left\{ \mathcal{V}_{r}\right\} _{r>0}$ is monotone
decreasing with respect to set inclusion. It follows that there exists the
limit%
\begin{equation*}
\mathrm{Q}_{\infty }:=s\text{-}\lim_{r\rightarrow \infty }\mathrm{Q}_{r}
\end{equation*}%
in the strong operator topology, which is an orthoprojector onto $\mathcal{V}%
_{\infty }=\tbigcap_{r>0}\mathcal{V}_{r}$. It follows that $\mathcal{V}%
_{\infty }$ consists of constant functions. If $\mu \left( X\right) =\infty $
then $\mathcal{V}_{\infty }=\left\{ 0\right\} $ and $\mathrm{Q}_{\infty }=0$%
, while in the case $\mu \left( X\right) <\infty $ we have $\dim \mathcal{V}%
_{\infty }=1$ and 
\begin{equation}
\mathrm{Q}_{\infty }f=\frac{1}{\mu \left( X\right) }\int_{X}fd\mu .
\label{Qi}
\end{equation}%
Set also $\mathrm{Q}_{0}:=\func{id}.$

\begin{lemma}
The family $\left\{ \mathrm{Q}_{r}\right\} _{r\in \lbrack 0,\infty )}$ of
orthoprojectors is strongly right continuous in $r$.
\end{lemma}

\begin{proof}
Let us first show that $r\mapsto \mathrm{Q}_{r}$ is strongly continuous at $%
r=0$, that is,%
\begin{equation}
s\text{-}\lim_{s\rightarrow 0+}\mathrm{Q}_{s}=\func{id}.  \label{sr0}
\end{equation}%
Let $f$ be a continuous function on $X$ with compact support. Then, for any $%
x\in X$,%
\begin{equation*}
\mathrm{Q}_{s}f\left( x\right) \rightarrow f\left( x\right) \ \text{as }%
s\rightarrow 0.
\end{equation*}%
Since the family $\left\{ \mathrm{Q}_{s}f\right\} _{s\in \left( 0,1\right) }$
is uniformly bounded by $\sup \left\vert f\right\vert $ and is uniformly
compactly supported, it follows by the dominated convergence theorem that 
\begin{equation}
\left\Vert \mathrm{Q}_{s}f-f\right\Vert _{L^{2}}\rightarrow 0\ \text{as\ }%
s\rightarrow 0.  \label{Qe}
\end{equation}%
Since the space of continuous functions with compact support is dense in $%
L^{2}$, by a standard approximation argument (\ref{Qe}) extends to all $f\in
L^{2}$, whence (\ref{sr0}) follows.

Next, let us prove that $r\mapsto \mathrm{Q}_{r}$ is strongly right
continuous at any $r>0$, that is,%
\begin{equation}
s\text{-}\lim_{s\rightarrow r+}\mathrm{Q}_{s}=\mathrm{Q}_{r}.  \label{srr}
\end{equation}%
It suffices to show that, for any continuous function $f$ with compact
support,%
\begin{equation}
\left\Vert \mathrm{Q}_{s}f-\mathrm{Q}_{r}f\right\Vert _{L^{2}}\rightarrow 0\
\ \text{as }s\rightarrow r\!+.  \label{Qg}
\end{equation}%
Indeed, for any $x\in X$, the function $r\mapsto \mathrm{Q}_{r}f\left(
x\right) $ is right continuous by (\ref{Averager}) as the balls are closed,
whence (\ref{Qg}) follows by the dominated convergence theorem.
\end{proof}

For any $\lambda \in \mathbb{R}$ set%
\begin{equation}
E_{\lambda }=\left\{ 
\begin{array}{l}
\mathrm{Q}_{1/\lambda },\ \lambda >0, \\ 
0,\ \ \ \ \ \lambda \leq 0.%
\end{array}%
\right.  \label{Ela}
\end{equation}%
Note that $E_{0+}=\mathrm{Q}_{\infty }$. It follows from the above
properties of $\mathrm{Q}_{r}$ that the family $\left\{ E_{\lambda }\right\} 
$ of orthoprojectors in $L^{2}$ is a left-continuous spectral resolution.
Consequently, for any Borel function $\varphi :[0,\infty )\rightarrow 
\mathbb{R}$, the integral%
\begin{equation*}
\int_{\lbrack 0,\infty )}\varphi \left( \lambda \right) dE_{\lambda }
\end{equation*}%
determines a self-adjoint non-negative definite operator, which is bounded
if and only if $\varphi $ is bounded.

\subsection{Basic properties of heat semigroup}

Consider now the operator $P$ defined by (\ref{convexcomb}) with a function $%
\sigma $ as in (\ref{sigma}). Observe that the integral in (\ref{convexcomb}%
) converges in the strong operator topology since, for any $f\in L^{2}$,%
\begin{equation*}
\int_{0}^{\infty }\left\Vert Q_{r}f\right\Vert _{L^{2}}d\sigma \left(
r\right) <\infty .
\end{equation*}%
On the other hand, for any $f\in \mathcal{B}_{b}\left( X\right) $, the
integral (\ref{convexcomb}) converges pointwise. Moreover, in this case the
function $Pf$ is continuous, because the function $x\mapsto
\int_{\varepsilon }^{\infty }\mathrm{Q}_{r}f\left( x\right) d\sigma \left(
r\right) $ is for any $\varepsilon >0$ locally constant and, hence,
continuous and it converges uniformly to $Pf\left( x\right) $ as $%
\varepsilon \rightarrow 0$.

As it was already observed, $P$ is a self-adjoint operator in \ $L^{2}$ and $%
\func{spec}P\subset \left[ 0,1\right] .$ In particular, for any $t>0$, the
power $P^{t}$ is well defined. Set also $P^{0}:=\func{id}$. In the next
statement we collect basic properties of $P^{t}$.

\begin{theorem}
\label{heat kernel}

\begin{itemize}
\item[$\left( a\right) $] The family $\{P^{t}\}_{t\geq 0}$ is a strongly
continuous symmetric Markov semigroup on $L^{2}(X,\mu )$.

\item[$\left( b\right) $] For any $t>0$, the operator $P^{t}$ has the
representation \emph{(\ref{convexcomb-Pt})}, that is,%
\begin{equation*}
P^{t}f=\int_{[0,\infty )}\mathrm{Q}_{r}f\,d\sigma ^{t}(r)\,.
\end{equation*}

\item[$\left( c\right) $] For any $t>0$, the operator $P^{t}$ admits an
integral kernel $p(t,x,y)$, that is, for all $f\in \mathcal{B}_{b}$ and $%
f\in L^{2}$, 
\begin{equation}
P^{t}f(x)=\int_{X}p(t,x,y)f(y)d\mu (y),  \label{Pt}
\end{equation}%
where $p\left( t,x,y\right) $ is given by%
\begin{equation}
p(t,x,y)=\int_{[d(x,y),\infty )}\frac{d\sigma ^{t}(r)}{\mu \left(
B_{r}(x)\right) }.  \label{p=id}
\end{equation}
\end{itemize}
\end{theorem}

The function $p\left( t,x,y\right) $ is called the \emph{heat kernel} of the
semigroup $\left\{ P^{t}\right\} $. It is clear from (\ref{p=id}) that $%
p\left( t,x,y\right) <\infty $ for all $t>0$ and $x\neq y$, whereas under
certain conditions $p\left( t,x,x\right) $ can be equal to $\infty $.

For $f\in \mathcal{B}_{b}$ the identity (\ref{Pt}) holds pointwise, that is,
for all $x\in X$, whereas for $f\in L^{2}$ (\ref{Pt}) is an identity of two $%
L^{2}$-functions, that it, it holds for $\mu $-almost all $x$.

\begin{proof}
It follows from (\ref{convexcomb}) by integrations by parts that, for any $%
f\in L^{2}$,%
\begin{equation}
Pf=\int_{[0,\infty )}\mathrm{Q}_{r}f~d\sigma (r)=Q_{\infty
}f-\int_{(0,\infty )}\sigma \left( r\right) d\mathrm{Q}_{r}f.  \label{ccc}
\end{equation}%
Changing $\lambda =1/r$ and using (\ref{Ela}), we obtain%
\begin{equation*}
Pf=\left( E_{0+}\right) f+\int_{(0,\infty )}\sigma \left( 1/\lambda \right)
dE_{\lambda }f=\int_{[0,\infty )}\sigma \left( 1/\lambda \right) dE_{\lambda
}f,
\end{equation*}%
using the convention $\sigma \left( \infty \right) =1$. Hence, we obtain the
spectral resolution of $P$ in the following form:%
\begin{equation}
P=\int_{[0,\infty )}\sigma \left( 1/\lambda \right) dE_{\lambda }.
\label{PEla}
\end{equation}%
It follows that 
\begin{equation}
P^{t}=\int_{[0,\infty )}\sigma ^{t}\left( 1/\lambda \right) dE_{\lambda }.
\label{PtEla}
\end{equation}

$\left( a\right) $ The semigroup identity $P^{t}P^{s}=P^{t+s}$ is a
straightforward consequence of(\ref{QrQs}), as observed in the Introduction.
It remains to show that%
\begin{equation*}
s\text{-}\lim_{t\rightarrow 0+}P^{t}=\func{id},
\end{equation*}%
which easily follows from (\ref{PtEla}) because $\sigma \left( 1/\lambda
\right) >0$ for $\lambda \in \lbrack 0,\infty )$.

$\left( b\right) $ Reversing the argument in the derivation of (\ref{PtEla})
from (\ref{ccc}), we obtain that (\ref{PtEla}) implies%
\begin{equation*}
P^{t}f=\int_{[0,\infty )}\mathrm{Q}_{r}f~d\sigma ^{t}(r).
\end{equation*}

$\left( c\right) $ It follows from $\left( b\right) $, (\ref{Averager}) and
Fubini that, for any $f\in \mathcal{B}_{b}$, 
\begin{eqnarray*}
P^{t}f\left( x\right) &=&\int_{[0,\infty )}\left( \frac{1}{\mu \left(
B_{r}(x)\right) }\int_{X}\mathbf{1}_{B_{r}(x)}\left( y\right) f(y)d\mu
(y)\right) d\sigma ^{t}(r) \\
&=&\int_{X}\left( \int_{[d\left( x,y\right) ,\infty )}\frac{1}{\mu \left(
B_{r}(x)\right) }d\sigma ^{t}(r)\right) f(y)d\mu (y)\, \\
&=&\int_{X}p\left( t,x,y\right) f\left( y\right) d\mu \left( y\right) .
\end{eqnarray*}%
For $f\in L^{2}$ it follows by approximation argument.
\end{proof}

\begin{remark}
\RM In the proof of Theorem \ref{heat kernel} we have not used at full
strength the fact that $\sigma $ is \emph{strictly} monotone increasing (cf.
(\ref{sigma})). For that theorem, it suffices to assume that $\sigma $ is
monotone increasing and $\sigma \left( r\right) >0$ for $r>0$.
\end{remark}

\begin{remark}
\RM If one takes (\ref{convexcomb-Pt}) as definition of the operator $P^{t}$%
, then one can prove the semigroup identity $P^{t}P^{s}=P^{t+s}$ by means of
(\ref{QrQs}). Indeed, for any given $s,t>0$ and $f\in L^{2}$, we have 
\begin{eqnarray*}
P^{s}P^{t}f(x) &=&\int_{0}^{\infty }d\sigma ^{s}(r)\int_{0}^{\infty }d\sigma
^{t}(r^{\prime })\,\mathrm{Q}_{r}\mathrm{Q}_{r^{\prime }}f(x)= \\
&=&\int_{0}^{\infty }d\sigma ^{s}(r)\int_{0}^{\infty }d\sigma ^{t}(r^{\prime
})\,\mathrm{Q}_{\max \{r,r^{\prime }\}}f(x).
\end{eqnarray*}%
Let $\xi _{1}$ and $\xi _{2}$ be two independent random variables with
distributions $\sigma ^{s}$ and $\sigma ^{t}$, respectively. Then the
distribution of the random variable $\xi =\max \{\xi _{1},\xi _{2}\}$ is $%
\sigma ^{t+s}.$ It follows that 
\begin{equation*}
P^{s}P^{t}f(x)=\mathbb{E}\left( \mathrm{Q}_{\max \{\xi _{1},\xi
_{2}\}}f(x)\right) =\int_{0}^{\infty }\mathrm{Q}_{r}f(x)\,d\sigma
^{t+s}(r)=P^{t+s}f(x).
\end{equation*}
\end{remark}

\begin{corollary}
\label{Corpmin}For all $x,y\in X$ and $t>0,$ we have $p\left( t,x,y\right)
>0,$%
\begin{equation*}
p\left( t,x,y\right) =p\left( t,y,x\right) ,
\end{equation*}%
and%
\begin{equation}
p(t,x,y)\leq \min \{p(t,x,x),p(t,y,y)\}.  \label{ptmin}
\end{equation}
\end{corollary}

\begin{proof}
The strict positivity of $p\left( t,x,y\right) $ follows from (\ref{p=id})
and the strict monotonicity of $\sigma $.

In the integral in (\ref{p=id}) we have $r\geq d\left( x,y\right) $ whence
it follows that $B_{r}\left( x\right) =B_{r}\left( y\right) $ and $p\left(
t,x,y\right) =p\left( t,y,x\right) $. Alternatively, the symmetry of the
heat kernel follows also from the fact that $P^{t}$ is self-adjoint.

By (\ref{p=id}) we have%
\begin{equation*}
p(t,x,y)=\int_{[d(x,y),\infty )}\frac{d\sigma ^{t}(r)}{\mu \left(
B_{r}(x)\right) }\leq \int_{\lbrack 0,\infty )}\frac{d\sigma ^{t}(r)}{\mu
\left( B_{r}(x)\right) }=p(t,x,x),
\end{equation*}%
whence (\ref{ptmin}) follows.
\end{proof}

Note that in general, heat kernels only satisfy the estimate%
\begin{equation*}
p\left( t,x,y\right) \leq \sqrt{p\left( t,x,x\right) p\left( t,y,y\right) }.
\end{equation*}%
The estimate (\ref{ptmin}) is obviously stronger, which reflects a special
feature of ultra-metricity.

\begin{corollary}
\label{RemF(p)}For any $t>0,$ the function%
\begin{equation}
x,y\mapsto \left\{ 
\begin{array}{ll}
\frac{1}{p\left( t,x,y\right) }, & x\neq y, \\ 
0, & x=y,%
\end{array}%
\right.   \label{1/p}
\end{equation}%
is an ultra-metric. 
\end{corollary}

\begin{proof}
Set 
\begin{equation*}
F\left( x,r\right) =\left( \int_{[r,+\infty )}\frac{d\sigma ^{t}(s)}{\mu
\left( B_{s}(x)\right) }\right) ^{-1}\ \ \text{for }r>0,
\end{equation*}%
$F\left( x,0\right) =0$, and observe the following two properties of $F$:

\begin{itemize}
\item[$\left( a\right) $] $r\mapsto F\left( x,r\right) $ is monotone
increasing in $r$;

\item[$\left( b\right) $] $F\left( x,r\right) =F\left( y,r\right) $ wherever 
$r\geq d\left( x,y\right) $ as in this case $B_{s}\left( x\right)
=B_{s}\left( y\right) $ for all $s\geq r$.
\end{itemize}

For any function $F$ with these properties, $\rho \left( x,y\right)
:=F\left( x,d\left( x,y\right) \right) $ is an ultra-metric, as the symmetry
follows from $\left( b\right) $, while the ultra-metric inequality (\ref{um}%
) follows from $\left( a\right) $ and $\left( b\right) $: if $d\left(
x,y\right) \leq d\left( x,z\right) $ then%
\begin{equation*}
\rho \left( x,y\right) =F\left( x,d\left( x,y\right) \right) \leq F\left(
x,d\left( x,z\right) \right) =\rho \left( x,z\right) ,
\end{equation*}%
and if $d\left( x,y\right) \leq d\left( y,z\right) $ then%
\begin{equation*}
\rho \left( x,y\right) =F\left( y,d\left( x,y\right) \right) \leq F\left(
y,d\left( y,z\right) \right) =\rho \left( y,z\right) .
\end{equation*}
\end{proof}

\subsection{Spectral distribution function}

\label{spectral}For the Markov semigroup $P$ associated with the triple $%
(d,\mu ,\sigma )$, define the \emph{intrinsic ultra-metric} $d_{\ast }$ by 
\begin{equation}
\frac{1}{d_{\ast }(x,y)}=\log \frac{1}{\sigma (d\left( x,y\right) )}.
\label{d*}
\end{equation}%
Since $d_{\ast }$ is expressed as a strictly monotone increasing function of 
$d$, which vanishes at $0$, it follows that $d_{\ast }$ is an ultra-metric
on $X$. Denote by $B_{r}^{\ast }\left( x\right) $ the metric balls of $%
d_{\ast }\,$.

\begin{lemma}
\label{star-balls}For any $r\geq 0$ set 
\begin{equation*}
s=\frac{1}{\log \frac{1}{\sigma \left( r\right) }}.
\end{equation*}%
Then the following identity holds for all $x\in X$: 
\begin{equation*}
B_{s}^{\ast }(x)=B_{r}(x).
\end{equation*}%
Consequently, the metrics $d$ and $d_{\ast }$ determine the same set of
balls and the same topology.
\end{lemma}

\begin{proof}
We have 
\begin{eqnarray*}
B_{s}^{\ast }\left( x\right) &=&\{y\in X:d_{\ast }\left( x,y\right) \leq s\}
\\
&=&\{y\in X:\sigma \left( d\left( x,y\right) \right) \leq \sigma \left(
r\right) \} \\
&=&\{y\in X:d\left( x,y\right) \leq r\} \\
&=&B_{r}\left( x\right) ,
\end{eqnarray*}%
where we have used that $\sigma $ is strictly monotone increasing.
\end{proof}

\begin{definition}
\label{def-spectral-d}\RM For any $x\in X$ we define the \emph{spectral
distribution function} $N\left( x,\cdot \right) :[0\,,\,\infty )\rightarrow
\lbrack 0\,,\,\infty )$ as 
\begin{equation}
N(x,\tau )=\frac{1}{\mu \left( B_{1/\tau }^{\ast }(x)\right) }.  \label{Ndef}
\end{equation}
\end{definition}

See Figures \ref{pic1}, \ref{pic2} and \ref{pic3} for qualitative graphs of $%
N\left( x,\cdot \right) .$

\FRAME{ftbphFU}{3.8744in}{1.8758in}{0pt}{\Qcb{The graph of the function $%
\protect\tau \mapsto N\left( x,\protect\tau \right) $ in the case when $%
\protect\mu \left( X\right) <\infty $}}{\Qlb{pic1}}{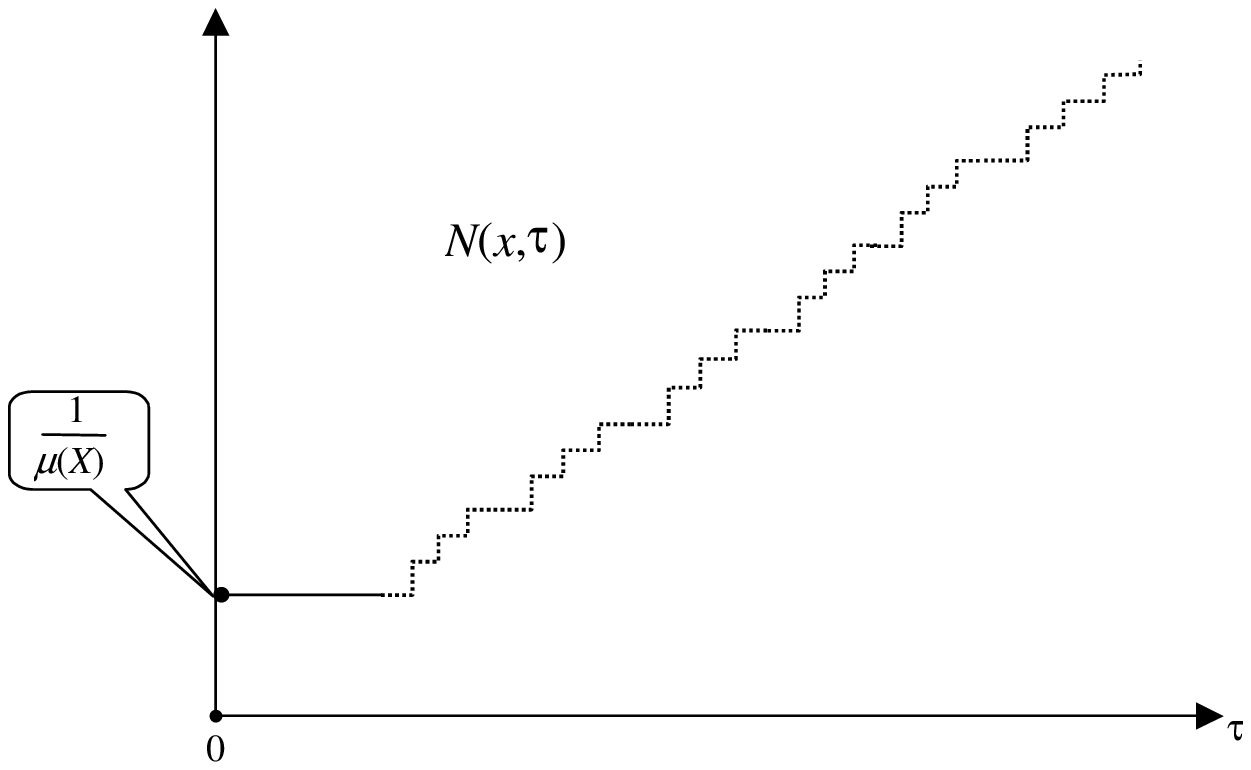}{\special%
{language "Scientific Word";type "GRAPHIC";maintain-aspect-ratio
TRUE;display "USEDEF";valid_file "F";width 3.8744in;height 1.8758in;depth
0pt;original-width 6.1037in;original-height 2.9332in;cropleft "0";croptop
"1";cropright "1";cropbottom "0";filename 'pic1.eps';file-properties
"XNPEU";}}

\FRAME{ftbphFU}{3.8398in}{2.1793in}{0pt}{\Qcb{The graph of the function $%
\protect\tau \mapsto N\left( x,\protect\tau \right) $ in the case, when $%
\protect\mu \left( x\right) >0$}}{\Qlb{pic2}}{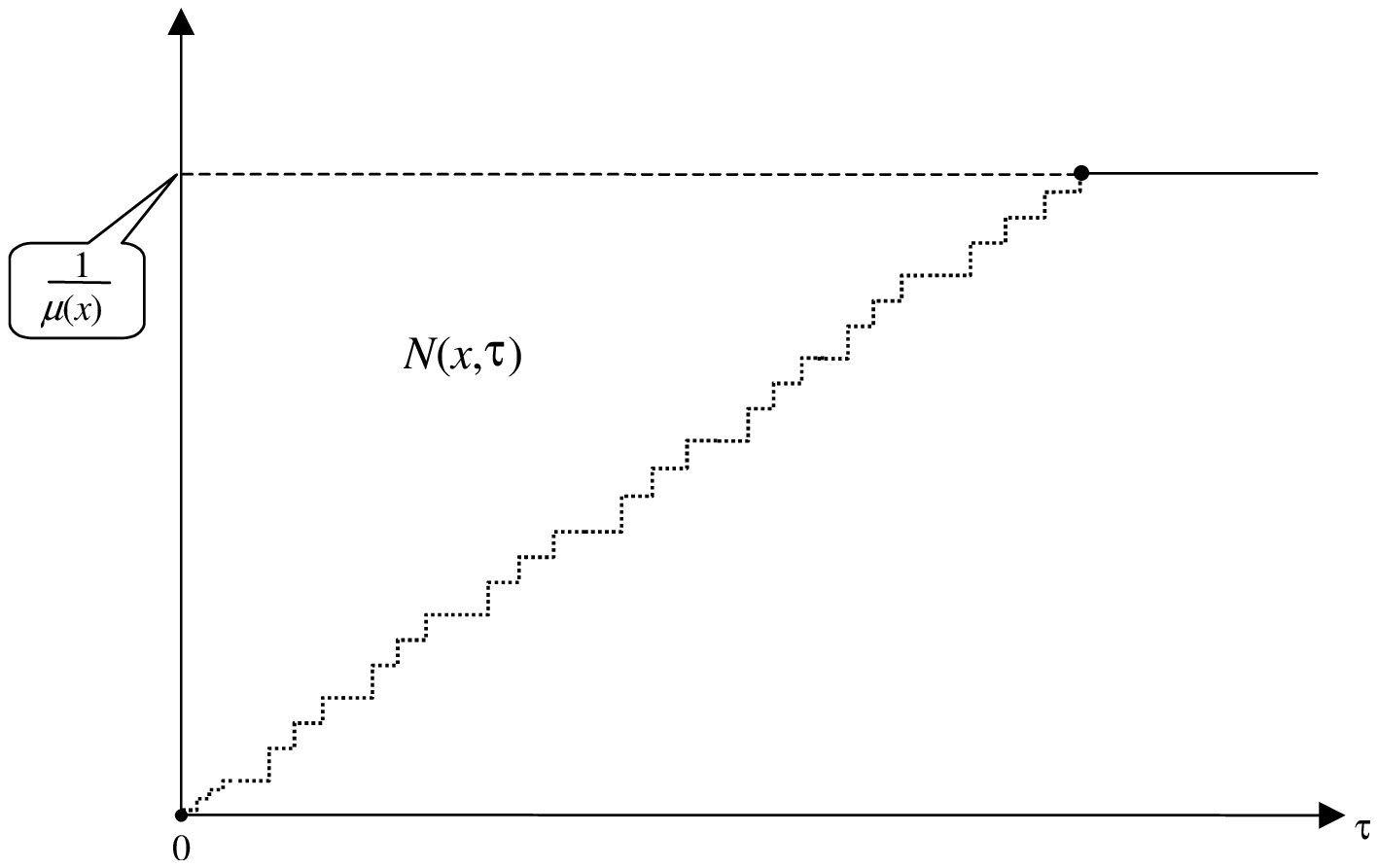}{\special{language
"Scientific Word";type "GRAPHIC";maintain-aspect-ratio TRUE;display
"USEDEF";valid_file "F";width 3.8398in;height 2.1793in;depth
0pt;original-width 6.0506in;original-height 3.4138in;cropleft "0";croptop
"1";cropright "1";cropbottom "0";filename 'pic2.eps';file-properties
"XNPEU";}}

\FRAME{ftbphFU}{3.8398in}{1.8939in}{0pt}{\Qcb{The graph of the function $%
\protect\tau \mapsto N\left( x,\protect\tau \right) $ in the case when $%
\protect\mu \left( x\right) =0$ and $\protect\mu \left( X\right) =\infty $}}{%
\Qlb{pic3}}{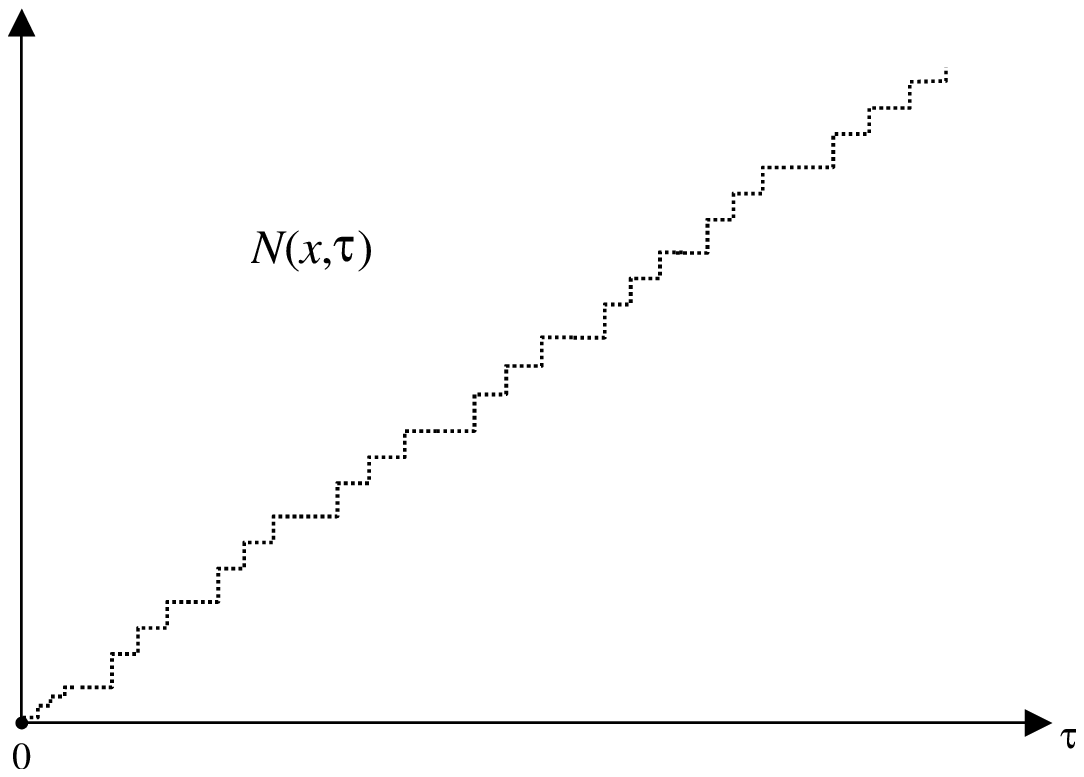}{\special{language "Scientific Word";type
"GRAPHIC";maintain-aspect-ratio TRUE;display "USEDEF";valid_file "F";width
3.8398in;height 1.8939in;depth 0pt;original-width 6.0506in;original-height
3.0826in;cropleft "0";croptop "1";cropright "1";cropbottom "0";filename
'pic3.eps';file-properties "XNPEU";}}

Let us define $\sigma _{\ast }\left( r\right) $ as the distribution function
of \textquotedblleft inverse exponential distribution\textquotedblright ,
that is, set%
\begin{equation}
\sigma _{\ast }\left( r\right) =\exp \left( -\frac{1}{r}\right) ,\ \ r>0.
\label{invexp}
\end{equation}%
As a distance distribution function, $\sigma _{\ast }$ will play an
important role in what follows.

\begin{definition}
\RM\label{def:standard}An isotropic Markov operator $P$ associated with a
triple $\left( d,\mu ,\sigma _{\ast }\right) $ will be referred to as a 
\emph{standard} Markov operator, associated with $\left( d,\mu \right) $.
\end{definition}

\begin{theorem}
\label{p-laplace}Let $d_{\ast }$ and $\sigma _{\ast }$ be defined by \emph{(%
\ref{d*})} and \emph{(\ref{invexp})}.

\begin{itemize}
\item[$\left( a\right) $] The triples $\left( d,\mu ,\sigma \right) $ and $%
\left( d_{\ast },\mu ,\sigma _{\ast }\right) $ induce the same isotropic
Markov operators.

\item[$\left( b\right) $] The heat kernel $p\left( t,x,y\right) $ associated
with the triple $\left( d,\mu ,\sigma \right) $ satisfies for all $x,y\in X$
and $t>0$ the following identities:%
\begin{equation}
p(t,x,y)=\int_{0}^{t/d_{\ast }\left( x,y\right) }N\left( x,\frac{s}{t}%
\right) e^{-s}ds  \label{p=N}
\end{equation}%
and 
\begin{equation}
p(t,x,y)=t\int_{0}^{1/d_{\ast }(x,y)}N(x,\tau )\exp (-\tau t)\,d\tau .
\label{p=Nt}
\end{equation}%
Consequently, $p\left( t,x,y\right) $ is a finite continuous function of $%
t,x,y$ for all $t>0$ and $x\neq y$.
\end{itemize}
\end{theorem}

As it follows from $\left( a\right) $, any isotropic Markov operator is at
the same time the standard Markov operator, associated with $\left( d_{\ast
},\mu \right) $.

\begin{proof}
$\left( a\right) $ It suffices to show that 
\begin{equation}
p\left( t,x,y\right) =\int_{[d_{\ast }\left( x,y\right) ,\infty )}\frac{%
d\sigma _{\ast }^{t}\left( u\right) }{\mu \left( B_{u}^{\ast }\left(
x\right) \right) },  \label{p*}
\end{equation}%
where by Theorem \ref{heat kernel} the right hand side represents the heat
kernel associated with the triple $\left( d_{\ast },\mu ,\sigma _{\ast
}\right) $. Consider the function 
\begin{equation*}
u\left( r\right) =\frac{1}{\log \frac{1}{\sigma \left( r\right) }},\ \ r\in
\lbrack 0,\infty )
\end{equation*}%
and observe that

\begin{enumerate}
\item $u\left( d\left( x,y\right) \right) =d_{\ast }\left( x,y\right) $, $\
\ u\left( \infty \right) =\infty ;$

\item $\sigma _{\ast }\left( u\left( r\right) \right) =\exp \left( -\frac{1}{%
u\left( r\right) }\right) =\sigma \left( r\right) ;$

\item $B_{u\left( r\right) }^{\ast }\left( x\right) =B_{r}\left( x\right) $
by Lemma \ref{star-balls}.
\end{enumerate}

Making the change $u=u\left( r\right) $ in the integral in (\ref{p*}), we
obtain%
\begin{equation*}
\int_{\lbrack d_{\ast }\left( x,y\right) ,\infty )}\frac{d\sigma _{\ast
}^{t}\left( u\right) }{\mu \left( B_{u}^{\ast }\left( x\right) \right) }%
=\int_{[d\left( x,y\right) ,\infty )}\frac{d\sigma ^{t}\left( r\right) }{\mu
\left( B_{r}\left( x\right) \right) },
\end{equation*}%
which together with (\ref{p=id}) implies (\ref{p*}). Clearly, (\ref{ptxx})
follows from (\ref{p=N}) as $d_{\ast }\left( x,x\right) =0$.

\smallskip

$\left( b\right) $ The change $s=t/u$ in (\ref{p*}) yields%
\begin{eqnarray*}
p\left( t,x,y\right) &=&\int_{[d_{\ast }\left( x,y\right) ,\infty )}\frac{%
d\exp \left( -\frac{t}{u}\right) }{\mu \left( B_{u}^{\ast }\left( x\right)
\right) } \\
&=&\int_{t/d_{\ast }\left( x,y\right) }^{0}\frac{de^{-s}}{\mu \left(
B_{t/s}^{\ast }\left( x\right) \right) } \\
&=&\int_{0}^{t/d_{\ast }\left( x,y\right) }N\left( x,\frac{s}{t}\right)
e^{-s}ds
\end{eqnarray*}%
which proves (\ref{p=N}). Another change $s=t\tau $ transforms (\ref{p=N})
to (\ref{p=Nt}).
\end{proof}

\smallskip

In the case $x=y$ we obtain from (\ref{p=N}) and (\ref{p=Nt}) 
\begin{equation}
p(t,x,x)=\int_{0}^{\infty }N\left( x,\frac{s}{t}\right)
e^{-s}ds=t\int_{0}^{\infty }N(x,\tau )\exp (-\tau t)\,d\tau .  \label{ptxx}
\end{equation}%
Depending on the function $N\left( x,\tau \right) $, the on-diagonal value $%
p\left( t,x,x\right) $ can be equal to $\infty $. For any $x\in X$ set%
\begin{equation}
T\left( x\right) :=\limsup_{\tau \rightarrow \infty }\frac{\log N(x,\tau )}{%
\tau }\,.  \label{critical-time}
\end{equation}

\begin{corollary}
\label{critical time}The function $t\mapsto p\left( t,x,x\right) $ is
monotone decreasing and $p(t,x,x)<\infty \;$for all$\ t>T\left( x\right) \,. 
$
\end{corollary}

\begin{proof}
The monotonicity of $p\left( t,x,x\right) $ follows from the first identity
in (\ref{ptxx}), while the second claim follows from the second identity in (%
\ref{ptxx}). Observe also that if $\lim_{\tau \rightarrow \infty }\frac{\log
N(x,\tau )}{\tau }$ exists and hence is equal to $T\left( x\right) $ then $%
p\left( t,x,x\right) =\infty $ for $t<T\left( x\right) $.
\end{proof}

\begin{proposition}
Assume that $T(x)<\infty $ for some $x\in X$.

\begin{itemize}
\item[$\left( a\right) $] \label{heat at infty-comp}For all $y\in X$, 
\begin{equation*}
\lim_{t\rightarrow \infty }p(t,x,y)=\frac{1}{\mu (X)},
\end{equation*}%
where the convergence is locally uniform in $y\in X.$

\item[$\left( b\right) $] \label{heat-infty-noncomp}For all $y\in X$, 
\begin{equation*}
\lim_{t\rightarrow \infty }\frac{p(t,x,y)}{p(t,x,x)}=1,
\end{equation*}%
where the convergence is locally uniform in $y\in X$.
\end{itemize}
\end{proposition}

\begin{proof}
$\left( a\right) $ As $t\rightarrow \infty $ we have 
\begin{equation*}
N\left( x,\frac{s}{t}\right) \rightarrow N\left( x,0\right) =\frac{1}{\mu
\left( X\right) }
\end{equation*}%
and $t/d_{\ast }\left( x,y\right) \rightarrow \infty .$ Hence, we obtain
from (\ref{p=N})%
\begin{equation*}
\lim_{t\rightarrow \infty }p\left( t,x,y\right) =\int_{0}^{\infty }\frac{1}{%
\mu \left( X\right) }e^{-s}ds=\frac{1}{\mu \left( X\right) },
\end{equation*}%
provided we justify that the integral and $\lim $ are interchangeable. The
latter follows from the dominated convergence theorem, because the
hypothesis $T\left( x\right) <\infty $ implies that, for some $A,a>0$ and
all $\tau >0$, 
\begin{equation}
N\left( x,\tau \right) \leq A\exp \left( a\tau \right)  \label{NAa}
\end{equation}%
whence%
\begin{equation}
N\left( x,\frac{s}{t}\right) e^{-s}\leq A\exp \left( \left( \frac{a}{t}%
-1\right) s\right) \leq A\exp \left( -\frac{1}{2}s\right)  \label{1/2s}
\end{equation}%
for $t>2a$, so that the domination condition is satisfied.

\smallskip

$\left( b\right) $ Set $r=d_{\ast }\left( x,y\right) $. It follows from (\ref%
{p=N}) and ~(\ref{ptxx}) that%
\begin{equation*}
p\left( t,x,x\right) -p\left( t,x,y\right) =\int_{t/r}^{\infty }N\left( x,%
\frac{s}{t}\right) e^{-s}ds.
\end{equation*}%
Assuming $t>2a$ and applying (\ref{1/2s}), we obtain%
\begin{equation*}
p\left( t,x,x\right) -p\left( t,x,y\right) \leq A\int_{t/r}^{\infty }e^{-%
\frac{1}{2}s}ds=2A\exp \left( -\frac{t}{2r}\right) ,
\end{equation*}%
whereas 
\begin{equation}
p(t,x,x)\geq \int_{\tfrac{t}{4r}}^{\infty }N\left( x,\frac{s}{t}\right)
e^{-s}ds\geq N\left( x,\frac{1}{4r}\right) \exp \left( -\frac{t}{4r}\right) .
\notag
\end{equation}%
It follows that%
\begin{equation*}
\frac{p(t,x,x)-p(t,x,y)}{p(t,x,x)}\leq \frac{2A\exp \left( -\frac{t}{4r}%
\right) }{N\left( x,\frac{1}{4r}\right) }\rightarrow 0\ \ \ \text{as }%
t\rightarrow \infty .
\end{equation*}
\end{proof}

\subsection{Estimates of the heat kernel}

\label{estimates}The purpose of this section is to provide some estimates of
the isotropic heat kernel. Recall that by Theorem \ref{p-laplace} 
\begin{equation}
p(t,x,y)=\int_{0}^{t/d_{\ast }\left( x,y\right) }N\left( x,\frac{s}{t}%
\right) e^{-s}ds.  \label{pform}
\end{equation}

\begin{definition}
\label{def-doubling}\RM A monotone increasing function $\Phi :\mathbb{R}%
_{+}\rightarrow \mathbb{R}_{+}$ is said to satisfy the \emph{doubling
property} if there exists a constant $D>0$ such that%
\begin{equation*}
\Phi (2s)\leq D\Phi (s)\quad \text{for all}\;s>0.
\end{equation*}
\end{definition}

It is known (Potter's theorem) that if $\Phi $ is doubling then 
\begin{equation}
\Phi (s_{2})\leq D\left( \frac{s_{2}}{s_{1}}\right) ^{\gamma }\Phi
(s_{1})\quad \text{for all}\;0<s_{1}<s_{2}\,,\;\text{where}\;\gamma =\log
_{2}D.  \label{propert-doubling}
\end{equation}

\begin{theorem}
\label{heat-k-doubling}Suppose that, for some $x\in X$, the function $\tau
\mapsto N(x,\tau )$ is doubling. Then 
\begin{equation}
\frac{c\,t}{t+d_{\ast }(x,y)}\,N\!\left( x,\frac{1}{t+d_{\ast }(x,y)}\right)
\leq p(t,x,y)\leq \frac{C\,t}{t+d_{\ast }(x,y)}\,N\!\left( x,\frac{1}{%
t+d_{\ast }(x,y)}\right)  \label{p2}
\end{equation}%
for all $t>0$, $y\in X$ and some constants $C,c>0$ depending on the doubling
constant.
\end{theorem}

In what follows we will use the relation $f\simeq g$ between two positive
function $f,g$, which means that the ratio $f/g$ is bounded from above and
below by positive constants, for a specified range of the variables. In
particular, we can write (\ref{p2}) shortly in the form%
\begin{equation}
p\left( t,x,y\right) \simeq \frac{t}{t+d_{\ast }(x,y)}\,N\!\left( x,\frac{1}{%
t+d_{\ast }(x,y)}\right)  \label{p21}
\end{equation}%
for a fixed $x$ and all $y\in X,t>0.$

\begin{example}
\label{example1}\label{ex-ab}\RM Assume that, for some $x\in X$ and $\alpha
>0,$ 
\begin{equation*}
N\left( x,\tau \right) \simeq \tau ^{\alpha }\ \ \ \text{for all }\tau >0%
\text{.}
\end{equation*}%
Then by (\ref{p21}) 
\begin{equation*}
p(t,x,y)\simeq \frac{t}{\left( t+d_{\ast }(x,y)\right) ^{1+\alpha }}\simeq 
\frac{t}{\left( t^{2}+d_{\ast }(x,y)^{2}\right) ^{\frac{1+\alpha }{2}}},
\end{equation*}%
that is, $p\left( t,x,y\right) $ behaves like the Cauchy distribution in
\textquotedblleft $\alpha $-dimensional\textquotedblright\ space.
\end{example}

\begin{example}
\label{example2}\label{ex-discrete}\RM More generally, assume that, for some 
$\alpha ,\beta \geq 0,$ 
\begin{equation}
N(x,\tau )\simeq \left\{ 
\begin{array}{cc}
\tau ^{\alpha }, & 0<\tau \leq 1, \\ 
\tau ^{\beta }, & \tau >1.%
\end{array}%
\right.  \label{Nab}
\end{equation}%
Then we obtain by (\ref{p21})%
\begin{equation}
p(t,x,y)\simeq \left\{ 
\begin{array}{ll}
\dfrac{t}{\left( t+d_{\ast }(x,y)\right) ^{1+\beta }}\,, & t+d_{\ast
}(x,y)\leq 1, \\[12pt] 
\dfrac{t}{\left( t+d_{\ast }(x,y)\right) ^{1+\alpha }}\,, & t+d_{\ast
}(x,y)>1.%
\end{array}%
\right.  \label{pab}
\end{equation}%
For example, let $X$ be a discrete locally finite group, like $%
X=\bigoplus_{k=1}^{\infty }\mathbb{Z}\left( n_{k}\right) $, and $\mu $ be
the Haar measure, normalized to $\mu \left( x\right) =1.$ With the discrete
ultra-metric \thinspace we obtain by (\ref{Ndef}) that $N\left( x,\tau
\right) \simeq 1$ for large enough $\tau $. Assuming additionally that 
\begin{equation*}
N\left( x,\tau \right) \simeq \tau ^{\alpha }\ \text{for small }\tau ,
\end{equation*}%
we see that (\ref{Nab}) and, hence, (\ref{pab}) hold with $\beta =0$ (cf. 
\cite{Car1988}).
\end{example}

\begin{example}
\label{Exlog}\RM Assume that $\tau \mapsto N\left( x,\tau \right) $ is
doubling and, for some $\alpha >0$, 
\begin{equation*}
N\left( x,\tau \right) \simeq \left( \log \frac{1}{\tau }\right) ^{-\alpha
}\ \ \ \text{for }\tau <\frac{1}{2}.
\end{equation*}%
Then by (\ref{p21}) 
\begin{equation*}
p\left( t,x,y\right) \simeq \frac{t}{\left( t+d_{\ast }(x,y)\right) \log
^{\alpha }\left( t+d_{\ast }\left( x,y\right) \right) }
\end{equation*}%
provided $t+d_{\ast }\left( x,y\right) >2.$
\end{example}

\begin{example}
\label{Exexp}\RM Assume that, for some $\alpha >0$, 
\begin{equation*}
N\left( x,\tau \right) \simeq \exp \left( -\tau ^{-\alpha }\right) .
\end{equation*}%
In this case Theorem \ref{heat-k-doubling} does not apply. An ad hoc method
of estimating the integral in (\ref{pform}) yields in this case 
\begin{equation*}
p\left( t,x,y\right) \leq \frac{C_{3}t}{t+d_{\ast }\left( x,y\right) }\exp
\left( -c_{3}\left( t^{\frac{\alpha }{\alpha +1}}+d_{\ast }\left( x,y\right)
^{\alpha }\right) \right)
\end{equation*}%
and 
\begin{equation*}
p\left( t,x,y\right) \geq \frac{C_{4}t}{t+d_{\ast }\left( x,y\right) }\exp
\left( -c_{4}\left( t^{\frac{\alpha }{\alpha +1}}+d_{\ast }\left( x,y\right)
^{\alpha }\right) \right) ,
\end{equation*}%
for all $x,y\in X,$ $t>0$ and some positive constants $%
C_{3,}C_{4},c_{3},c_{4}$.\label{rem: details}
\end{example}

For the proof of Theorem \ref{heat-k-doubling} we need a sequence of lemmas.

\begin{lemma}
\label{1st estimates} For all $x,y\in X$ and $t>0$ the following estimates
hold.

\begin{enumerate}
\item[$\left( a\right) $] 
\begin{equation}
p(t,x,y)\leq \frac{t}{d_{\ast }(x,y)}\,N\!\left( x,\frac{1}{d_{\ast }(x,y)}%
\right) .  \label{p<N}
\end{equation}

\item[$\left( b\right) $] 
\begin{equation}
p(t,x,y)\geq \frac{1}{2e}\left\{ 
\begin{array}{ll}
\dfrac{t}{d_{\ast }(x,y)}\,N\!\left( x,\dfrac{1}{2d_{\ast }(x,y)}\right) ,\ 
& \ t\leq d_{\ast }\left( x,y\right) , \\[12pt] 
N\left( x,\dfrac{1}{2t}\right) , & \ t\geq d_{\ast }\left( x,y\right) .%
\end{array}%
\right.  \label{p>N}
\end{equation}

\item[$\left( c\right) $] 
\begin{equation}
p(t,x,x)\geq \frac{1}{e}N\left( x,\frac{1}{t}\right) .  \label{ptxx>N}
\end{equation}
\end{enumerate}
\end{lemma}

\begin{proof}
$\left( a\right) $ Inequality (\ref{p<N}) follows from (\ref{pform}) using
the monotonicity of $\tau \mapsto N\left( x,\tau \right) $ that yields%
\begin{equation*}
N\left( x,\frac{s}{t}\right) e^{-s}\leq N\left( x,\frac{1}{d_{\ast }\left(
x,y\right) }\right) .
\end{equation*}

$\left( b\right) $ Set $a=\min \left( 1,\frac{t}{d_{\ast }\left( x,y\right) }%
\right) $. It follows from (\ref{pform}) that%
\begin{equation*}
p(t,x,y)\geq \int_{a/2}^{a}N(x,\frac{s}{t})e^{-s}\,ds\geq N\left( x,\frac{a}{%
2t}\right) \frac{a}{2e}\,,
\end{equation*}%
which is equivalent to (\ref{p>N}).

\smallskip

$\left( c\right) $ We have by (\ref{ptxx})%
\begin{equation*}
p(t,x,x)\geq \int_{1}^{\infty }N(x,\frac{s}{t})e^{-s}ds\geq N\!\left( x,%
\frac{1}{t}\right) \int_{1}^{\infty }e^{-s}ds,
\end{equation*}%
whence (\ref{ptxx>N}) follows.
\end{proof}

\begin{lemma}
\label{heat-k-estimates}The following inequalities hold for all $x,y\in X$
and $t>0$: 
\begin{equation}
p(t,x,y)\geq \frac{1}{2e}\,\frac{t}{t+d_{\ast }(x,y)}\,N\!\left( x,\frac{1}{%
2\left( t+d_{\ast }(x,y)\right) }\right) ,  \label{ptxy>}
\end{equation}%
and%
\begin{equation}
p(t,x,y)\leq 2e\,\frac{t}{t+d_{\ast }(x,y)}\,p\!\left( \frac{t+d_{\ast }(x,y)%
}{2},x,x\right) .  \label{ptxy<}
\end{equation}
\end{lemma}

\begin{proof}
The lower bound (\ref{ptxy>}) follows immediately from (\ref{p>N}). To prove
(\ref{ptxy<}), observe that by (\ref{p<N}) and (\ref{ptxx>N})%
\begin{equation*}
p\left( t,x,y\right) \leq e\,\frac{t}{d_{\ast }\left( x,y\right) }\,p\left(
d_{\ast }\left( x,y\right) ,x,x\right) ,
\end{equation*}%
which yields (\ref{ptxy<}) in the case $t\leq d_{\ast }\left( x,y\right) $
as the function $p\left( \cdot ,x,x\right) $ is monotone decreasing. In the
case $t>d_{\ast }\left( x,y\right) $ (\ref{ptxy<}) follows trivially from (%
\ref{ptmin}), that is, from 
\begin{equation*}
p\left( t,x,y\right) \leq p\left( t,x,x\right) ,
\end{equation*}%
using again the monotonicity of $p\left( \cdot ,x,x\right) $.
\end{proof}

\begin{lemma}
\label{ondiag-doubling}For any given $x\in X$, the following two properties
are equivalent.

\begin{enumerate}
\item[$\left( i\right) $] For some constant $C$ and all $t>0$, 
\begin{equation}
p(t,x,x)\leq CN\left( x,\frac{1}{t}\right) .  \label{ptxxtwo}
\end{equation}

\item[$\left( ii\right) $] The function $\tau \mapsto N(x,\tau )$ is
doubling, that is, for some constant $D$,%
\begin{equation*}
N\left( x,2\tau \right) \leq DN\left( x,\tau .\right)
\end{equation*}
\end{enumerate}
\end{lemma}

\begin{proof}
$\left( ii\right) \Rightarrow \left( i\right) $. The estimate (\ref{ptxxtwo}%
) follows from (\ref{ptxx}) and (\ref{propert-doubling}) as follows: 
\begin{eqnarray*}
p(t,x,x) &=&N\left( x,\frac{1}{t}\right) \int_{0}^{\infty }\frac{N\left( x,%
\frac{s}{t}\right) }{N\left( x,\frac{1}{t}\right) }e^{-s}ds \\
&\leq &DN\left( x,\frac{1}{t}\right) \int_{0}^{\infty }\max \{1,s^{\gamma
}\}e^{-s}ds \\
&=&CN\left( x,\frac{1}{t}\right) .
\end{eqnarray*}

$\left( i\right) \Rightarrow \left( ii\right) $. The upper bound (\ref%
{ptxxtwo}) implies, for any $t>0$, 
\begin{eqnarray*}
CN\!\left( x,\frac{1}{t}\right) &\geq &p(t,x,x)\geq \int_{2}^{\infty }N(x,%
\frac{s}{t})e^{-s}ds \\
&\geq &e^{-2}\,N\!\left( x,\frac{2}{t}\right) ,
\end{eqnarray*}%
whence the doubling property of $N\left( x,\cdot \right) $ follows.
\end{proof}

\smallskip

\begin{proof}[Proof of Theorem \protect\ref{heat-k-doubling}]
The lower bound in (\ref{p2}) follows from (\ref{ptxy>}), the upper bound
follows from (\ref{ptxy<}) and (\ref{ptxxtwo}).
\end{proof}

\smallskip

In conclusion of this section we provide practicable conditions for the
validity of the doubling property of $N\left( x,\cdot \right) $.

\begin{definition}
\label{Reverse-doubling}\RM A monotone increasing function $\Psi :\mathbb{R}%
_{+}\rightarrow \mathbb{R}_{+}$ is said to satisfy the \emph{reverse
doubling property,} if there is a constant $\delta \in \left( 0,1\right) $
such that for all $r>0$%
\begin{equation*}
\Psi (r)\geq 2\Psi (\delta r).
\end{equation*}
\end{definition}

\begin{proposition}
\label{Ndoubling} Fix some $x\in X$. The function $\tau \mapsto N(x,\tau )$
is doubling provided the following two conditions hold:

\begin{enumerate}
\item[$\left( i\right) $] The function $\Psi (r)=-1/\log \sigma (r)$
satisfies the reverse doubling property.

\item[$\left( ii\right) $] The volume function $r\mapsto \mu \left(
B_{r}(x)\right) $ satisfies the doubling property.
\end{enumerate}
\end{proposition}

\begin{proof}
We use the following short notation for the balls centered at $x$: $%
B_{r}=B_{r}\left( x\right) $ and $B_{r}^{\ast }=B_{r}^{\ast }\left( x\right) 
$. It follows from the Definition \ref{def-spectral-d} of the spectral
distribution function that $\tau \mapsto N(x,\tau )$ is doubling if and only
if the function $s\mapsto \mu (B_{s}^{\ast })$ is doubling. Set $\Phi =\Psi
^{-1}$ and observe that the reverse doubling property for $\Psi $ is
equivalent to the doubling property for $\Phi $. By Lemma \ref{star-balls}
we have $B_{\Psi \left( r\right) }^{\ast }=B_{r}$ which implies that $%
B_{s}^{\ast }=B_{\Phi \left( s\right) }$. Using the hypotheses $\left(
ii\right) $ and (\ref{propert-doubling}) for the function $\mu \left(
B_{r}\right) $, we obtain%
\begin{equation*}
\mu \left( B_{2s}^{\ast }\right) =\mu \left( B_{\Phi \left( 2s\right)
}\right) \leq D\left( \frac{\Phi \left( 2s\right) }{\Phi \left( s\right) }%
\right) ^{\gamma }\mu \left( B_{\Phi \left( s\right) }\right) \leq \func{%
const}\mu \left( B_{s}^{\ast }\right) ,
\end{equation*}%
which was to be proved.
\end{proof}

\subsection{Heat kernels in $\mathbb{Q}_{p}$}

\label{Sec p-adic-doubling copy} Given a prime $p$, the $p$-adic norm on $%
\mathbb{Q}$ is defined as follows: if $x=p^{n}\,\frac{a}{b}$, where $a,b$
are integers not divisible by $p$, then 
\begin{equation*}
\left\Vert x\right\Vert _{p}:=p^{-n}.
\end{equation*}%
If $x=0$ then $\left\Vert x\right\Vert _{p}:=0.$ The $p$-adic norm on $%
\mathbb{Q}$ satisfies the ultra-metric inequality. Indeed, if $y=p^{m}\frac{c%
}{d}$ and $m\leq n$ then 
\begin{equation*}
x+y=p^{m}\left( \frac{p^{n-m}a}{b}+\frac{c}{d}\right)
\end{equation*}%
whence%
\begin{equation*}
\left\Vert x+y\right\Vert _{p}\leq p^{-m}=\max \left\{ \left\Vert
x\right\Vert _{p},\left\Vert y\right\Vert _{p}\right\} .
\end{equation*}%
Hence, $\mathbb{Q}$ with the metric $d\left( x,y\right) =\left\Vert
x-y\right\Vert _{p}$ is an ultra-metric space, and so is its completion $%
\mathbb{Q}_{p}$ -- the field of $p$-adic numbers.

Every $p$-adic number $x$ has a representation%
\begin{equation}
x=\sum_{k=-N}^{\infty }a_{k}p^{k}=...a_{k}...a_{2}a_{1}a_{0\cdot
}a_{-1}a_{-2}...a_{-N}  \label{xa}
\end{equation}%
where $N\in \mathbb{N}$ and $a_{k} \in \{ 0, \dots, p-1 \}$ are $p$-adic
digits. The rational number $0.a_{-1}...a_{-N}=\sum_{k=-N}^{k=-1}a_{k}p^{k}$
is called the fractional part of $x$ and the rest $\sum_{k=0}^{\infty
}a_{k}p^{k}$ is the integer part of $x$.

For any $n\in \mathbb{Z}$, the $d$-ball $B_{p^{-n}}\left( x\right) $
consists of all numbers%
\begin{equation*}
y=\sum_{k=-N}^{\infty }b_{k}p^{k}=...b_{k}...b_{2}b_{1}b_{0\cdot
}b_{-1}b_{-2}...b_{-N}
\end{equation*}%
such that $b_{k}$ are arbitrary for $k\geq n$ and $b_{k}=a_{k}$ for $k<n$.
It follows that $B_{p^{-n}}\left( x\right) $ decomposes into a disjoint
union of $p$ balls of radii $p^{-\left( n+1\right) }$ depending on the
choice of $b_{n}$.

For example, $B_{1}\left( 0\right) $ coincides with the set $\mathbb{Z}_{p}$
of all $p$-adic integers, that is, any $y\in B_{1}\left( 0\right) $ has the
form%
\begin{equation*}
y=...b_{k}...b_{2}b_{1}b_{0}
\end{equation*}%
with arbitrary $p$-adic digits $b_{k}$. For any fixed $c=0,1,...,p-1$, the
additional restriction $b_{0}=c$ determines a ball of radius $1/p$ centered
at $c$, so that $B_{1}\left( 0\right) $ is a disjoint union of $p$ such
balls, as on the following diagram, where every cell renders one of the
balls $B_{1/p}\left( c\right) $:

\begin{equation*}
\begin{tabular}{|c|c|c|c|}
\hline
$...b_{k}...b_{2}b_{1}0$ & $...b_{k}...b_{2}b_{1}1$ & $\ \ \ \ \ \ \ \ \
...\ \ \ \ \ \ \ \ $ & $...b_{k}...b_{2}b_{1}\left( p-1\right) $ \\ \hline
\end{tabular}%
\end{equation*}

Let $\mu $ be the additive Haar measure on $\mathbb{Q}_{p}$ normalized so
that $\mu \left( B_{1}\left( 0\right) \right) =1$. Since 
\begin{equation*}
B_{r}\left( x\right) =x+B_{r}\left( 0\right)
\end{equation*}%
and $\mu $ is translation invariant, we obtain that $\mu \left( B_{r}\left(
x\right) \right) $ does not depend on $x$. The above argument with the
decomposition of the ball $B_{p^{-n}}\left( x\right) $ implies that%
\begin{equation*}
\mu \left( B_{p^{-n}}\left( x\right) \right) =p\mu \left( B_{p^{-\left(
n+1\right) }}\left( x\right) \right) ,
\end{equation*}%
whence it follows that%
\begin{equation}
\mu \left( B_{p^{-n}}\left( x\right) \right) =p^{-n}.  \label{p-n}
\end{equation}%
For any $r>0$, the ball $B_{r}\left( x\right) $ coincides with $%
B_{p^{-n}}\left( x\right) $, where $n\in \mathbb{Z}$ is such that $%
p^{-n}\leq r<p^{-\left( n-1\right) },$ which implies that, for all $r>0$,%
\begin{equation}
r/p<\mu \left( B_{r}\left( x\right) \right) \leq r.  \label{rp}
\end{equation}

\begin{example}
\label{p-adic-doubling}\RM Let $\left( X,d,\mu \right) $ be $\mathbb{Q}_{p}$
with $p$-adic distance and the Haar measure $\mu $. Consider the distance
distribution function 
\begin{equation*}
\sigma (r)=\exp \left( -(b/r)^{\alpha }\right) ,
\end{equation*}%
where $\alpha ,b>0.$ Since 
\begin{equation*}
\Psi (r):=\frac{1}{\log \frac{1}{\sigma \left( r\right) }}=\left( r/b\right)
^{\alpha },
\end{equation*}%
we obtain by (\ref{d*})%
\begin{equation}
d_{\ast }(x,y)=\Psi \left( d\left( x,y\right) \right) =\left( \frac{%
\left\Vert x-y\right\Vert _{p}}{b}\right) ^{\alpha }.  \label{d*p}
\end{equation}%
By Lemma \ref{star-balls}, we have%
\begin{equation*}
B_{s}^{\ast }\left( x\right) =B_{\Psi ^{-1}\left( s\right) }\left( x\right) ,
\end{equation*}%
which together with (\ref{rp}) yields 
\begin{equation}
\quad \mu \left( B_{s}^{\ast }(x)\right) \simeq \,s^{1/\alpha }\,.
\label{muBp}
\end{equation}%
Consequently, we obtain%
\begin{equation*}
N(x,\tau )\simeq \tau ^{1/\alpha }.
\end{equation*}%
Since this function is doubling, Theorem \ref{heat-k-doubling} (cf. also
Example \ref{example1}) yields the estimate%
\begin{equation*}
p(t,x,y)\simeq \frac{t}{\left( t+d_{\ast }(x,y)\right) ^{1+1/\alpha }}\simeq 
\frac{\,t}{(t^{1/\alpha }+\left\Vert x-y\right\Vert _{p})^{1+\alpha }}.
\end{equation*}%
In particular, for all $t>0$ and $x\in X$%
\begin{equation*}
p\left( t,x,x\right) \simeq t^{1/\alpha }.
\end{equation*}
\end{example}

\begin{example}
\RM Let $X=\mathbb{Z}_{p}$, that is, $X$ is the unit ball $B_{1}\left(
0\right) $ in $\mathbb{Q}_{p}$, with the $p$-adic distance and the Haar
measure $\mu $. Consider the distance distribution function 
\begin{equation*}
\sigma \left( r\right) =\exp \left( 1-\exp r^{-\alpha }\right) ,
\end{equation*}%
for some $\alpha >0$. Since for $r\leq 1$%
\begin{equation*}
\Psi (r):=\frac{1}{\log \frac{1}{\sigma \left( r\right) }}=\frac{1}{\exp
r^{-\alpha }-1}\simeq \exp \left( -r^{-\alpha }\right) ,
\end{equation*}%
we obtain that%
\begin{equation*}
d_{\ast }(x,y)=\Psi \left( d\left( x,y\right) \right) \simeq \exp \left(
-\left\Vert x-y\right\Vert _{p}^{-\alpha }\right) .
\end{equation*}%
By Lemma \ref{star-balls} and (\ref{rp}), we have, for all $s\leq \frac{1}{2}
$,%
\begin{equation*}
\mu \left( B_{s}^{\ast }\left( x\right) \right) =\mu \left( B_{\Psi
^{-1}\left( s\right) }\left( x\right) \right) \simeq \Psi ^{-1}\left(
s\right) \simeq \frac{1}{\log ^{1/\alpha }\frac{1}{s}},
\end{equation*}%
whereas for $s>\frac{1}{2}$ we have $\mu \left( B_{s}^{\ast }\left( x\right)
\right) \simeq 1$. Therefore, we obtain, for all $\tau >0$,%
\begin{equation*}
N\left( x,\tau \right) =\frac{1}{\mu \left( B_{1/\tau }^{\ast }\left(
x\right) \right) }\simeq \log ^{1/\alpha }\left( 2+\tau \right) .
\end{equation*}%
Hence, the function $N\left( x,\tau \right) $ is doubling, and we obtain by (%
\ref{p21}) that%
\begin{equation*}
p(t,x,y)\simeq \frac{t}{t+\exp \left( -\left\Vert x-y\right\Vert
_{p}^{-\alpha }\right) }\log ^{1/\alpha }\left( 2+\frac{1}{t+\exp \left(
-\left\Vert x-y\right\Vert _{p}^{-\alpha }\right) }\right) .
\end{equation*}
\end{example}

\begin{example}
\RM\label{ExQp-frac}Let $X$ be the subset of $\mathbb{Q}_{p}$ consisting of
all $p$-adic fractions, that is, the numbers of the form $%
x=0.a_{-1}....a_{-N}$. Then the $p$-adic distance $d$ on $X$ takes only
integer values so that $\left( X,d\right) $ is a discrete space. Let $\mu $
be the counting measure on $X$, that is, $\mu \left( x\right) =1$ for any $%
x\in X$. Consider the following distance distribution function%
\begin{equation}
\sigma \left( r\right) =\exp \left( -\frac{1}{\log ^{\alpha }\left(
2r\right) }\right) \ \ \text{for }r\geq 1,  \label{sigmalog}
\end{equation}%
that is arbitrarily extended to $r<1$ to be strictly monotone increasing and
to have $\sigma \left( 0\right) =0$. Since 
\begin{equation*}
\Psi (r):=\frac{1}{\log \frac{1}{\sigma \left( r\right) }}=\log ^{\alpha
}\left( 2r\right) \ \ \text{for }r\geq 1,
\end{equation*}%
we obtain, for $x\neq y$, 
\begin{equation}
d_{\ast }(x,y)=\Psi \left( d\left( x,y\right) \right) =\log ^{\alpha }\left(
2\left\Vert x-y\right\Vert _{p}\right) .  \label{d*log}
\end{equation}%
For $s\geq s_{0}:=\log ^{\alpha }2$, we have%
\begin{equation}
\mu \left( B_{s}^{\ast }\left( x\right) \right) =\mu \left( B_{\Psi
^{-1}\left( s\right) }\left( x\right) \right) \simeq \Psi ^{-1}\left(
s\right) =\frac{1}{2}\exp \left( s^{1/\alpha }\right) ,  \label{muBexp}
\end{equation}%
whereas for $s<s_{0}$ we have $\mu \left( B_{s}^{\ast }\left( x\right)
\right) \simeq \mu \left( x\right) =1$. We see that (\ref{muBexp}) holds for
all $s>0$. It follows that, for all $\tau >0$,%
\begin{equation}
\quad N\left( x,\tau \right) =\frac{1}{\mu \left( B_{1/\tau }^{\ast }\left(
x\right) \right) }\simeq \exp \left( -\tau ^{-1/\alpha }\right) .
\label{Nexp}
\end{equation}%
By Example \ref{Exexp}, we obtain 
\begin{equation*}
p\left( t,x,y\right) \leq \frac{Ct}{t+\log _{+}^{\alpha }\left( 2\left\Vert
x-y\right\Vert _{p}\right) }\exp \left( -c\left( t^{\frac{1}{\alpha +1}%
}+\log _{+}\left( 2\left\Vert x-y\right\Vert _{p}\right) \right) \right) ,
\end{equation*}%
and a similar lower bound.
\end{example}

\subsection{Green function and transience}

\label{Green}Given an isotropic heat semigroup $\left\{ P_{t}\right\} $,
define the Green operator $G$ on non-negative Borel functions $f$ on $X$ by%
\begin{equation*}
Gf\left( x\right) =\int_{0}^{\infty }P_{t}f\left( x\right) dt.
\end{equation*}%
Of course, the value of $Gf\left( x\right) $ could be $\infty $. By Fubini's
theorem, we obtain%
\begin{equation*}
Gf\left( x\right) =\int_{X}g\left( x,y\right) f\left( y\right) d\mu \left(
y\right)
\end{equation*}%
where 
\begin{equation*}
g\left( x,y\right) =\int_{0}^{\infty }p\left( t,x,y\right) dt.
\end{equation*}%
Substituting the heat kernel from (\ref{p=Nt}) and using again Fubini's
theorem, we obtain%
\begin{equation}
g(x,y)=\int_{0}^{1/d_{\ast }(x,y)}\frac{N(x,\tau )\,d\tau }{\tau ^{2}}%
=\int_{d_{\ast }\left( x,y\right) }^{\infty }\frac{ds}{\mu \left(
B_{s}^{\ast }\left( x\right) \right) },  \label{green function}
\end{equation}%
where the second identity follows from (\ref{Ndef}). The function $g\left(
x,y\right) $ is called the \emph{Green function} of the semigroup $\left\{
P_{t}\right\} $. Note that the Green function can be identically equal to $%
\infty $. For example, this is the case when $\mu \left( X\right) <\infty $
(cf. Figure \ref{pic1}) and the second integral (\ref{green function})
diverges at $\infty $.

\begin{definition}
\label{Transient semigroup-def}\RM The process $\{X_{t}\}$ and the semigroup 
$\{P^{t}\}$ are called \emph{transient} if $Gf$ is a bounded function
whenever $f$ is bounded and has compact support, and \emph{recurrent}
otherwise.
\end{definition}

\begin{theorem}
\label{Transient semigroup}The following statements are equivalent.

\begin{enumerate}
\item[$\left( i\right) $] The semigroup $\{P^{t}\}$ is transient.

\item[$\left( ii\right) $] $g(x,y)<\infty $ for some/all distinct $x,y\in X.$

\item[$\left( iii\right) $] For some/all $x\in X,$ 
\begin{equation}
\int^{\infty }\frac{ds}{\mu \left( B_{s}^{\ast }\left( x\right) \right) }%
<\infty .  \label{Bx}
\end{equation}
\end{enumerate}
\end{theorem}

The inequality (\ref{Bx}) is equivalent to 
\begin{equation}
\int_{0}\frac{N(x,\tau )\,d\tau }{\tau ^{2}}<\infty \,.  \label{Nx}
\end{equation}%
Observe that, in the transient case, the function $x,y\mapsto \frac{1}{%
g\left( x,y\right) }$ determines an ultra-metric on $X$, which is proved
similarly to Corollary \ref{RemF(p)}.

\begin{proof}
The validity of the condition (\ref{Bx}) is independent of the choice of $x$
because for any two $x,x^{\prime }\in X$ the balls $B_{s}^{\ast }\left(
x\right) $ and $B_{s}^{\ast }\left( x^{\prime }\right) $ are identical
provided $s\geq d\left( x,x^{\prime }\right) $. The finiteness of the second
integral in (\ref{green function}) for $x\neq y$ is clearly equivalent to (%
\ref{Bx}), whence the equivalence $\left( ii\right) \Leftrightarrow \left(
iii\right) $ follows, with all combinations of some/all options.

The finiteness of $Gf$ for any bounded function $f$ with compact support
clearly implies that $g\left( x,y\right) \not\equiv \infty $, that is, $%
\left( i\right) \Rightarrow \left( ii\right) $. So, it remains to prove $%
\left( iii\right) \Rightarrow \left( i\right) $. It suffices to show that $%
Gf $ is bounded for $f=\mathbf{1}_{A}$ where $A$ is a bounded Borel subset
of $X $. Let $R$ be the diameter of $A$ with respect to the distance $%
d^{\ast }$. Then we have $A\subset B_{R}^{\ast }\left( x\right) $ for any $%
x\in A$ whence by (\ref{green function}) 
\begin{eqnarray*}
Gf\left( x\right) &=&\int_{A}g\left( x,y\right) d\mu \left( y\right) \leq
\int_{B_{R}^{\ast }\left( x\right) }g\left( x,y\right) d\mu \left( y\right)
\\
&=&\int_{B_{R}^{\ast }\left( x\right) }\int_{0}^{\infty }\mathbf{1}%
_{[d_{\ast }\left( x,y\right) ,\infty )}\left( s\right) \frac{ds}{\mu \left(
B_{s}^{\ast }\left( x\right) \right) }d\mu \left( y\right) \\
&=&\int_{0}^{\infty }\frac{1}{\mu \left( B_{s}^{\ast }\left( x\right)
\right) }\left( \int_{B_{R}^{\ast }\left( x\right) }\mathbf{1}_{\left[ 0,s%
\right] }\left( d_{\ast }\left( x,y\right) \right) d\mu \left( y\right)
\right) ds \\
&=&\int_{0}^{\infty }\frac{1}{\mu \left( B_{s}^{\ast }\left( x\right)
\right) }\mu \left( B_{R}^{\ast }\left( x\right) \cap B_{s}^{\ast }\left(
x\right) \right) ds.
\end{eqnarray*}%
For $s\geq R$ the integrand is equal to $\frac{1}{\mu \left( B_{s}^{\ast
}\left( x\right) \right) }\mu \left( B_{R}^{\ast }\left( x\right) \right) $
so that the convergence at $\infty $ follows from (\ref{Bx}). The
convergence is clearly uniform in $x\in A$ because $\mu \left( B_{R}^{\ast
}\left( x\right) \right) $ and $\mu \left( B_{s}^{\ast }\left( x\right)
\right) \ $are independent of $x\in A$ for $s\geq R$. For $s\leq R$ the
integrand is equal to%
\begin{equation*}
\frac{1}{\mu \left( B_{s}^{\ast }\left( x\right) \right) }\mu \left(
B_{s}^{\ast }\left( x\right) \right) =1,
\end{equation*}%
whence the uniform convergence at $0$ follows. Hence, $\sup_{A}Gf\left(
x\right) <\infty $. That $\sup_{X}Gf\left( x\right) <\infty $ follows from
the decay of $g\left( x,y\right) $ in $d_{\ast }\left( x,y\right) $.
\end{proof}

Let us note that if $X$ is a locally finite group with the Haar measure $\mu 
$, then the transience criterion $\left( iii\right) $ of Theorem \ref%
{Transient semigroup} coincides with the general sufficient condition of
transience of \cite{Lawler1995}.

Now let us provide some estimate of the Green function. Set 
\begin{equation}
V(x,r)=\mu \left( B_{r}^{\ast }(x)\right) .  \label{volume}
\end{equation}

\begin{theorem}
\label{Green-f asymptotic}Assume that there exist constants $1<c<c^{\prime
}<c^{\prime \prime }$ such that for all $r>r_{0}\geq 0$ and some $x\in X$%
\begin{equation}
c^{\prime }\leq \frac{V\left( x,cr\right) }{V\left( x,r\right) }\leq
c^{\prime \prime }.  \label{N-Tauberian}
\end{equation}%
Then the semigroup $\left\{ P_{t}\right\} $ is transient and, for all $y\in
X $ such that $r:=d^{\ast }\left( x,y\right) >r_{0}$, we have%
\begin{equation*}
g\left( x,y\right) \simeq \frac{r}{V\left( x,r\right) }.
\end{equation*}
\end{theorem}

Note that the condition $\frac{V\left( x,cr\right) }{V\left( x,r\right) }%
\leq c^{\prime \prime }$ is equivalent to the doubling property of $r\mapsto
V\left( x,r\right) $ (cf. Definition \ref{def-doubling}), whereas the
condition $\frac{V\left( x,cr\right) }{V\left( x,r\right) }\geq c^{\prime }$
with $c^{\prime }>c$ is somewhat stronger than the reverse doubling property
(cf. Definition \ref{Reverse-doubling}). For example, (\ref{N-Tauberian})
holds for $V\left( x,r\right) \simeq r^{\alpha }$ if and only if $\alpha >1$.

\begin{proof}
Set for simplicity of notation $V\left( s\right) :=V\left( x,s\right) $. For 
$r>r_{0}$ we have%
\begin{equation*}
g\left( x,y\right) =\int_{r}^{\infty }\frac{ds}{V\left( s\right) }%
=\sum_{k=0}^{\infty }\int_{c^{k}r}^{c^{k+1}r}\frac{ds}{V\left( s\right) }%
=\sum_{k=0}^{\infty }c^{k}\int_{r}^{cr}\frac{ds}{V\left( c^{k}s\right) }
\end{equation*}%
Using the lower bound in (\ref{N-Tauberian}), we obtain%
\begin{equation*}
\int_{r}^{\infty }\frac{ds}{V\left( s\right) }\leq \sum_{k=0}^{\infty
}c^{k}\int_{r}^{cr}\frac{\left( c^{\prime }\right) ^{-k}ds}{V\left( s\right) 
}\leq \sum_{k=0}^{\infty }\left( \frac{c}{c^{\prime }}\right) ^{k}\frac{cr}{%
V\left( r\right) }\leq \func{const}\frac{r}{V\left( r\right) },
\end{equation*}%
where the series converges due to $c^{\prime }>c$. Similarly, using the
upper bound in (\ref{N-Tauberian}), we obtain%
\begin{equation*}
\int_{r}^{\infty }\frac{ds}{V\left( s\right) }\geq \int_{r}^{cr}\frac{ds}{%
V\left( s\right) }\geq \frac{\left( c-1\right) r}{V\left( cr\right) }\geq 
\func{const}\frac{r}{V\left( r\right) },
\end{equation*}%
which finishes the proof.
\end{proof}

\begin{example}
\RM\label{p-adic-transient} Let $\left( X,d,\mu \right) $ and $\sigma $ be
as in Example \ref{p-adic-doubling}, that is, $X=\mathbb{Q}_{p}$ is the
field of $p$-adic numbers with ultra-metric $d(x,y)=\left\Vert
x-y\right\Vert _{p}$ and $\sigma (r)=\exp \left( -(b/r)^{\alpha }\right) $.
Then by (\ref{d*p}) we have 
\begin{equation*}
d_{\ast }\left( x,y\right) =\func{const}\left\Vert x-y\right\Vert
_{p}^{\alpha }
\end{equation*}%
and by (\ref{muBp}) 
\begin{equation*}
V(x,r)\simeq r^{1/\alpha }.
\end{equation*}%
Therefore, by Theorem \ref{Transient semigroup}, the semigroup $\{P^{t}\}$
is transient if and only if $\alpha <1.$ Moreover, the condition (\ref%
{N-Tauberian}) is fulfilled also if and only if $\alpha <1$, and in this
case we obtain by Theorem \ref{Green-f asymptotic} that, for all $x,y,$%
\begin{equation*}
g(x,y)\simeq d_{\ast }\left( x,y\right) ^{1-\frac{1}{\alpha }}\simeq
\left\Vert x-y\right\Vert _{p}^{\alpha -1}.
\end{equation*}
\end{example}

\begin{example}
\RM Let $\left( X,d,\mu \right) $ and $\sigma $ be as in Example \ref%
{ExQp-frac}, that is, $X$ is the set of fractional $p$-adic numbers and $%
\sigma $ is given by (\ref{sigmalog}). By (\ref{d*log}) we have, for $x\neq
y $, 
\begin{equation*}
d_{\ast }(x,y)=\log ^{\alpha }\left( 2\left\Vert x-y\right\Vert _{p}\right)
\end{equation*}%
and by (\ref{muBexp}) 
\begin{equation*}
V\left( x,r\right) \simeq \exp \left( r^{1/\alpha }\right) .
\end{equation*}%
By Theorem \ref{Transient semigroup} we conclude that the semigroup $\left\{
P_{t}\right\} $ is transient. Theorem \ref{Green-f asymptotic} does not
apply in this case, by a direct estimate of the integral in (\ref{green
function}) yields, for $r:=d^{\ast }\left( x,y\right) ,$%
\begin{equation*}
g\left( x,y\right) =\int_{r}^{\infty }\frac{ds}{V\left( x,s\right) }\simeq
\int_{r}^{\infty }\exp \left( -s^{1/\alpha }\right) ds\simeq r^{1-1/\alpha
}\exp \left( -r^{1/\alpha }\right) ,
\end{equation*}%
whence, for $x\neq y$,%
\begin{equation*}
g\left( x,y\right) \simeq \left\Vert x-y\right\Vert _{p}^{-1}\log ^{\alpha
-1}\left( 2\left\Vert x-y\right\Vert _{p}\right) .
\end{equation*}
\end{example}

\section{The Laplacian and its spectrum}

\label{generator}\setcounter{equation}{0}In this section we are concerned
with the properties of the generator of the isotropic semigroup $\left\{
P^{t}\right\} $. By definition, the generator $\mathcal{L}$ of a strongly
continuous semigroup $\left\{ P_{t}\right\} _{t\geq 0}$ in a Banach space is
defined by%
\begin{equation*}
\mathcal{L}f=s\text{-}\lim_{t\rightarrow 0}\frac{f-P_{t}f}{t}
\end{equation*}%
and the domain $\func{dom}_{\mathcal{L}}$ consists of those $f$ for which
the above limit exists. Since the isotropic semigroup $\left\{ P^{t}\right\} 
$ is symmetric and acts in a Hilbert space $L^{2}\left( X,\mu \right) $, the
above definition is equivalent to the following: $\mathcal{L}$ is a
self-adjoint (unbounded) operator in $L^{2}\left( X,\mu \right) $ such that 
\begin{equation*}
P^{t}=\exp \left( -t\mathcal{L}\right) \ \text{for all }t>0.
\end{equation*}%
Obviously, this is equivalent to $P=\exp \left( -\mathcal{L}\right) $, which
leads to the identity%
\begin{equation*}
\mathcal{L}=\log \frac{1}{P}~,
\end{equation*}%
where the right hand side is understood in the sense of functional calculus
of self-adjoint operators. We refer to $\mathcal{L}$ as an \emph{isotropic
Laplace operator} associated with $\left( d,\mu ,\sigma \right) $.

\subsection{Subordination}

Using the spectral decomposition (\ref{PEla}) of $P$, we obtain that 
\begin{equation*}
\mathcal{L}=\int_{[0,+\infty )}\log \frac{1}{\sigma \left( 1/\lambda \right) 
}dE_{\lambda }
\end{equation*}%
where $\left\{ E_{\lambda }\right\} $ is the spectral resolution defined by (%
\ref{Ela}). Denote for simplicity 
\begin{equation}
\varphi \left( \lambda \right) :=\log \frac{1}{\sigma \left( 1/\lambda
\right) }  \label{fisi}
\end{equation}%
so that 
\begin{equation}
\mathcal{L}=\int_{[0,+\infty )}\varphi \left( \lambda \right) dE_{\lambda }~.
\label{L=}
\end{equation}%
The domain $\func{dom}_{\mathcal{L}}$ is then given by%
\begin{equation*}
\func{dom}_{\mathcal{L}}=\left\{ f\in L^{2}:\int_{0}^{\infty }\varphi \left(
\lambda \right) ^{2}d\left( E_{\lambda }f,f\right) <\infty \right\} .
\end{equation*}%
Observe that the function $\varphi $ has the following properties that
follow from the assumptions (\ref{sigma}) about $\sigma $: 
\begin{equation}
\left. 
\begin{array}{l}
\varphi :[0,\infty ]\rightarrow \left[ 0,\infty \right] \ \text{\emph{is\ a\
strictly\ monotone\ increasing\ right-continuous\ function}} \\ 
\text{\emph{such that} }\varphi \left( 0\right) =0\ \text{\emph{and}}\
\varphi \left( \infty -\right) =\infty .%
\end{array}%
\right.   \label{fi}
\end{equation}%
Conversely, any function $\varphi $ satisfying (\ref{fi}) determines the
function 
\begin{equation*}
\sigma \left( \lambda \right) =\exp \left( -\varphi \left( 1/\lambda \right)
\right) 
\end{equation*}%
that satisfies (\ref{sigma}). This observation leads us to the following
interesting subordination property of isotropic Laplacians.

\begin{theorem}
\label{subord}Let $\mathcal{L}$ be an isotropic Laplacian associated with $%
\left( d,\mu ,\sigma \right) $. Let $\psi $ be any function satisfying \emph{%
(\ref{fi})}. Then $\psi \left( \mathcal{L}\right) $ is also an isotropic
Laplacian associated with $\left( d,\mu ,\widetilde{\sigma }\right) $ for
some other distance distribution function $\widetilde{\sigma }.$
\end{theorem}

\begin{proof}
It follows from (\ref{L=}) that 
\begin{equation*}
\psi \left( \mathcal{L}\right) =\int_{[0,+\infty )}\psi \circ \varphi \left(
\lambda \right) dE_{\lambda }.
\end{equation*}%
Since the composition $\psi \circ \varphi $ also satisfies (\ref{fi}), we
obtain that $\psi \left( \mathcal{L}\right) $ is an isotropic Laplacian.
Moreover, using (\ref{fisi}), we obtain the following formula for $%
\widetilde{\sigma }$:%
\begin{equation*}
\widetilde{\sigma }\left( r\right) =\exp \left( -\psi \left( \log \frac{1}{%
\sigma \left( r\right) }\right) \right) .
\end{equation*}
\end{proof}

\begin{remark}
\label{Bochner property} \RM Any a non-negative definite, self-adjoint
operator $\mathcal{L}$ in $L^{2}$ generates a semigroup $\left\{ e^{-t%
\mathcal{L}}\right\} _{t\geq 0}$. We refers to $\mathcal{L}$ as a Laplacian
if the semigroup $\left\{ e^{-t\mathcal{L}}\right\} $ is Markovian. In
general, by Bochner's theorem, for any Laplacian $\mathcal{L}$, the operator 
$\psi (\mathcal{L})$ is again a Laplacian, provided $\psi $ is a Bernstein
function (see, for example, Schilling, Song and Vondra{\v{c}}ek~\cite%
{Schilling2012}). It is known that $\psi (\lambda )=\lambda ^{\alpha }$ is a
Bernstein function if and only if $0<\alpha \leq 1.$ Thus, for a general
Laplacian $\mathcal{L}$, the power $\mathcal{L}^{\alpha }$ is guaranteed a
Laplacian only for $\alpha \leq 1$. For example, for the classical Laplace
operator $\mathcal{L}=-\Delta $ in $\mathbb{R}^{n}$, the power $\left(
-\Delta \right) ^{\alpha }$ with $\alpha >1$ is not a Laplacian. In a
striking contrast to that, by Theorem \ref{subord}, the powers $\mathcal{L}%
^{\alpha }$ of the isotropic Laplacian are again Laplacians for all $\alpha
>0$.
\end{remark}

\subsection{The $L^{2}$-spectrum of the Laplacian}

Our next goal in this section is to give an explicit expression for $%
\mathcal{L}f$ and to describe the spectrum of $\mathcal{L}$. Recall that by
Theorem \ref{p-laplace} the triples $\left( d,\mu ,\sigma \right) $ and $%
\left( d_{\ast },\mu ,\sigma _{\ast }\right) $ induce the same Markov
operator $P$ and, hence, the same Laplace operator $\mathcal{L}$, where $%
d_{\ast }$ is the intrinsic ultra-metric defined by (\ref{d*}) and 
\begin{equation*}
\sigma _{\ast }\left( r\right) =\exp \left( -\frac{1}{r}\right)
\end{equation*}%
From now on we will use only the metric $d_{\ast }$ and $\sigma _{\ast }$.
Let the spectral resolution $\left\{ E_{\lambda }\right\} $ be also defined
using the metric $d_{\ast }$, which means that in the definition (\ref{Ela})
of $E_{\lambda }$ we now use the averaging operator $\mathrm{Q}_{r}$ with
respect to the metric $d_{\ast }$. The function $\varphi _{\ast }$
associated with $\sigma _{\ast }$ by (\ref{fisi}) has especially simple
form: $\varphi _{\ast }\left( \lambda \right) =\lambda $. Therefore, we
obtain from (\ref{L=}) the spectral decomposition of $\mathcal{L}$ in the
classical form%
\begin{equation}
\mathcal{L}=\int_{[0,+\infty )}\lambda dE_{\lambda }=\int_{\left( 0,\infty
\right) }\lambda dE_{\lambda }.  \label{LEla}
\end{equation}%
The change $s=\frac{1}{\lambda }$ gives%
\begin{equation*}
\mathcal{L}=-\int_{\left( 0,\infty \right) }\frac{1}{s}d\mathrm{Q}_{s}.
\end{equation*}%
For any $x\in X$, denote by $\Lambda \left( x\right) $ the set of values of $%
d_{\ast }\left( x,y\right) $ for all $y\in X,$ $y\neq x$, that is,%
\begin{equation}
\Lambda \left( x\right) =\left\{ d\left( x,y\right) :y\in X\setminus \left\{
x\right\} \right\} .  \label{Lambda}
\end{equation}

\begin{lemma}
\label{LemLam}The set $\Lambda \left( x\right) $ has no accumulation point
in $\left( 0,\infty \right) $. Consequently, $\Lambda \left( x\right) $ is
at most countable.
\end{lemma}

\begin{proof}
Let $r\in \left( 0,\infty \right) $ be an accumulation point of $\Lambda
\left( x\right) $, that is, there is a sequence $\left\{ r_{k}\right\} $
from $\Lambda \left( x\right) \setminus \left\{ r\right\} $ such that $%
r_{k}\rightarrow r$ as $k\rightarrow \infty $. Then $r_{k}=d_{\ast }\left(
x,y_{k}\right) $ for some $y_{k}\in X$. Since the sequence $\left\{
y_{k}\right\} $ is bounded, by the compactness of all balls in $X$ it has a
convergent subsequence. Without loss of generality, we can then assume that $%
\left\{ y_{k}\right\} $ converges, say to $y\in X$. Then we have $r=d\left(
x,y\right) $. Since $r>0$, we have for large enough $k$ that $r_{k}>r/2$ and 
$d\left( y,y_{k}\right) <r/2$. Then we obtain by the ultra-metric inequality
that%
\begin{equation*}
r_{k}\leq \max \left( r,d\left( y,y_{k}\right) \right) =r
\end{equation*}%
and analogously%
\begin{equation*}
r\leq \max \left( r_{k},d\left( y,y_{k}\right) \right) =r_{k}
\end{equation*}%
whence $r_{k}=r$, which contradicts the assumptions.
\end{proof}

\begin{definition}
\RM For any ball $B$ in $X$ denote by $\rho \left( B\right) $ the minimal $%
d_{\ast }$-radius of $B$.
\end{definition}

Note that $\rho \left( B\right) $ exists because all balls are defined as
closed balls.

\begin{lemma}
If $\rho \left( B\right) >0$ then $\rho \left( B\right) \in \Lambda \left(
x\right) $ for any $x\in B$. Conversely, any number in $\Lambda \left(
x\right) $ is equal to $\rho \left( B\right) $ for some ball $B$ containing $%
x$.
\end{lemma}

\begin{proof}
Set $r=\rho \left( B\right) $ so that $B=B_{r}^{\ast }\left( x\right) $. For
any $y\in B$ we have $d_{\ast }\left( x,y\right) \leq r$, and we have to
show that $d_{\ast }\left( x,y\right) =r$ for some $y$. Assume that $d_{\ast
}\left( x,y\right) <r$ for all $y\in B$. Then the set $\left\{ d_{\ast
}\left( x,y\right) :y\in B\setminus \left\{ x\right\} \right\} $ is a subset
of $\left( 0,r\right) \cap \Lambda \left( x\right) $. By Lemma \ref{LemLam},
the latter set has a maximal element, say $r^{\prime }$. Then $B\subset
B_{r^{\prime }}^{\ast }\left( x\right) $, which contradicts the minimality
of radius $r$. Conversely, if $r\in \Lambda \left( x\right) $ then the ball $%
B=B_{r}\left( x\right) $ has $\rho \left( B\right) =r$ since there exists $%
y\in X$ with $d\left( x,y\right) =r$.
\end{proof}

\begin{definition}
\label{predecessor}\RM Let $B,C$ be two balls in $X$ such that $C\subset B$.
We say that $C$ is a \emph{child} or \emph{successor} of $B$ (and $B$ is a 
\emph{parent} or \emph{predecessor} of $C$) if $C\neq B$ and, for any ball $%
A $, such that $C\subset A\subset B$ we have $A=C$ or $A=B$. In other words, 
$B $ is a minimal ball containing $C$ as a proper subset. If $C$ is a child
of $B$ then we write $C\prec B.$
\end{definition}

Denote by $\mathcal{K}$ be the family of all balls $C$ in $X$ with positive
radii. If $C=B_{r}^{\ast }\left( x\right) $ is a ball from $\mathcal{K}$
with $r>0$ then for the minimal radius $\rho \left( C\right) $ we have two
possibilities:

\begin{enumerate}
\item either $\rho \left( C\right) >0$,

\item or $\rho \left( C\right) =0$ and the center of $C$ is an isolated
point of $X$.
\end{enumerate}

\begin{lemma}
\label{Lemchildren}For any ball $C\in \mathcal{K}$ such that $C\neq X$ there
is a unique parent ball $B$. For any ball $B$ with $\rho \left( B\right) >0$
the number $\deg \left( B\right) $ of its children satisfies $2\leq \deg
\left( B\right) <\infty .$ Moreover, all the children of $B$ are disjoint
and their union is equal to $B$.
\end{lemma}

\begin{proof}
Fix some $x\in C$. It follows from Lemma \ref{LemLam} and the definition of $%
\mathcal{K}$ that the set $\left( \rho \left( C\right) ,\infty \right) \cap
\Lambda \left( x\right) $ has a minimum that we denote by $r$. Then the ball 
$B_{r}^{\ast }\left( x\right) $ is a parent of $C$. The uniqueness of the
parent follows from definition.

If $C_{1}$ and $C_{2}$ are two distinct children of $B$ then $C_{1}$ and $%
C_{2}$ are disjoint. Indeed, if they intersect then one of them contains the
other, say $C_{1}\subset C_{2}$. By definition of a parent/child, we must
have then $C_{2}=C_{1}$ or $C_{2}=B$, whence $C_{1}=C_{2}$ follows.

Let us show that for any $x\in B$ there is a ball $C$ such that $x\in C\prec
B$. Indeed, if the set $\left( 0,\rho \left( B\right) \right) \cap \Lambda
\left( x\right) $ is empty, then $C=B_{0}^{\ast }\left( x\right) =\left\{
x\right\} $ is the child of $B$. If the set $\left( 0,\rho \left( B\right)
\right) \cap \Lambda \left( x\right) $ is non-empty then by Lemma \ref%
{LemLam} is has a maximum, say $r$. Then $C=B_{r}^{\ast }\left( x\right) $
is a child of $B$. Hence, the set of all children of $B$ is a covering of $B$%
.

Each child $C$ of $B$ is an open set (being also a closed ball) because $C$
coincides with an open ball of radius $\rho \left( B\right) $. Since $B$ is
compact, it follows that the set of its children is finite, that is, $\deg
\left( B\right) <\infty $. Finally, $\deg \left( B\right) $ cannot be equal
to $1$ since then $B$ would coincide with its only child. Hence, $\deg
\left( B\right) \geq 2$.
\end{proof}

\smallskip

For any $C\in \mathcal{K}$ define the function $f_{C}$ on $X$ as follows. If 
$C $ is a proper subset of $X$ then, denoting by $B$ the parent of $C$, set%
\begin{equation}
f_{C}=\frac{1}{\mu (C)}\mathbf{1}_{C}-\frac{1}{\mu (B)}\mathbf{1}_{B}
\label{fC}
\end{equation}%
(note that always $\mu \left( C\right) >0$). Set also $\lambda \left(
C\right) :=1/\rho \left( B\right) $. If $C=X$ (which can only be the case
when $X$ is compact), then set $f_{C}\equiv 1$ and $\lambda \left( C\right)
=0.$

\begin{theorem}
\label{Eigenvalues-thm} {For any }$C\in \mathcal{K}$ the function $f_{C}$ is
an eigenfunction of $\mathcal{L}$ with the eigenvalue $\lambda \left(
C\right) $. The family $\{f_{C}:C\in \mathcal{K}\}$ is complete (its linear
span is dense) in $L^{2}\left( X,\mu \right) .$ Consequently, the operator $%
\mathcal{L}$ has a complete system of compactly supported eigenfunctions.
\end{theorem}

\begin{proof}
Fix a ball $C\in \mathcal{K}$ or radius $r=\rho \left( C\right) $, and let $%
B $ be the\ parent of radius $r^{\prime }=\rho \left( B\right) $. Any ball
of radius $s<r^{\prime }$ either is disjoint with $C$ or is contained in $C$%
, which implies that $\mathbf{1}_{C}$ is constant in any such ball. It
follows that, for any $s<r^{\prime }$, we have $\mathrm{Q}_{s}\mathbf{1}_{C}=%
\mathbf{1}_{C}$ and, similarly $\mathrm{Q}_{s}\mathbf{1}_{B}=\mathbf{1}_{B}$%
, whence%
\begin{equation*}
\mathrm{Q}_{s}f_{C}=f_{C}.
\end{equation*}%
For $s\geq r^{\prime }$ any ball of radius $s$ either contains both balls $%
C,B$ or is disjoint from $B$. Since the averages of the two functions $\frac{%
1}{\mu (C)}\mathbf{1}_{C}$ and $\frac{1}{\mu (B)}\mathbf{1}_{B}$ over any
ball containing $C$ and $B$ are equal, we obtain that in this case $\mathrm{Q%
}_{s}f_{C}=0$. It follows that%
\begin{equation*}
\mathcal{L}f_{C}=-\int_{\left( 0,\infty \right) }\frac{1}{s}\mathrm{Q}%
_{s}f_{C}~ds=\frac{1}{r^{\prime }}\,f_{C}=\lambda \left( C\right) f_{C},
\end{equation*}%
which proves that $f_{C}$ is an eigenfunction of $\mathcal{L}$ with the
eigenvalue $\lambda \left( C\right) $. In the case of compact $X$ we have $%
\mathrm{Q}_{s}f_{X}=f_{X}$ for all $s>0$, whence $\mathcal{L}f_{X}=0=\lambda
\left( X\right) $.

Let us show that the system $\left\{ f_{C}:C\in \mathcal{K}\right\} $ is
complete. We assume that some function $f\in L^{2}$ is orthogonal to all
functions $f_{C}\,$and prove that $f\equiv \func{const}.$ We have for any $%
r>0$, 
\begin{equation*}
\left( \mathrm{Q}_{r}f,f_{C}\right) _{L^{2}}=\left( f,\mathrm{Q}%
_{r}f_{C}\right) _{L^{2}}=\func{const}\left( f,f_{C}\right) _{L^{2}}=0,
\end{equation*}%
where we have used the fact that any eigenfunction of $\mathcal{L}$ is also
eigenfunction of $\mathrm{Q}_{r}$ with an eigenvalue that we denoted by $%
\func{const}$. Hence, $\mathrm{Q}_{r}f$ is also orthogonal to all $f_{C}$.
We will prove below that $\mathrm{Q}_{r}f=0$, which will imply by (\ref{sr0}%
) that $f=0$.

Since $\mathrm{Q}_{r}f$ is constant in any ball of radius $r$, by renaming $%
\mathrm{Q}_{r}f$ back to $f$ we can assume from now on that $f$ is constant
in any ball of radius $r$. Fix some ball $C\in \mathcal{K}$ and its parent $%
B $. It follows from (\ref{fC}) that $\left( f,f_{C}\right) _{L^{2}}=0$ is
equivalent to%
\begin{equation*}
\frac{1}{\mu (C)}\int_{C}f\,d\mu =\frac{1}{\mu (B)}\int_{B}f\,d\mu ,
\end{equation*}%
that is, the average value of $f$ over a ball is preserved when switching to
its parent. Starting with two balls $C_{1}$ and $C_{2}$ of radii $r$, we can
build a sequence of their predecessors which end up with the same (large
enough) ball. This implies that the averages of $f$ in $C_{1}$ and $C_{2}$
are the same. Since $f$ is constant in $C_{1}$ and $C_{2}$, it follows that
the values of these constants are the same. It follows that $f\equiv \func{%
const}$ on $X$. If $\mu \left( X\right) =\infty $ then we obtain $f\equiv 0$%
. If $\mu \left( X\right) <\infty $ then using the orthogonality of $f$ to $%
f_{X}\equiv 1$ we obtain again that $f\equiv 0.$
\end{proof}

For any ball $B$ with $\rho \left( B\right) >0$ define the subspace $%
\mathcal{H}_{B}$ of $L^{2}$ as follows: 
\begin{equation}
\mathcal{H}_{B}=\limfunc{span}\{f_{C}:C\prec B\}.  \label{H-span}
\end{equation}%
By Theorem \ref{Eigenvalues-thm}, all non-zero functions in $\mathcal{H}_{B}$
are the eigenfunctions of $\mathcal{L}$ with eigenvalue $\frac{1}{\rho
\left( B\right) }$.

It follows from Lemma \ref{Lemchildren} that the functions $\left\{ \mathbf{1%
}_{C}:C\prec B\right\} $ are linearly independent and%
\begin{equation*}
\sum_{C\prec B}\mathbf{1}_{C}=\mathbf{1}_{B}.
\end{equation*}%
This entails 
\begin{equation}
\sum_{C\prec B}\mu (C)f_{C}=0  \label{Eigen-functions relation}
\end{equation}%
and that this is the only dependence between functions $f_{C}$. Hence, we
obtain that%
\begin{equation}
\dim \mathcal{H}_{B}=\deg \left( B\right) -1.  \label{HB}
\end{equation}%
Clearly, the spaces $\mathcal{H}_{B}$ and $\mathcal{H}_{B^{\prime }}$ are
orthogonal provided the balls $B,B^{\prime }$ are disjoint.

Define the set%
\begin{equation}
\Lambda :=\left\{ d_{\ast }\left( x,y\right) :x,y\in X,\ x\neq y\right\}
=\tbigcup_{x\in X}\Lambda \left( x\right) .  \label{Lambdastar}
\end{equation}
Theorem \ref{Eigenvalues-thm} implies the following.

\begin{corollary}
\label{L^2-Spectrum} The spectrum $\func{spec}\mathcal{L}$ of the Laplacian $%
\mathcal{L}$ is pure point and 
\begin{equation*}
\func{spec}\mathcal{L}=\overline{\left\{ \frac{1}{r}:r\in \Lambda \,\right\} 
}\cup \left\{ 0\right\} .
\end{equation*}%
\\[4pt]
The space $L^{2}(X,\mu )$ decomposes into an orthogonal sum of
finite-dimensional (subspaces of) eigenspaces as follows: if $\mu \left(
X\right) =\infty $ then 
\begin{equation*}
L^{2}(X,\mu )=\bigoplus_{\rho \left( B\right) >0}\mathcal{H}_{B},
\end{equation*}%
and if $\mu \left( X\right) <\infty $ then 
\begin{equation*}
L^{2}(X,\mu )=\left\{ \func{const}\right\} \oplus \bigoplus_{\rho \left(
B\right) >0}\mathcal{H}_{B}~.
\end{equation*}
\end{corollary}

\begin{example}
\RM Let $\left( X,d,\mu \right) $ and be as in Example \ref{p-adic-doubling}%
, that is, $X=\mathbb{Q}_{p}$, $d\left( x,y\right) =\left\Vert
x-y\right\Vert _{p}$ is the $p$-adic distance and $\mu $ be the Haar
measure. Set for some $\alpha >0$%
\begin{equation*}
\sigma \left( r\right) =\exp \left( -\left( \frac{p}{r}\right) ^{\alpha
}\right) ,
\end{equation*}%
so that by (\ref{d*p})%
\begin{equation*}
d_{\ast }\left( x,y\right) =\left( \frac{\left\Vert x-y\right\Vert _{p}}{p}%
\right) ^{\alpha }.
\end{equation*}%
Since the set of non-zero values of $\left\Vert x-y\right\Vert _{p}$ is $%
\left\{ p^{k}\right\} _{k\in \mathbb{Z}}$, it follows that the set $\Lambda $
of all non-zero values of $d_{\ast }\left( x,y\right) $ is 
\begin{equation*}
\Lambda =\left\{ p^{\alpha k}:k\in \mathbb{Z}\right\} \ .
\end{equation*}%
Hence, 
\begin{equation*}
\func{spec}\mathcal{L}=\left\{ p^{\alpha k}:k\in \mathbb{Z}\right\} \cup
\left\{ 0\right\} .
\end{equation*}
\end{example}

\begin{corollary}
\label{generality of the spectrum} Let $(X,d)$ be a non-compact, proper
ultra-metric space. Let $M\subset \lbrack 0,\infty )$ be any closed set
(unbounded, if $X$ contains at least one non-isolated point) that
accumulates at $0$. Then the following is true.

$\left( a\right) $ There exists a proper ultra-metric $d^{\prime }$ on $X$
that generates the same topology as $d$ and the isotropic Laplacian $%
\mathcal{L}^{\prime }$ of the triple $\left( d^{\prime },\mu ,\sigma _{\ast
}\right) $ has the spectrum $\func{spec}\mathcal{L}^{\prime }=M$.{\ }

$\left( b\right) $ Suppose in addition that there exists a partition of $X$
into $d$-balls that consists of infinitely many non-singletons. Then the
ultra-metric $d^{\prime }$ of part $\left( a\right) $ can be chosen so that
the collections of $d$-balls and $d^{\prime }$-balls coincide.
\end{corollary}

\begin{proof}
The set 
\begin{equation*}
D=\{x\in (0,\infty ):x^{-1}\in M\}\cup \{0\}
\end{equation*}%
is a closed, unbounded subset of $[0\,,\,\infty )$ containing $0.$ The the
statement $\left( a\right) $ is equivalent to the existence of a proper
ultra-metric $d^{\prime }$ on $X$ that generates the same topology as $d$
and such that the closure of the value set $\left\{ d^{\prime }\left(
x,y\right) \right\} _{x,y\in X}$ of that metric coincides with $D$. This
metric property is proved by Bendikov and Krupski~\cite[\S 2]{BenKru}. Given 
$\mu $, the Laplacian associated with the triple $(d^{\prime },\mu ,\sigma
_{\ast })$ has the required property by Corollary \ref{L^2-Spectrum}. The
proof of $\left( b\right) $ follows in the same way from a result of \cite[%
\S 2]{BenKru}.
\end{proof}

\subsection{The Dirichlet form and jump kernel}

Let us construct a Dirichlet form $\left( \mathcal{E},\func{dom}_{\mathcal{E}%
}\right) $ associated with the isotropic semigroup $\left\{ P^{t}\right\} $.
It is well known that if $P^{t}1=1$, which is the case here, then%
\begin{equation*}
\mathcal{E}\left( f,f\right) =\lim_{t\rightarrow 0}\frac{1}{2t}%
\int_{X}\int_{X}p_{t}\left( x,y\right) \left( f\left( x\right) -f\left(
y\right) \right) ^{2} \,d\mu \left( x\right) \,d\mu \left( y\right)
\end{equation*}%
and 
\begin{equation*}
\func{dom}_{\mathcal{E}}=\left\{ f\in L^{2}:\mathcal{E}\left( f,f\right)
<\infty \right\}
\end{equation*}%
(see \cite{FOT}). Using the identity (\ref{p=Nt}), we obtain that 
\begin{equation*}
\frac{p\left( t,x,y\right) }{t}\nearrow \int_{0}^{1/d_{\ast }(x,y)}N(x,\tau
)\,d\tau \ \ \text{as }t\searrow 0.
\end{equation*}%
Setting 
\begin{equation}
J(x,y):=\int_{0}^{1/d_{\ast }(x,y)}N(x,\tau )\,d\tau =\int_{d_{\ast }\left(
x,y\right) }^{\infty }\frac{1}{V\left( x,s\right) }\,\frac{ds}{s^{2}},
\label{Jxy}
\end{equation}%
we obtain by the monotone convergence theorem that, for all $f\in L^{2}$,%
\begin{equation*}
\mathcal{E}\left( f,f\right) =\frac{1}{2}\int_{X}\int_{X}\left( f\left(
x\right) -f\left( y\right) \right) ^{2}J\left( x,y\right) \,d\mu \left(
x\right)\, d\mu \left( y\right) .
\end{equation*}%
Note that $0<J\left( x,y\right) =J\left( y,x\right) <\infty $ for all $x\neq
y$, while $J\left( x,x\right) =\infty $.

The polarization identity implies then, for all $f,g\in \func{dom}_{\mathcal{%
E}}$ that%
\begin{equation}
\mathcal{E}\left( f,g\right) =\frac{1}{2}\int_{X}\int_{X}\left( f\left(
x\right) -f\left( y\right) \right) \left( g\left( x\right) -g\left( y\right)
\right) J\left( x,y\right) d\mu \left( x\right) d\mu \left( y\right) .
\label{Edef}
\end{equation}%
The function $J$ is called the \emph{jump kernel} of the Dirichlet form $%
\mathcal{E}$. We show here that it can be used also to describe the
generator $\mathcal{L}$ of $\left\{ P^{t}\right\} $. Recall that by the
theory of Dirichlet forms, the generator $\mathcal{L}$ has the following
equivalent definition: it is the self-adjoint operator in $L^{2}$ with $%
\func{dom}_{\mathcal{L}}\subset \func{dom}_{\mathcal{E}}$ such that 
\begin{equation*}
\left( \mathcal{L}f,g\right) =\mathcal{E}\left( f,g\right) \ 
\end{equation*}%
for all $f\in \func{dom}_{\mathcal{L}}$ and $g\in \func{dom}_{\mathcal{E}}$.

Denote by $\mathcal{V}_{r}$ the image of the operator $\mathrm{Q}_{r}$
(defined with respect to $d_{\ast }$), that is, the space of all $L^{2}$%
-functions that are constant on each ball of radius $r$. Set also 
\begin{equation*}
\mathcal{V}:=\tbigcup_{r>0}\mathcal{V}_{r}
\end{equation*}%
and observe that $\mathcal{V}$ is a linear subspace of $L^{2}$. Observe also
that the space $\mathcal{V}_{c}$ of all locally constant functions with
compact support is contained in $\mathcal{V}$.

\begin{theorem}
\label{Thm-Dirichlet form/Laplacian}The space $\mathcal{V}$ is dense in $%
L^{2}$, it is a subset of $\func{dom}_{\mathcal{L}}$ and, for any $f\in 
\mathcal{V}$,%
\begin{equation}
\mathcal{L}f\left( x\right) =\int_{X}\left( f\left( x\right) -f\left(
y\right) \right) J\left( x,y\right) d\mu \left( y\right) .  \label{Ldef}
\end{equation}
\end{theorem}

\begin{proof}
That $\mathcal{V}$ is dense in $L^{2}$ follows from (\ref{sr0}). In fact, $%
\mathcal{V}_{c}$ is also dense in $L^{2}$, which follows from the fact that
all the eigenfunctions of $\mathcal{L}$ lie in $\mathcal{V}_{c}$.

By (\ref{Ela}) and (\ref{LEla}) we have $\mathrm{Q}_{r}=\mathbf{1}%
_{[0,1/r)}\left( \mathcal{L}\right) $. Therefore, $\mathcal{L}\mathrm{Q}_{r}$
is a bounded operator, which implies that $\func{dom}_{\mathcal{L}}\supset 
\mathcal{V}_{r}$ and, hence, $\func{dom}_{\mathcal{L}}\supset \mathcal{V}.$

Fix a function $f\in \mathcal{V}_{r}$ with $r>0$, set%
\begin{equation*}
u\left( x\right) =\int_{X}\left\vert f\left( x\right) -f\left( y\right)
\right\vert J\left( x,y\right) \, d\mu \left( y\right) .
\end{equation*}%
We show that $u\in L^{2}$. Observe that $f\left( x\right) =f\left( y\right) $
whenever $d_{\ast }\left( x,y\right) \leq r$. Hence, we can restrict the
integration to the domain $\left\{ d_{\ast }\left( x,y\right) >r\right\} $.
We have by the Cauchy-Schwarz inequality%
\begin{equation}
u^{2}\left( x\right) \leq \left( \int_{X}\left\vert f\left( x\right)
-f\left( y\right) \right\vert ^{2}J\left( x,y\right)\, d\mu \left( y\right)
\right) \left( \int_{\left\{ y:d_{\ast }\left( x,y\right) >r\right\}
}J\left( x,y\right) \, d\mu \left( y\right) \right) .  \label{u2}
\end{equation}%
Let us show that%
\begin{equation*}
\int_{\left\{ y:d_{\ast }\left( x,y\right) >r\right\} }J\left( x,y\right)
d\mu \left( y\right) \leq \frac{1}{r}.
\end{equation*}%
Indeed, by (\ref{Jxy}) and Fubini's theorem, the latter integral is equal to 
\begin{eqnarray*}
\int_{\left\{ y:d_{\ast }\left( x,y\right) >r\right\} }\int_{\left\{ s:s\geq
d_{\ast }\left( x,y\right) \right\} }^{\infty }\frac{1}{V\left( x,s\right) }%
\,\frac{ds}{s^{2}}\,d\mu \left( y\right) &=&\int_{r}^{\infty }\frac{ds}{%
s^{2}V\left( x,s\right) }\int_{\left\{ y:r<d_{\ast }\left( x,y\right) \leq
s\right\} }d\mu \left( y\right) \\
&=&\int_{r}^{\infty }\frac{V\left( x,s\right) -V\left( x,r\right) }{%
s^{2}V\left( x,s\right) }\,ds \\
&\leq &\int_{r}^{\infty }\frac{ds}{s^{2}}=\frac{1}{r}.
\end{eqnarray*}%
It follows from (\ref{u2}) that%
\begin{equation*}
\int_{X}u^{2}\,d\mu \leq \frac{1}{r}\,\mathcal{E}\left( f,f\right) .
\end{equation*}%
Since $f\in \func{dom}_{\mathcal{L}}\subset \func{dom}_{\mathcal{E}}$, we
obtain that $u\in L^{2}.$ In particular, $u\left( x\right) <\infty $ for
almost all $x\in X$. Consequently, for almost all $x\in X$, the function 
\begin{equation*}
y\mapsto \left( f(x)-f(y)\right) J(x,y)
\end{equation*}%
is in $L^{1}$, and its integral 
\begin{equation*}
v\left( x\right) =\int_{X}\left( f(x)-f(y)\right) J(x,y)\,d\mu (y)
\end{equation*}%
is an $L^{2}$ function. We need to verify that $\mathcal{L}f=v$. For that
purpose it suffices to verify that, for any $g\in \func{dom}_{\mathcal{E}}$,%
\begin{equation*}
\left( v,g\right) _{L^{2}}=\mathcal{E}\left( f,g\right) .
\end{equation*}%
Indeed, using Fubini's theorem, we obtain%
\begin{eqnarray*}
\left( v,g\right) _{L^{2}} &=&\int_{X}\int_{X}\left( f(x)-f(y)\right)
g\left( x\right) J(x,y)\,d\mu (y)\,d\mu \left( x\right) \\
&=&\int_{X}\int_{X}\left( f(y)-f(x)\right) g\left( y\right) J(y,x)\,d\mu
(x)\, d\mu \left( y\right) \\
&=&\frac{1}{2}\int_{X}\int_{X}\left( f\left( x\right) -f\left( y\right)
\right) \left( g\left( x\right) -g\left( y\right) \right) J\left( x,y\right)
\,d\mu \left( x\right)\, d\mu \left( y\right) \\
&=&\mathcal{E}\left( f,g\right) ,
\end{eqnarray*}%
which was to be proved.
\end{proof}

\subsection{The $L^{p}$-spectrum of the Laplacian}

It is known that any continuous symmetric Markov semigroup can be extended
to all spaces $L^{p},$ $1\leq p<\infty ,$ as a continuous contraction
semigroup. In particular, this is true for the semigroup $\left\{
P^{t}\right\} .$ We use the same notation for the extended semigroup, while
we denote by $\mathcal{L}_{p}$ its infinitesimal generator and by $\func{dom}%
_{\mathcal{L}_{p}}$ its domain in $L^{p}$.

\begin{theorem}
\label{Lp-Lq thm} For all $1\leq p<\infty $ we have%
\begin{equation*}
\func{spec}\mathcal{L}_{p}=\func{spec}\mathcal{L}_{2}\,.
\end{equation*}
\end{theorem}

\begin{proof}
Since by Theorem \ref{Eigenvalues-thm} all the eigenfunctions of $\mathcal{L}%
_{2}$ are compactly supported, they belong also to $L^{p}$, which implies
that 
\begin{equation*}
\func{spec}\mathcal{L}_{2}\subset \func{spec}\mathcal{L}_{p}.
\end{equation*}%
To prove the opposite inclusion, we choose $\lambda _{0}\notin \func{spec}%
\mathcal{L}_{2}$ and show that $\lambda _{0}\notin \func{spec}\mathcal{L}%
_{p} $. For that purpose it suffices to show that the resolvent operator 
\begin{equation*}
R:=\left( \mathcal{L}_{2}-\lambda _{0}\func{id}\right) ^{-1}
\end{equation*}%
being a bounded operator in $L^{2}$, extends to a bounded operator in $L^{p}$%
. The latter amounts to showing that, for any functions $f\in L^{2}\cap
L^{p} $ and $g\in L^{2}\cap L^{q}$, where $q=\frac{p}{p-1}$ is the H\"{o}%
lder conjugate of $p$, the following inequality holds:%
\begin{equation*}
\left\vert \left( Rf,g\right) _{L^{2}}\right\vert \leq C\left\Vert
f\right\Vert _{L^{p}}\left\Vert g\right\Vert _{L^{q}}
\end{equation*}%
with a constant $C$ that does not depend on $f,g$.

Let us restrict to the case $\lambda _{0}>0$ (the case when $\lambda _{0}<0$
is simpler). Choose $a,b>0$ such that $a<\lambda _{0}<b$ and $\left[ a,b%
\right] $ is disjoint from $\func{spec}\mathcal{L}_{2}$. Using the spectral
decomposition (\ref{LEla}), we obtain%
\begin{equation*}
R=\int_{\func{spec}\mathcal{L}_{2}}\frac{dE_{\lambda }}{\lambda -\lambda _{0}%
}=\int_{[0,a)}\frac{dE_{\lambda }}{\lambda -\lambda _{0}}+\int_{[b,\infty )}%
\frac{dE_{\lambda }}{\lambda -\lambda _{0}},
\end{equation*}%
whence%
\begin{equation*}
\left( Rf,g\right) =\int_{[0,a)}\frac{d\left( E_{\lambda }f,g\right) }{%
\lambda -\lambda _{0}}+\int_{[b,\infty )}\frac{d\left( E_{\lambda
}f,g\right) }{\lambda -\lambda _{0}}.
\end{equation*}%
Integration by parts gives%
\begin{eqnarray*}
\left( Rf,g\right) &=&\frac{\left( E_{a}f,g\right) }{a-\lambda _{0}}%
+\int_{[0,a)}\frac{\left( E_{\lambda }f,g\right) }{\left( \lambda -\lambda
_{0}\right) ^{2}}\,d\lambda \\
&&-\frac{\left( E_{b}f,g\right) }{b-\lambda _{0}}+\int_{[b,\infty )}\frac{%
\left( E_{\lambda }f,g\right) }{\left( \lambda -\lambda _{0}\right) ^{2}}%
\,d\lambda .
\end{eqnarray*}%
Since $E_{\lambda }=\mathrm{Q}_{1/\lambda }$ is a Markov operator, it
standardly extends to a bounded operator in $L^{p}$ with the norm bound $1$,
so that%
\begin{equation*}
\left\vert \left( E_{\lambda }f,g\right) \right\vert \leq \left\Vert
f\right\Vert _{L^{p}}\left\Vert g\right\Vert _{L^{q}}.
\end{equation*}%
It follows that%
\begin{equation*}
\left\vert \left( Rf,g\right) \right\vert \leq \left\Vert f\right\Vert
_{L^{p}}\left\Vert g\right\Vert _{L^{q}}\left( \frac{1}{\lambda _{0}-a}+%
\frac{1}{b-\lambda _{0}}+\int_{[0,a)\cup \lbrack b,\infty )}\frac{d\lambda }{%
\left( \lambda -\lambda _{0}\right) ^{2}}\right) ,
\end{equation*}%
which finishes the proof since the quantity in the large parentheses is
finite.
\end{proof}

\smallskip

The last theorem of this section concerns a Liouville property. Note that
the semigroup $\left\{ P^{t}\right\} $ defined by (\ref{convexcomb}) acts on
the space $\mathcal{B}_{b}$ of bounded Borel functions as a contraction
semigroup, but it is not continuous unless $X$ is discrete. Define
convergence of sequence in $\mathcal{B}_{b}$ as a bounded pointwise
convergence, that is, a sequence $\left\{ f_{k}\right\} \subset \mathcal{B}%
_{b}$ converges in $\mathcal{B}_{b}$ to a function $f$ if all sequence $%
\left\{ f_{k}\right\} $ is uniformly bounded and $f_{k}\left( x\right)
\rightarrow f\left( x\right) $ as $k\rightarrow \infty $ for all $x\in X$.
Define a weak infinitesimal generator $\mathcal{L}_{\infty }$ of the
semigroup $\left\{ P^{t}\right\} $ in $\mathcal{B}_{b}$ as follows: the
domain $\func{dom}_{\mathcal{L}_{\infty }}\ $consists of functions $f\in 
\mathcal{B}_{b}$ such that the limit 
\begin{equation*}
\mathcal{L}_{\infty }f:=\lim_{t\rightarrow 0}\frac{f-P^{t}f}{t}
\end{equation*}%
exists in the sense of convergence in $\mathcal{B}_{b}$. This yields $%
\mathcal{L}_{\infty }f\in \mathcal{B}_{b}$ for any $f\in \func{dom}_{%
\mathcal{L}_{\infty }}$.

\begin{theorem}[Strong Liouville property]
\label{Liouville} Any Borel function $f:X\rightarrow \lbrack 0\,,\,\infty )$
that satisfies $Pf=f$ must be constant.

Consequently, $0$ is an eigenvalue of $\mathcal{L}_{\infty }$ of
multiplicity $1$.
\end{theorem}

\begin{proof}
Since $P$ and $\mathrm{Q}_{r}$ commute, we obtain from $f=Pf$ and 
\begin{equation}
Pf=\int_{0}^{\infty }\mathrm{Q}_{s}f\,d\sigma _{\ast }(s),
\label{Harmonic-1}
\end{equation}%
that, for all $r\geq 0$,%
\begin{equation*}
\mathrm{Q}_{r}f=P\mathrm{Q}_{r}f=\int_{[0,\infty )}\mathrm{Q}_{s}\mathrm{Q}%
_{r}f\,d\sigma _{\ast }\left( s\right) .
\end{equation*}%
Observing that 
\begin{equation*}
\mathrm{Q}_{s}\mathrm{Q}_{r}=\mathrm{Q}_{\max \left( r,s\right) },
\end{equation*}
we obtain%
\begin{equation*}
\mathrm{Q}_{r}f=\int_{[0,r)}\mathrm{Q}_{r}f~d\sigma _{\ast }\left( s\right)
+\int_{[r,\infty )}\mathrm{Q}_{s}f\,d\sigma _{\ast }\left( s\right) .
\end{equation*}%
The first integral here is equal to $\sigma _{\ast }\left( r\right) \mathrm{Q%
}_{r}f$, which implies%
\begin{equation}
\left( 1-\sigma _{\ast }\left( r\right) \right) \mathrm{Q}%
_{r}f=\int_{[r,\infty )}\mathrm{Q}_{s}f\,d\sigma _{\ast }\left( s\right) .
\label{1-s}
\end{equation}%
Fix some $x\in X$. By Lemma \ref{LemLam}, the set $\Lambda \left( x\right) $
of all values $d_{\ast }\left( x,y\right) $ for $y\neq x$ has no
accumulation point in $(0,+\infty )$. Choose $r_{0}$ as follows: if $\Lambda
\left( x\right) $ does not accumulate to $0$, then $r_{0}=0$, and if $%
\Lambda \left( x\right) $ accumulates at $0$ then $r_{0}$ is any value from $%
\Lambda \left( x\right) $. In the both cases the set $\Lambda \left(
x\right) \cap (r,\infty )$ consists of a (finite or infinite) sequence $%
r_{1}<r_{2}<...$ that converges to $\infty $ in the case when it is
infinite. Applying (\ref{1-s}) to $r=r_{k}$ and $r=r_{k+1}$ instead of $r$,
where $k\geq 0$, we obtain%
\begin{eqnarray*}
\left( 1-\sigma _{\ast }\left( r_{k}\right) \right) \mathrm{Q}%
_{r_{k}}f\left( x\right) -\left( 1-\sigma _{\ast }\left( r_{k+1}\right)
\right) \mathrm{Q}_{r_{k+1}}f\left( x\right) &=&\int_{[r_{k},r_{k+1})}%
\mathrm{Q}_{s}f\left( x\right) \,d\sigma _{\ast }\left( s\right) \\
&=&\mathrm{Q}_{r_{k}}f\left( x\right) \left( \sigma _{\ast }\left(
r_{k+1}\right) -\sigma _{\ast }\left( r_{k}\right) \right) ,
\end{eqnarray*}%
whence it follows that%
\begin{equation*}
\left( 1-\sigma _{\ast }\left( r_{k+1}\right) \right) \mathrm{Q}%
_{r_{k}}f\left( x\right) =\left( 1-\sigma _{\ast }\left( r_{k+1}\right)
\right) \mathrm{Q}_{r_{k+1}}f\left( x\right)
\end{equation*}%
and, hence,%
\begin{equation*}
\mathrm{Q}_{r_{k}}f\left( x\right) =\mathrm{Q}_{r_{k+1}}f\left( x\right) .
\end{equation*}%
Consequently, we obtain that 
\begin{equation*}
\mathrm{Q}_{r_{k}}f\left( x\right) =\mathrm{Q}_{r_{0}}f\left( x\right) \ \ 
\text{for all }k\geq 1.
\end{equation*}%
Since $r_{0}$ can be chosen arbitrarily close to $0$, we obtain that $%
\mathrm{Q}_{r}f\left( x\right) $ does not depend on $r$. For any two points $%
x,y\in X$, we have $\mathrm{Q}_{r}f\left( x\right) =\mathrm{Q}_{r}f\left(
y\right) $ for $r\geq d_{\ast }\left( x,y\right) $. Therefore, the function $%
\mathrm{Q}_{r}f\left( x\right) $ is constant both in $r$ and $x$. It follows
from (\ref{Harmonic-1}) that $f=Pf$ is also a constant.

For the second statement of the theorem, $0$ is an eigenvalue of $\mathcal{L}%
_{\infty }$ because $\mathcal{L}_{\infty }1=0$. Assume that $\mathcal{L}%
_{\infty }f=0$ and prove that $f\equiv \func{const}$, which will imply that
the multiplicity of $0$ is $1$. By assumption we have $f\in \mathcal{B}_{b}$
and%
\begin{equation*}
\frac{f-P^{t}f}{t}\overset{\mathcal{B}_{b}}{\longrightarrow }0\ \text{as}\
t\rightarrow 0.
\end{equation*}%
Since the family $\left\{ \frac{f-P_{t}f}{t}\right\} _{t>0}$ is uniformly
bounded, we obtain by the dominated convergence theorem that, for any $r\geq
0$, 
\begin{equation*}
\mathrm{Q}_{r}\left( \frac{f-P^{t}f}{t}\right) \overset{\mathcal{B}_{b}}{%
\longrightarrow }0\ \text{as }t\rightarrow 0,
\end{equation*}%
which in turn implies that, for all $s\geq 0$,%
\begin{equation*}
\frac{P^{s}f-P^{s+t}f}{t}=P^{s}\left( \frac{f-P^{t}f}{t}\right) \overset{%
\mathcal{B}_{b}}{\longrightarrow }0\ \text{as }t\rightarrow 0.
\end{equation*}%
It follows that, for any $x\in X$, the function $s\mapsto P^{s}f\left(
x\right) $ has derivative $0$ and, hence, is constant. It follows that $f=Pf$%
, and by the first statement of the theorem, we conclude that $f=\func{const}%
.$
\end{proof}

\section{Moments of the Markov process}

\label{moments}\setcounter{equation}{0}Let $\{\mathcal{X}_{t}\}$ be the
Markov process associated with the semigroup $\{P_{t}\}.$ For any $\gamma
>0, $ the moment of order $\gamma $ of the process is defined as%
\begin{equation*}
M_{\gamma }(x,t)=\mathbb{E}_{x}\left( d_{\ast }(x,\mathcal{X}_{t})^{\gamma
}\right) ,
\end{equation*}%
where $\mathbb{E}_{x}$ is expectation with respect to the probability
measure on the trajectory space of $\{\mathcal{X}_{t}\}$ that governs the
process starting at $x$. In terms of the heat kernel $p(t,x,y)$ the moment
is given by%
\begin{equation}
M_{\gamma }(x,t)=\int_{X}d_{\ast }(x,y)^{\gamma }p(t,x,y)\,d\mu (y).
\label{M-p-identity}
\end{equation}%
The aim of this section is to estimate $M_{\gamma }(x,t)$ as a function of $%
t $ and $\gamma .$

Let us start with two lemmas. We use the intrinsic volume function (\ref%
{volume}), that is%
\begin{equation*}
V\left( x,r\right) =\mu \left( B_{r}^{\ast }\left( x\right) \right)
\end{equation*}%
and its average moment function of order $\gamma $, that is 
\begin{equation*}
R_{\gamma }(x,\tau )=\frac{1}{V(x,\tau )}\int_{(0,\tau ]}r^{\gamma
}\,dV(x,r).
\end{equation*}

\begin{lemma}
\label{M-R-identity} For all $x\in X$, $t>0$ and $\gamma >0$, 
\begin{equation*}
M_{\gamma }(x,t)=t\int_{0}^{\infty }R_{\gamma }\!\left( x,\frac{1}{\tau }%
\right) e^{-\tau t}\,d\tau =\int_{0}^{\infty }R_{\gamma }\!\left( x,\frac{t}{%
s}\right) e^{-s}\,ds.
\end{equation*}
\end{lemma}

\begin{proof}
Using the equations (\ref{M-p-identity}) and (\ref{pform}), as well as the
Definition \ref{def-spectral-d} of the spectral distribution function in
terms of the volume function, we obtain%
\begin{eqnarray*}
M_{\gamma }(x,t) &=&\int_{X}d_{\ast }(x,y)^{\gamma }\,p(t,x,y)\,d\mu (y) \\
&=&\int_{(0\,,\,\infty )}r^{\gamma }\left( t\int_{0}^{1/r}N(x,\tau
)\,e^{-\tau t}\,d\tau \right) \,dV(x,r) \\
&=&\int_{0}^{\infty }\left( \int_{(0\,,\,1/\tau )}\frac{r^{\gamma }}{%
V(x,1/\tau )}\,dV(x,r)\right) t\,e^{-\tau t}\,d\tau =\int_{0}^{\infty
}R_{\gamma }\!\left( x,\frac{1}{\tau }\right) t\,e^{-\tau t}\,d\tau .
\end{eqnarray*}%
In the 3rd identity, we have used Fubini's theorem.
\end{proof}

The volume function $r\mapsto V(x,r)$ non-decreasing. In view of its
relation with the spectral distribution function (see Definition \ref%
{def-spectral-d}) it is a step function whose shape can be understood from
figures \ref{pic1} -- \ref{pic3}. The function varies from $0$ to $\mu (X)$.
In the compact case, $V(x,r)=\mu (X)$ for all $r\geq r_{\max }^{\ast
}=r_{\max }^{\ast }(x)$, the largest value in $\Lambda (x)$ (see (\ref%
{Lambda})). When $x$ is isolated, $V(x,r)=\mu \{x\}$ for all $0\leq
r<r_{0}^{\ast }=r_{0}^{\ast }(x)\,$, the smallest positive value in $\Lambda
(x)$.

\begin{lemma}
\label{volume bounds} For any given $x\in X$ and $\gamma >0,$ the following
properties hold.

\begin{enumerate}
\item[\RM (a)] The function $\tau \mapsto R_{\gamma }(x,\tau)$ is
non-decreasing.

If $X$ is compact $R_{\gamma }( x,\tau) = R_{\gamma
}\left(x,r_{\max}^{\ast}(x)\right)$ for all $\tau \geq r_{\max}^{\ast}(x)$.

If $X$ is discrete and infinite, $R_{\gamma }( x,\tau) = R_{\gamma
}\left(x,r_{0}^{\ast}(x)\right)$ for all $0 < \tau \leq r_{0}^{\ast}(x)$.

\item[\RM (b)] For all $\tau >0,$ we have%
\begin{equation*}
R_{\gamma }(x,\tau) \leq \tau^{\gamma }
\end{equation*}%
and, if the volume function $r\mapsto V(x,r)$ satisfies the reverse doubling
property, then there exists a constant $c >0,$ such that 
\begin{equation}  \label{lowerb}
R_{\gamma }( x,\tau) \geq c\, \tau^{\gamma }
\end{equation}%
for all $\tau>0$. In the non-discrete compact case, if the volume function
just satisfies the reverse doubling property at zero, \emph{(\ref{lowerb})}
holds for all $0 < \tau < r_{\max}^{\ast}(x)$. In the discrete infinite
case, if the volume function just satisfies the reverse doubling property at
infinity, \emph{(\ref{lowerb})} holds for all $\tau > r_{0}^{\ast}(x)$.
\end{enumerate}
\end{lemma}

\begin{proof}
For the first part of (a), we integrate by parts: 
\begin{equation*}
R_{\gamma }(x,\tau) =\frac{1}{V(x,\tau )}\left( \tau^{\gamma}\, V(x,\tau ) -
\int_{(0\,,\,\tau ]}V(x,s)\,ds^{\gamma }\right) = \int_{(0\,,\,\tau ]}\left(
1-\frac{V(x,s)}{V(x,\tau )}\right)\, ds^{\gamma },
\end{equation*}%
whence $\tau \mapsto \mathcal{R}_{\gamma }\left(x,\tau \right) $ is
non-decreasing.

The second part (a) is straightforward.

Regarding (b), the general upper bound on $R_{\gamma }(x,\tau )$ is obvious.
If the volume function satisfies the reverse doubling property, then in the
respective range, 
\begin{eqnarray*}
R_{\gamma }(x,\tau ) &\geq &\frac{1}{V(x,\tau )}(\delta \tau )^{\gamma
}\left( V(x,\tau )-V(x,\delta \tau )\right) \\
&=&(\delta \tau )^{\gamma }\left( 1-\frac{V(x,\delta \tau )}{V(x,\tau )}%
\right) \geq \delta ^{\gamma }(1-\kappa )\tau ^{\gamma }=c\,\tau ^{\gamma }
\end{eqnarray*}%
for suitable constants $0<\kappa ,c<1$.
\end{proof}

Now, in order to estimate the moment function $t\mapsto M_{\gamma }(x,t)$,
we need to estimate a Laplace-type integral as given by the formula of Lemma %
\ref{M-R-identity}. We will treat such estimates in the two technical
Propositions \ref{Suppl1} and \ref{Suppl2} at the end of this section.
Before that, in the next three theorems, we anticipate the statements of the
results regarding the moment function.

\begin{theorem}
\label{Moments-non-compact} Assume that $(X,d)$ is non-compact and has no
isolated points. Then the following properties hold.

\begin{enumerate}
\item[\RM (1)] For all $x \in X$, $t>0$ and $0<\gamma <1,$ 
\begin{equation*}
M_{\gamma }(x,t)\leq \frac{t^{\gamma }}{1-\gamma }.
\end{equation*}

\item[\RM (2)] If for some $x\in X$, the volume function satisfies the
reverse doubling property, then for any $0<\gamma <1,$ 
\begin{equation*}
M_{\gamma }( x,t) \geq \frac{c }{1-\gamma }t^{\gamma },
\end{equation*}%
for all $x,t>0$ and some $c >0$. Moreover,%
\begin{equation*}
M_{\gamma }\left( z,t\right) =\infty ,
\end{equation*}%
for all $z,$ $t>0$ and $\gamma \geq 1.$
\end{enumerate}
\end{theorem}

\begin{theorem}
\label{Moments-noncompact-discrete} Assume that $(X,d)$ is discrete and
infinite. Then the following properties hold.

\begin{enumerate}
\item[\RM (a)] For all $x,$ $t>0$ and $0<\gamma <1,$ 
\begin{equation*}
M_{\gamma }(x,t)\leq \frac{C}{1-\gamma }\min \left\{ t,t^{\gamma }\right\}
\end{equation*}%
for some $C>0.$

\item[\RM (b)] If for some (equivalently, all) $x\in X$ the volume function
satisfies the reverse doubling property at infinity, then for any $0<\gamma
<1,$ 
\begin{equation*}
M_{\gamma }( z,t) \geq \frac{c}{1-\gamma }\min \left\{t,t^{\gamma }\right\}
\end{equation*}%
for all $z$ $,$ $t>0$ and for some $c>0$. Moreover, 
\begin{equation*}
M_{\gamma }( z,t) =\infty
\end{equation*}%
for all $z,$ $t>0$ and all $\gamma \geq 1.$
\end{enumerate}
\end{theorem}

Assume now that $(X,d)$ is compact and let $D$ be its $d_{\ast }$-diameter.
By Lemmas \ref{M-R-identity} and \ref{volume bounds}, for all $x\in X,$ $%
\gamma >0$ and $t>0,$%
\begin{equation*}
M_{\gamma }(x,t)\leq R_{\gamma }( x,D) \leq D^{\gamma },
\end{equation*}%
whence we study the behavior of the moment function $t\mapsto
M_{\gamma}(x,t) $ at zero.

\begin{theorem}
\label{Moments-compact} Assume that $(X,d)$ is non-discrete and compact.
Then the following properties hold.

\begin{enumerate}
\item[\RM(1)] There exists a constant $C>0$ such that 
\begin{equation*}
M_{\gamma }(x,t)\leq C\left\{ 
\begin{array}{lll}
t & \text{if} & \gamma >1, \\[3pt] 
t\,\left( \log \frac{1}{t}+1\right) & \text{if} & \gamma =1, \\[3pt] 
t^{\gamma } & \text{if} & \gamma <1,%
\end{array}%
\right.
\end{equation*}%
holds for all $x$ and all $0<t\leq 1.$

\item[\RM(2)] If for some $x\in X$ the volume function satisfies the reverse
doubling property at zero, then there exists a constant $c>0$ such that 
\begin{equation*}
M_{\gamma }(z,t)\geq c\left\{ 
\begin{array}{lll}
t & \text{if} & \gamma >1, \\[3pt] 
t\,\left( \log \frac{1}{t}+1\right) & \text{if} & \gamma =1, \\[3pt] 
t^{\gamma } & \text{if} & \gamma <1%
\end{array}%
\right.
\end{equation*}%
holds for all $z$ and all $0<t\leq 1.$
\end{enumerate}
\end{theorem}

We now provide the technical details regarding the Laplace-type estimates
that imply Theorems \ref{Moments-non-compact}, \ref%
{Moments-noncompact-discrete} and \ref{Moments-compact}. In the following
two propositions, $M$ and $R$ will always be two non-negative,
non-decreasing functions related by the Laplace-type integral 
\begin{equation*}
M(t)=\int_{0}^{\infty }R\!\left( \frac{t}{\tau }\right) e^{-\tau }\,d\tau\,.
\end{equation*}

\begin{proposition}
\label{Suppl1}Let $\gamma >0$ be given.

\begin{enumerate}
\item[\RM(1)] Assume that%
\begin{equation}
As^{\gamma }\geq R(s)\,,\quad \text{or that respectively}\quad R(s)\geq
B\,s^{\gamma }  \label{Fbounds}
\end{equation}%
for some $A>0$ (resp. $B>0$) and all $s>0.$ Then the inequality 
\begin{equation*}
\frac{A\,t^{\gamma }}{1-\gamma }\geq M(t)\,,\quad \text{respectively}\quad
M(t)\geq \frac{B\,t^{\gamma }}{(1-\gamma )\,e}
\end{equation*}%
holds for all $0<\gamma <1$ and all $t>0.$

\item[\RM (2)] Assume that there is $t_{0}> 0$ such that $R(s)=0$ for all $%
0<s<t_{0}\,.$ Assume also that one of the respective inequalities of (\ref%
{Fbounds}) holds for all $s>t_{0}\,.$ Then 
\begin{equation*}
M(t)\leq \frac{c}{1-\gamma } \min \left\{ \frac{t}{t_{0}},\left( \frac{t}{t_0%
}\right) ^{\gamma }\right\}, \quad \text{respectively} \quad M(t)\geq \frac{%
c^{\prime }}{1-\gamma } \min \left\{ \frac{t}{t_0},\left( \frac{t}{t_0}%
\right)^{\gamma }\right\},
\end{equation*}%
for all $0<\gamma <1,$ all $t>0$ and some constants $c,c^{\prime }>0.$

\item[\RM (3)] The assumption $\gamma \geq 1$ and the lower bound $R(s)\geq
B\, s^{\gamma }$ imply that $M(t)=\infty $ for all $t>0.$
\end{enumerate}
\end{proposition}

\begin{proof}
It is known that for $0<\gamma <1$ the Gamma-function satisfies%
\begin{equation*}
\frac{1}{(1-\gamma)\,e}<\Gamma (1-\gamma )<\frac{1}{1-\gamma },
\end{equation*}%
whence by monotonicity of the Laplace-type integral the first claim follows.

To prove the second statement, we write%
\begin{equation*}
M(t)=\int_{\left\{ t/s\geq t_{0}\right\} }R\!\left( \frac{t}{s}\right)
e^{-s}\,ds.
\end{equation*}%
First assume that $R(\tau )\leq A\,s^{\gamma }$ for all $0<s<\infty \,$.
Then we obtain%
\begin{eqnarray*}
M(t) &\leq &A\int_{\left\{ t/s\geq t_{0}\right\} }\left( \frac{t}{s}\right)
^{\gamma }e^{-s}\,ds=A\,t^{\gamma }\int_{\left\{ s\leq t/t_{0}\right\}
}s^{-\gamma }e^{-s}\,ds \\
&\leq &A\,t^{\gamma }\int_{0}^{t/t_{0}}s^{-\gamma }\,ds=\left( \frac{t}{t_{0}%
}\right) \frac{At_{0}^{-\gamma }}{1-\gamma }\,,\quad \text{and}
\end{eqnarray*}%
\begin{equation*}
M(t)\leq A\,t^{\gamma }\int_{0}^{\infty }s^{-\gamma }e^{-s}\,ds\leq \frac{%
A\,t^{\gamma }}{1-\gamma }=\left( \frac{t}{t_{0}}\right) ^{\gamma }\frac{%
A\,t_{0}^{\gamma }}{1-\gamma }.
\end{equation*}%
It follows that%
\begin{equation*}
M(t)\leq \frac{A\,\max \left\{ t_{0},t_{0}^{-1}\right\} }{1-\gamma }\min
\left\{ \frac{t}{t_{0}},\left( \frac{t}{t_{0}}\right) ^{\gamma }\right\} .
\end{equation*}%
Second, assume that $R(s)\geq B\,s^{\gamma },$ for all $s\geq t_{0}\,$. Then
for $t/t_{0}\geq 1$ 
\begin{equation*}
M(t)\geq B\,t^{\gamma }\int_{0}^{t/t_{0}}s^{-\gamma }e^{-s}\,ds\geq \frac{%
B\,t^{\gamma }}{e}\int_{0}^{1}s^{-\gamma }\,ds=\frac{B\,t^{\gamma }}{%
(1-\gamma )e}=\frac{B\,t_{0}^{\gamma }}{(1-\gamma )\,e}\left( \frac{t}{t_{0}}%
\right) ^{\gamma }.
\end{equation*}%
When $t/t_{0}\leq 1$ we get 
\begin{equation*}
M(t)\geq B\,t^{\gamma }\int_{0}^{t/t_{0}}s^{-\gamma }e^{-s}\,ds\geq \frac{%
B\,t^{\gamma }}{e}\int_{0}^{t/t_{0}}s^{-\gamma }\,ds=\frac{{B\,t^{\gamma }}}{%
{(1-\gamma )\,e}}\left( \frac{t}{t_{0}}\right) ^{1-\gamma }=\frac{%
B\,t_{0}^{\gamma }}{(1-\gamma )\,e}\left( \frac{t}{t_{0}}\right) .
\end{equation*}%
It follows that%
\begin{equation*}
M(t)\geq \frac{B\,t_{0}^{\gamma }}{(1-\gamma )\,e}\min \left\{ \frac{t}{t_{0}%
},\left( \frac{t}{t_{0}}\right) ^{\gamma }\right\} \geq \frac{B\min
\{t_{0},1\}}{(1-\gamma )\,e}\min \left\{ \frac{t}{t_{0}},\left( \frac{t}{%
t_{0}}\right) ^{\gamma }\right\} .
\end{equation*}%
This proves the second claim. For the third claim observe that that if $%
R(s)\geq B\,s^{\gamma }$ for all $s\geq t_{0}$ and $\gamma \geq 1,$ 
\begin{equation*}
M(t)\geq B\,t^{\gamma }\int_{0}^{t/t_{0}}s^{-\gamma }e^{-s}\,ds=\infty
\end{equation*}%
for all $t>0.$
\end{proof}

\begin{proposition}
\label{Suppl2}Assume that there is $t_0>0$ such that $R(s)=R(t_0)$ for all $%
s\geq t_0\,$. Assume also that one of the respective inequalities in (\ref%
{Fbounds}) holds for all $0<s\leq t_0\,.$ Then 
\begin{eqnarray*}
M(t) &\leq &\left\{ 
\begin{array}{ccc}
c_{1}\, \frac{t}{t_0} & \text{if} & \gamma >1, \\[4pt] 
c_{2}\,t\,\left(\log \frac{t_0}{t}+1\right) & \text{if} & \gamma =1, \\[4pt] 
c_{3}\left( \frac{t}{t_0}\right)^{\gamma } & if & \gamma <1,%
\end{array}%
\right. \\
\text{respectively,}\quad M(t)&\geq& \left\{ 
\begin{array}{ccc}
c_{1}^{\prime }\,\frac{t}{t_0} & \text{if} & \gamma >1, \\[4pt] 
c_{2}^{\prime }\,t\left(\log \frac{t_0}{t}+1\right) & \text{if} & \gamma =1,
\\[4pt] 
c_{3}^{\prime }\left( \frac{t}{t_0}\right)^{\gamma } & if & \gamma <1,%
\end{array}%
\right.
\end{eqnarray*}%
for all $0<t\leq t_0$ and some positive constants $c_{1}\,,c_{1}^{\prime}%
\,,c_{2}\,,c_{2}^{\prime }\,,c_{3}\,,c_{3}^{\prime }\,.$
\end{proposition}

\begin{proof}
Let $\gamma >1$ and $0<t<t_0.$ According to our assumption%
\begin{equation*}
M(t)=\int_{\left\{ t/s\leq t_0\right\} }R\!\left( \frac{t}{s}%
\right)e^{-s}\,ds +R(t_0)\left( 1-e^{-t/t_0}\right).
\end{equation*}%
Observe that for $0<t<t_0\,,$%
\begin{equation*}
\frac{t}{2t_0}\leq \left( 1-e^{-t/t_0}\right) \leq \frac{t}{t_0}\,.
\end{equation*}%
First, if $R(s)\leq A\,s^{\gamma }$ for all $0<s <t_0\,,$ then%
\begin{eqnarray*}
M(t) &\leq &A\,t^{\gamma }\int_{t/t_0}^{\infty }s^{-\gamma }e^{-s}\,ds+\frac{%
R(t_0)t}{t_0} \leq A\,s^{\gamma }\int_{t/t_0}^{\infty }s^{-\gamma }\,ds+%
\frac{R(t_0)t}{t_0} \\
&\leq &\frac{A\,t^{\gamma }}{\gamma -1}\left( \frac{t}{t_0}\right)^{1-\gamma
}+\frac{R(t_0)t}{t_0} =\frac{t}{t_0}\left( R(t_0)+\frac{A\,t_0^{\gamma }}{%
\gamma -1}\right).
\end{eqnarray*}%
Second, if $R(s)\geq B\,s^{\gamma },$ for all $0<s<t_0\,,$ then%
\begin{equation*}
M(t)\geq \frac{R(t_0)}{2}\frac{t}{t_0}.
\end{equation*}%
Assume that $0<\gamma <1$ and $0<t<t_0\,.$ Again first, if $R(s)\leq
A\,s^{\gamma }$ for all $0<\tau <t_0\,,$ then%
\begin{eqnarray*}
M(t) &\leq &\,At^{\gamma }\int_{t/t_0}^{\infty }s^{-\gamma }e^{-s}\,ds+\frac{%
R(t_0)t}{t_0} \leq A\,t^{\gamma }\,\Gamma (1-\gamma )+R(t_0)\frac{t}{t_0} \\
&\leq &\frac{A\,t^{\gamma }}{1-\gamma }+R(t_0)\frac{t}{t_0} =\frac{%
A\,t_0^{\gamma }}{1-\gamma }\left( \frac{t}{t_0}\right)^{\gamma }+R(t_0)%
\frac{t}{t_0} \\
&\leq &\left( \frac{t}{t_0}\right) ^{\gamma }\left( \frac{A\,T^{\gamma }}{%
1-\gamma }+R(t_0)\right).
\end{eqnarray*}%
Second, once more, when $R(s)\geq B\,s^{\gamma },$ for all $0<s <T$, then%
\begin{equation*}
M(t)\geq B\,t^{\gamma }\int_{t/t_0}^{\infty }s^{-\gamma }e^{-s}\,ds \geq
B\,t^{\gamma} \int_{1}^{\infty }s^{-\gamma }e^{-s}\,ds\geq \left( \frac{t}{%
t_0}\right)^{\gamma } \left( \frac{B\,\min \{t_0,1\}}{e^{2}}\right).
\end{equation*}%
Finally, assume that $\gamma =1$ and $0<t<t_0\,.$ First, if $R(s)\leq
A\,s^{\gamma }$ for all $0<\tau <t_0\,,$ then%
\begin{eqnarray*}
M(t) &\leq &At\int_{t/T}^{\infty }s^{-1}e^{-s}ds+\frac{R(T)t}{T} \\
&=&\,At\left(
\int_{1}^{\infty}s^{-1}e^{-s}\,ds+\int_{t/t_0}^{1}s^{-1}e^{-s}\,ds\right) +%
\frac{R(t_0)t}{t_0} \\
&\leq &A\,t\left( \int_{1}^{\infty }\frac{ds}{s^{2}}+\int_{t/t_0}^{1}\frac{ds%
}{s}\right) +\frac{R(t_0)t}{t_0} =\left( A+\frac{R(t_0)}{t_0}\right) t\left(
\log \frac{t_0}{t}+1\right) .
\end{eqnarray*}%
And at last, if $R(s)\geq B\,s^{\gamma }$ for all $0<\tau <t_0,$ then%
\begin{eqnarray*}
M(t) &\geq &B\,t\int_{t/t_0}^{\infty }s^{-1}e^{-s}\,ds+\frac{R(t_0)t}{2t_0}
\\
&\geq &\frac{B\,t}{e}\int_{t/t_0}^{1}\frac{ds}{s}+\frac{R(t_0)t}{2t_0} =%
\frac{B\,t}{e}\log \frac{t_0}{t}+\frac{R(t_0)t}{2t_0} \\
&=&\frac{B\,t}{e}\left( \log \frac{t_0}{t}+\frac{R(t_o)e}{2B\,t_0}\right)
\geq \min\left\{ \frac{R(t_0)}{2t_0},\frac{B}{e}\right\} t\left( \log \frac{%
t_0}{t}+1\right) .
\end{eqnarray*}%
The proof is finished.
\end{proof}

Theorems \ref{Moments-non-compact}, \ref{Moments-noncompact-discrete} and %
\ref{Moments-compact} follow.

\section{Analysis in $\mathbb{Q}_{p}$ and $\mathbb{Q}_{p}^{n}$}

\setcounter{equation}{0}\label{SecQp}

\subsection{The $p$-adic fractional derivative}

\label{derivative}Consider the field $\mathbb{Q}_{p}$ of $p$-adic numbers
endowed with the $p$-adic norm $\left\Vert x\right\Vert _{p}$ and the $p$%
-adic ultra-metric $d_{p}(x,y)=\left\Vert x-y\right\Vert _{p}.$ Let $\mu
_{p} $ be the Haar measure on $\mathbb{Q}_{p}\,$, normalized such that $\mu
_{p}(\mathbb{Z}_{p})=1.$ Let $\mathcal{V}_{c}$ be the space of locally
constant functions on $\mathbb{Q}_{p}$ with compact support which will be
considered as test functions on $\mathbb{Q}_{p}$.

The notion of $p$-adic fractional derivative, closely related to the concept
of $p$-adic Quantum Mechanics, was introduced in several papers by
Vladimirov~\cite{Vladimirov}, Vladimirov and Volovich~ \cite%
{VladimirovVolovich} and Vladimirov, Volovich and Zelenov~\cite{Vladimirov94}%
. In particular, a one-parameter family $\{\mathfrak{D}^{\alpha }\}_{\alpha
>0}$ of operators, called operators of fractional derivative of order $%
\alpha $, was introduced in \cite{Vladimirov}.

Recall that the Fourier transform $\mathcal{F}:f\mapsto \widehat{f}$ of a
function $f$ on the self-dual locally compact Abelian group $\mathbb{Q}_{p}$
is defined by 
\begin{equation*}
\widehat{f}(\theta )=\int_{\mathbb{Q}_{p}}\left\langle x,\theta
\right\rangle f(x)d\mu _{p}(x),
\end{equation*}%
where $x,\theta \in \mathbb{Q}_{p}$,%
\begin{equation*}
\left\langle x,\theta \right\rangle =\exp \left( 2\pi \sqrt{-1}\left\{ {%
x\theta }\right\} \right) ,
\end{equation*}%
and $\left\{ x\theta \right\} $ is the fractional part of the $p$-adic
number $x\theta $ (cf. (\ref{xa})). It is known that $\mathcal{F}$ is a
linear isomorphism of $\mathcal{V}_{c}$ onto itself.

\begin{definition}
\RM\label{def-frac-derivative}The operator $(\mathfrak{D}^{\alpha },\mathcal{%
V}_{c}),$ $\alpha >0,$ is defined via the Fourier transform on the locally
compact Abelian group $\mathbb{Q}_{p}$ by 
\begin{equation*}
\widehat{\mathfrak{D}^{\alpha }f}(\xi )=\left\Vert \xi \right\Vert
_{p}^{\alpha }\widehat{f}(\xi ),\text{ \ \ \ }\xi \in \mathbb{Q}_{p}.
\end{equation*}
\end{definition}

It was shown by the above named authors that each operator $(\mathfrak{D}%
^{\alpha },\mathcal{V}_{c}\mathcal{)}$ can be written as a Riemann-Liouville
type singular integral operator%
\begin{equation}
\mathfrak{D}^{\alpha }f(x)=\frac{p^{\alpha }-1}{1-p^{-\alpha -1}}\int_{%
\mathbb{Q}_{p}}\frac{f(x)-f(y)}{\left\Vert x-y\right\Vert _{p}^{1+\alpha }}%
\,d\mu _{p}(y).  \label{sing-int}
\end{equation}%
The aim of this section is in particular to show that the operator $(%
\mathfrak{D}^{\alpha },\mathcal{V}_{c})$ is in fact the restriction to $%
\mathcal{V}_{c}$ of an appropriate isotropic Laplacian. We use the following
distance distribution function 
\begin{equation*}
\sigma _{\alpha }(r)=\exp \left( -\left( \frac{p}{r}\right) ^{\alpha
}\right) .
\end{equation*}%
Denote by $\{P_{\alpha }^{t}\}$ the isotropic semigroup associated with the
triple $(d_{p},\mu _{p},\sigma _{\alpha })$, and let $\mathcal{L}_{\alpha }$
be the corresponding Laplacian.

\begin{theorem}
\label{Vladimirov Laplacian} For any $\alpha >0,$ we have%
\begin{equation}
(\mathcal{L}_{\alpha },\mathcal{V}_{c})=(\mathfrak{D}^{\alpha }\,,\mathcal{V}%
_{c}).  \label{L=D}
\end{equation}
\end{theorem}

\begin{proof}
By Theorem \ref{Thm-Dirichlet form/Laplacian}, we have, for any $f\in 
\mathcal{V}_{c},$%
\begin{equation*}
\mathcal{L}_{\alpha }f(x)=\int_{\mathbb{Q}_{p}}\left( f(x)-f(y)\right)
J_{\alpha }(x,y)\,d\mu _{p}(y),
\end{equation*}%
where 
\begin{equation*}
J_{\alpha }(x,y)=\int_{d_{\ast }\left( x,y\right) }^{\infty }\frac{s^{-2}ds}{%
\mu _{p}\left( B_{s}^{\ast }(x)\right) }.
\end{equation*}%
As in Example (\ref{p-adic-doubling}), we have 
\begin{equation}
d_{\ast }(x,y)=\left( \frac{\left\Vert x-y\right\Vert _{p}}{p}\right)
^{\alpha },  \label{d*pp}
\end{equation}%
whence%
\begin{equation*}
B_{s}^{\ast }(x)=B_{ps^{1/\alpha }}(x).
\end{equation*}%
The change $r=ps^{1/\alpha }$ yields%
\begin{equation*}
J_{\alpha }(x,y)=p^{\alpha }\int_{\left\Vert x-y\right\Vert _{p}}^{\infty }%
\frac{\alpha r^{-\alpha -1}\,dr}{\mu _{p}\left( B_{r}(x)\right) }.
\end{equation*}%
Since the value set of the metric $\left\Vert x-y\right\Vert _{p}$ is $%
\left\{ p^{n}\right\} _{k\in \mathbb{Z}}$, we obtain from (\ref{p-n}) that%
\begin{equation}
\mu _{p}\left( B_{r}\left( x\right) \right) =p^{n}\ \ \text{if\ \ }p^{n}\leq
r<p^{n+1},  \label{mupB}
\end{equation}
which implies, for $\left\Vert x-y\right\Vert _{p}=p^{k}$, that%
\begin{eqnarray*}
\int_{p^{k}}^{\infty }\frac{\alpha\, r^{-\alpha -1}\,dr}{\mu _{p}\left(
B_{r}(x)\right) } &=&\sum_{n\geq k}\int_{p^{n}}^{p^{n+1}}\frac{\alpha
\,r^{-\alpha -1}\,dr}{\mu _{p}\left( B_{r}(x)\right) } \\
&=&\sum_{n\geq k}\int_{p^{n}}^{p^{n+1}}\frac{-dr^{-\alpha }}{p^{n}}%
=\sum_{n\geq k}\frac{1}{p^{n}}\left( \frac{1}{p^{n\alpha }}-\frac{1}{%
p^{(n+1)\alpha }}\right) \\
&=&\left( 1-\frac{1}{p^{\alpha }}\right) \sum_{n\geq k}\frac{1}{p^{n(\alpha
+1)}}=\left( 1-\frac{1}{p^{\alpha }}\right) \frac{p^{-k(\alpha +1)}}{%
1-p^{-(\alpha +1)}} \\
&=&\frac{1-p^{-\alpha }}{1-p^{-(\alpha +1)}}\left( \frac{1}{p^{k}}\right)
^{\alpha +1}=\frac{1-p^{-\alpha }}{1-p^{-(\alpha +1)}}\left( \frac{1}{%
\left\Vert x-y\right\Vert _{p}}\right) ^{\alpha +1}.
\end{eqnarray*}%
Hence, we obtain the identity%
\begin{equation}
J_{\alpha }(x,y)=\frac{p^{\alpha }-1}{1-p^{-\alpha -1}}\,\frac{1}{\left\Vert
x-y\right\Vert _{p}^{\alpha +1}},  \label{jump-kernel-Q_p}
\end{equation}%
which in view of (\ref{sing-int}) finishes the proof.
\end{proof}

\smallskip

The heat kernel for the semigroup $\left\{ P_{\alpha }^{t}\right\} $ was
estimated in Example \ref{p-adic-doubling}. We restate this estimate here as
a theorem.

\begin{theorem}
\label{Vladimirov heat kernel}The semigroup $\{P_{\alpha }^{t}\}$ admits a
continuous transition density $p_{\alpha }(t,x,y)$ with respect to Haar
measure $\mu _{p}$, which satisfies for all $t>0$ and $x,y\in \mathbb{Q}_{p}$
the estimate%
\begin{equation}
p_{\alpha }(t,x,y)\simeq \frac{t}{(t^{1/\alpha }+\left\Vert x-y\right\Vert
_{p})^{1+\alpha }}.  \label{pa}
\end{equation}
\end{theorem}

The upper bound in (\ref{pa}) was also obtained by a different method by
Kochubei \cite[Ch.4.1, Lemma 4.1]{Kochubei2001}.

\begin{theorem}
\label{Vladimirov Green function} The semigroup $\{P_{\alpha }^{t}\}$ is
transient if and only if $\alpha <1.$ In the transient case, its Green
function $g_{\alpha }$ is given explicitly by 
\begin{equation}
g_{\alpha }(x,y)=\frac{1-p^{-\alpha }}{1-p^{\alpha -1}}\left\Vert
x-y\right\Vert _{p}^{\alpha -1}.  \label{Ga}
\end{equation}
\end{theorem}

The formula (\ref{Ga}) for a fundamental solution of $\mathfrak{D}_{\alpha }$
acting in the space $\mathcal{V}_{c}^{\prime }$ of distribution, was
obtained by Vladimirov~\cite[Thm 1, p.51]{Vladimirov} and Kochubei~\cite[%
Ch.2.2]{Kochubei2001}.

\begin{proof}
That $\alpha <1$ is equivalent to transience was shown in Example \ref%
{p-adic-transient}. Assuming $\alpha <1$, we obtain by (\ref{green function})%
\begin{equation*}
g_{\alpha }(x,y)=\int_{d_{\ast }(x,y)}^{\infty }\frac{ds}{\mu
_{p}(B_{s}^{\ast }\left( x\right) )}=\frac{1}{p^{\alpha }}\int_{\left\Vert
x-y\right\Vert _{p}}^{\infty }\frac{\alpha r^{\alpha -1}\,dr}{\mu _{p}\left(
B_{r}(x)\right) }.
\end{equation*}%
Setting $\left\Vert x-y\right\Vert _{p}=p^{k}$ and using (\ref{mupB}), we
obtain%
\begin{eqnarray*}
g_{\alpha }(x,y) &=&\frac{1}{p^{\alpha }}\sum_{n\geq k}\int_{p^{n}}^{p^{n+1}}%
\frac{dr^{\alpha }}{p^{n}}=\frac{1}{p^{\alpha }}\sum_{n\geq k}\frac{1}{p^{n}}%
\left( p^{(n+1)\alpha }-p^{n\alpha }\right) \\
&=&\frac{1-p^{-\alpha }}{1-p^{\alpha -1}}p^{\left( \alpha -1\right) k},
\end{eqnarray*}%
which finishes the proof.
\end{proof}

\smallskip

Denote by $\mathcal{L}_{\alpha ,q}$ the generator of the semigroup $%
\{P_{\alpha }^{t}\}$ acting in $L^{q}(\mu _{p})$, $1\leq q<\infty .$
Applying Corollary \ref{L^2-Spectrum} and Theorem \ref{Lp-Lq thm}, we obtain
the following.

\begin{theorem}
\label{Vladimirov-spectrum} For any $\alpha >0$ and $1\leq q<\infty ,$ we
have%
\begin{equation*}
\func{spec}\mathcal{L}_{\alpha ,q}=\{p^{\alpha k}:k\in \mathbb{Z}\}\cup
\{0\}.
\end{equation*}%
Each $\lambda _{k}=p^{\alpha k}$ is an eigenvalue with infinite multiplicity.
\end{theorem}

\begin{proof}
We only need to show that the multiplicity of $\lambda _{k}$ is infinite. In
the general setting of Theorem \ref{Eigenvalues-thm} and Corollary \ref%
{L^2-Spectrum}, some eigenvalues may well have finite multiplicity and some
not. Indeed, each ball $B$ with the minimal positive $d_{\ast }$-radius $%
\rho $ generates a finite dimensional eigenspace $\mathcal{H}_{B}$ that
consists of eigenfunctions with the eigenvalue $\frac{1}{\rho }$. It follows
that the eigenvalue $\frac{1}{\rho }$ has finite multiplicity if and only if
there is only a finite number of distinct balls of $d_{\ast }$-radius $\rho $%
.

In the present setting in $\mathbb{Q}_{p}\ $there are infinitely many
disjoint balls of the same radius $\rho $, as they all can be obtained by
translations of one such ball. Thus, all the eigenvalues have infinite
multiplicity.
\end{proof}

\smallskip

Let $\{X_{t}\}$ be the Markov process on $\mathbb{Q}_{p}\,$ driven by the
Markov semigroup $\{P_{\alpha }^{t}\}_{t>0}$. The semigroup is translation
invariant, whence the process has independent and stationary increments. For
any given $\gamma >0$ and $t>0$ , consider the moment of order $\gamma $ of $%
\mathcal{X}_{t}$ defined in terms of the $p$-adic distance $d_{p}(x,y)$: 
\begin{equation*}
\mathcal{M}_{\gamma }(t)=\mathbb{E}(\left\Vert \mathcal{X}_{t}\right\Vert
_{p}^{\gamma }),
\end{equation*}%
where $\mathbb{E}$ is expectation with respect to the probability measure on
the trajectory space of the process starting at $0$. Applying Theorem \ref%
{Moments-non-compact} and using the relation (\ref{d*pp}) between $d_{\ast }$
and $\left\Vert \cdot \right\Vert _{p}$, we obtain the following estimates.

\begin{theorem}
\label{Vladimirov moments} The moment $\mathcal{M}_{\gamma }(t)$ is finite
if and only if $\gamma <\alpha .$ In that case, 
there exists a constant $\kappa =\kappa (\alpha )>0$ such that 
\begin{equation*}
\frac{\kappa\, t^{\gamma /\alpha }}{\alpha -\gamma }\leq \mathcal{M}_{\gamma
}(t)\leq \frac{\alpha\, t^{\gamma /\alpha }}{\alpha -\gamma }.
\end{equation*}
\end{theorem}

\subsection{Rotation invariant Markov semigroups}

Let $\{P_{t}\}_{t\geq 0}$ be a symmetric, translation invariant Markov
semigroup on the additive Abelian group $\mathbb{Q}_{p}$. This semigroup
acts in $C_{0}(\mathbb{Q}_{p})$, the Banach space of continuous functions
vanishing at $\infty $. It follows that there exists a weakly continuous
convolution semigroup $\{p_{t}\}_{t>0}$ of symmetric probability measures on 
$\mathbb{Q}_{p}$ such that%
\begin{equation}
P_{t}f(x)=p_{t}\ast f(x).  \label{p-t-convolution}
\end{equation}%
As the probability measures $p_{t}$ are symmetric, the following identity
holds, which is basic in the theory of infinite divisible distributions: 
\begin{equation*}
\widehat{p_{t}}(\zeta )=\exp \left( -t\,\Psi (\zeta )\right) ,
\end{equation*}%
where $\Psi :\mathbb{Q}_{p}\mapsto \mathbb{R}_{+}$ is a non-negative
definite symmetric function on $\mathbb{Q}_{p}\,$. By the L\'{e}vy-Khinchin
formula,%
\begin{equation*}
\Psi (\zeta )=\int_{\mathbb{Q}_{p}\setminus \{0\}}\left( 1-\func{Re}\langle
x,\zeta \rangle \right) \,d\mathfrak{J}(x),
\end{equation*}%
where $\mathfrak{J}$ is a symmetric Radon measure on $\mathbb{Q}%
_{p}\setminus \{0\}$ -- the Levy measure associated with the non-negative
definite function $\Psi $ (see for the details the book of Berg and Forst~%
\cite{Berg75}).

\begin{definition}
\label{Isotropic semigroup}\RM For any $a\in \mathbb{Q}_{p}$ with $%
\left\Vert a\right\Vert _{p}=1$ define the rotation operator $\theta _{a}:%
\mathbb{Q}_{p}\rightarrow \mathbb{Q}_{p}$ by $\theta _{a}\left( x\right)
=ax. $ We say that the Markov semigroup $\left\{ P_{t}\right\} $ as above is 
\emph{rotation invariant} if 
\begin{equation}
\theta _{a}(p_{t})=p_{t}\ \ \text{for all}\;a\in \mathbb{Q}_{p}\;\text{with}%
\;\left\Vert a\right\Vert _{p}=1,  \label{isotropic-p}
\end{equation}
\end{definition}

Let $\mathcal{L}$ be the (positive definite) generator of $P_{t}$, that is, $%
P_{t}=\exp \left( -t\mathcal{L}\right) $. It is easy to see that (\ref%
{isotropic-p}) is equivalent to $\theta _{a}\circ \mathcal{L}=\mathcal{L}%
\circ \theta _{a}\,.$ In this case we also say that $\mathcal{L}$ is
rotation invariant. By construction, any isotropic Markov semigroup $\left\{
P^{t}\right\} $ defined on the ultra-metric measure space $(\mathbb{Q}%
_{p},d_{p},\mu _{p})$ is rotation invariant. As we will see the class of all
isotropic Markov semigroups is indeed a proper subset of the class of
rotation invariant Markov semigroups.

Assume that the semigroup $\{P_{t}\}$ is rotation invariant. Then, for all $%
a $ such that $\left\Vert a\right\Vert _{p}=1$ we have%
\begin{equation}
\Psi (a\zeta )=\Psi (\zeta )\quad \text{and}\quad \theta _{a}(\mathfrak{J})=%
\mathfrak{J}.  \label{isotropic-J}
\end{equation}%
Since the Haar measure $\mu _{p}$ of each sphere is strictly positive, (\ref%
{isotropic-p}) and (\ref{isotropic-J}) imply that the measures $p_{t}$ and $%
\mathfrak{J}$ are absolutely continuous with respect to $\mu _{p}$ and have
densities $p_{t}(x)$ and $J(x)$ which depend only on $\left\Vert
x\right\Vert _{p}\,$. The same is true for the function $\Psi $, so that%
\begin{equation*}
J(x)=\mathfrak{j}(\left\Vert x\right\Vert _{p})\quad \text{and}\quad \Psi
(\zeta )=\psi (\left\Vert \zeta \right\Vert _{p}).
\end{equation*}%
All the above shows that, for the generator $\mathcal{L}$ of $\left\{
P_{t}\right\} $, we have $\mathcal{V}_{c}\subset \func{dom}_{\mathcal{L}}$
and 
\begin{equation}
\mathcal{L}u=\psi (\mathfrak{D})u,\quad u\in \mathcal{V}_{c},  \label{L-D-eq}
\end{equation}%
where $\mathfrak{D=D}^{1}$ is the operator of fractional derivative of order 
$\alpha =1$, which we identify with the isotropic Laplacian $\mathcal{L}_{1}$
by Theorem \ref{Vladimirov Laplacian}.

It follows from (\ref{L-D-eq}) and (\ref{L=D}) that the eigenfunctions of
the operator $(\mathcal{L},\mathcal{V}_{c})$ in $L^{2}$ has a complete
system of eigenfunctions $\{f_{C}:C\in \mathcal{K}\}$ as described in
Theorem \ref{Eigenvalues-thm}. Associated with each ball $B$ of radius $%
p^{m} $, there is the $(p-1)$-dimensional eigenspace $\mathcal{H}_{B}$
spanned by all functions $f_{C},$ where $C$ runs through all balls that are
children of $B$, and the corresponding eigenvalue is 
\begin{equation*}
\lambda (m)=\psi (p^{-m+1}).
\end{equation*}%
Let $\{a(m)\}_{m\in \mathbb{Z}}$ be a sequence of real numbers satisfying%
\begin{equation}
a(m)\geq a(m+1),\quad a(+\infty )=0\quad \text{and }\quad 0<a(-\infty
)=W\leq +\infty .  \label{A-sequence}
\end{equation}%
Define the sequence $\{\lambda (m)\}_{m\in \mathbb{Z}}$ by%
\begin{equation}
\lambda (m)=a(m)-(p-1)^{-1}\{a(m+1)-a(m)\}.  \label{Lambda-sequence}
\end{equation}

\begin{theorem}
\label{Isotropic Laplacian}A sequence $\left\{ \lambda \left( m\right)
\right\} _{m\in \mathbb{Z}}$ of reals represents the spectrum $\func{spec}%
\mathcal{L}$ of a rotation invariant Laplacian $\mathcal{L}$ on $\mathbb{Q}%
_{p}$ if and only if it is given by \emph{(\ref{Lambda-sequence})} with a
sequence $a\left( m\right) $ that satisfies \emph{(\ref{A-sequence})}$.$
\end{theorem}

\begin{proof}
Consider a rotation invariant Laplacian $\mathcal{L}=\psi (\mathfrak{D})$.
Let us compute the non-negative definite function $\Psi (\zeta )=\psi
(\left\Vert \zeta \right\Vert _{p})$ associated with $\mathcal{L}.$ We have%
\begin{eqnarray*}
\psi (\left\Vert \zeta \right\Vert _{p}) &=&\int_{\mathbb{Q}_{p}\setminus
\{0\}}\left( 1-\func{Re}\langle x,\zeta \rangle \right) \,\mathfrak{j}%
(\left\Vert x\right\Vert _{p})\,d\mu _{p}(x) \\
&=&\sum_{k\in \mathbb{Z}}\mathfrak{j}(p^{k})\int_{\{x:\left\Vert
x\right\Vert _{p}=p^{k}\}}\left( 1-\func{Re}\langle x,\zeta \rangle \right)
\,d\mu _{p}(x).
\end{eqnarray*}%
According to Vladimirov~\cite[ Example 4]{Vladimirov}, 
\begin{equation*}
\int_{\{x:\left\Vert x\right\Vert _{p}=p^{k}\}}\langle x,\zeta \rangle
\,d\mu _{p}(x)=\left\{ 
\begin{array}{ccc}
p^{k}-p^{k-1} & \text{if} & \left\Vert \zeta \right\Vert _{p}\leq p^{-k}, \\ 
-p^{k-1} & \text{if} & \left\Vert \zeta \right\Vert _{p}=p^{-k+1}, \\ 
0 & \text{if} & \left\Vert \zeta \right\Vert _{p}\geq p^{-k+2}.%
\end{array}%
\right.
\end{equation*}%
In particular, we have%
\begin{equation*}
\int_{\{x:\left\Vert x\right\Vert _{p}=p^{k}\}}d\mu _{p}(x)=p^{k}-p^{k-1}.
\end{equation*}%
Let $\left\Vert \zeta \right\Vert _{p}=p^{-m+1},$ then the above
computations yield%
\begin{equation}
\psi (p^{-m+1})=\mathfrak{j}(p^{m})\,p^{m}+\left( 1-p^{-1}\right)
\sum_{k\geq m+1}\mathfrak{j}(p^{k})\,p^{k}.  \label{psi-eq}
\end{equation}%
Define the non-increasing sequence $\{a(m)\}_{m\in \mathbb{Z}}$ by 
\begin{equation}
a(m)=\left( 1-p^{-1}\right) \sum\limits_{k\geq m}\mathfrak{j}%
(p^{k})\,p^{k}=\left( 1-p^{-1}\right) \int_{\{x:\left\Vert x\right\Vert
_{p}\geq p^{m}\}}\mathfrak{j}(\left\Vert x\right\Vert _{p})\,d\mu _{p}(x).
\label{a-J-eq}
\end{equation}%
By (\ref{a-J-eq}), the equation (\ref{psi-eq}) will get the following form%
\begin{eqnarray}
\psi (p^{-m+1}) &=&\frac{p}{p-1}\left( a(m)-a(m+1)\right) +a(m+1)
\label{psi-a-eq} \\
&=&a(m)-(p-1)^{-1}\left( a(m+1)-a(m)\right) \}.  \notag
\end{eqnarray}%
Let $\lambda (m)$ be the eigenvalue of the Laplacian $(\psi (\mathfrak{D}),%
\mathcal{V}_{c})$ corresponding to the ball $B$ of radius $p^{m}.$ Then $%
\lambda (m)=\psi (p^{-m+1})$ and the identity (\ref{psi-a-eq}) gives the
desired result, namely, the equation (\ref{Lambda-sequence}).

Conversely, given a sequence $\{a(m)\}$ as in (\ref{A-sequence}), we define
the sequence $\{\lambda (m)\}$ by \ref{Lambda-sequence} and set 
\begin{eqnarray}
\Psi (\xi ) &=&\psi (\left\Vert \xi \right\Vert _{p}),\quad \text{where}%
\quad \psi (p^{m})=\lambda (-m+1),\quad \text{and}  \notag \\
J(x) &=&\mathfrak{j}(\left\Vert x\right\Vert _{p}),\quad \text{where}\quad 
\mathfrak{j}(p^{m})=\left. \left( a(m)-a(m+1)\right) \right/
(p^{m}-p^{m-1})\,.  \label{J-Psi}
\end{eqnarray}%
It is straightforward to show that 
\begin{equation*}
\Psi (\zeta )=\int_{\mathbb{Q}_{p}\setminus \{0\}}\left( 1-\func{Re}\langle
x,\zeta \rangle \right) \,J(x)\,d\mu _{p}(x)\,,
\end{equation*}%
whence $\Psi $ is a non-negative definite function. It follows that the
function $\exp (-t\,\Psi )$ is positive definite, whence it is the Fourier
transform of a probability measure $p_{t}\,$. Clearly, $\{p_{t}\}_{t>0}$ is
a weakly continuous convolution semigroup of probability measures. By
construction, each measure $p_{t}$ is rotation invariant. Finally, we can
define the translation invariant Markov semigroup $P_{t}f=f\ast p_{t}\,$, as
desired.
\end{proof}

\begin{corollary}
\label{Isotropic-Monotone Class} In the above notation the following
statements are equivalent.

\begin{enumerate}
\item[$\left( 1\right) $] The sequence $\lambda (m)$ is non-increasing.

\item[$\left( 2\right) $] The sequence $\psi (p^{m})$ is non-decreasing.

\item[$\left( 3\right) $] The sequence $\mathfrak{j}(p^{m})$ is
non-increasing.

In particular, if the sequence $a(m)$ is convex, then each of the equivalent
properties $\left( 1\right) -\left( 3\right) $ holds.
\end{enumerate}
\end{corollary}

\begin{proof}
The equivalence $(1)\Longleftrightarrow (2)$ follows from the relation $%
\lambda (m)=\psi (p^{-m+1}).$ To prove that $(1)\Longleftrightarrow (3)$, we
apply (\ref{J-Psi}) and obtain%
\begin{equation*}
\lambda (m)-\lambda (m+1)=(p^{m}-p^{m-1})\left( \mathfrak{j}(p^{m})-%
\mathfrak{j}(p^{m+1})\right) .
\end{equation*}%
The equivalence $\left( 1\right) \Leftrightarrow \left( 2\right)
\Leftrightarrow \left( 3\right) $ follows. Finally, (\ref{Lambda-sequence})
and the convexity of $a\left( m\right) $ yield $(1).$
\end{proof}

Next, we consider strict monotonicity.

\begin{corollary}
\label{Our Class}The following statements are equivalent

\begin{enumerate}
\item[$\left( i\right) $] The sequence $\lambda (m)$ is strictly decreasing,
and $\lambda (-\infty )=+\infty \,$.

\item[$\left( ii\right) $] The sequence $\psi (p^{m})$ is strictly
increasing, and $\psi (+\infty )=+\infty \,$.

\item[$\left( iii\right) $] The sequence $\mathfrak{j}(p^{m})$ is strictly
decreasing, and $\int \mathfrak{j}(\left\Vert x\right\Vert _{p})\,d\mu
_{p}(x)=+\infty \,$.

\item[$\left( iv\right) $] The associated rotation invariant Markov
semigroup $\{P_{t}\}$ is isotropic.

In particular, if the sequence $a(m)$ is strictly convex and $a(-\infty
)=+\infty \,$, then each of the equivalent properties (i)--(iv) holds.
\end{enumerate}
\end{corollary}

\begin{proof}
The equivalence $(i)\Leftrightarrow (ii)\Leftrightarrow (iii)$ follows by
the same arguments as in the proof of Corollary \ref{Isotropic-Monotone
Class}. The convexity of $a\left( m\right) $ together with $a\left( -\infty
\right) =+\infty $ imply $(i)$ following the same argument. We are left to
show that $(iv)\Longleftrightarrow (ii)$.

Assume that $\{P_{t}\}$ is a isotropic Markov semigroup as constructed in (%
\ref{Averager}) -- (\ref{convexcomb-Pt}). The semigroup admits a continuous
transition density $p(t,x,y)=p_{t}(x-y)$ with respect to the Haar measure $%
\mu _{p}$; the function $p_{t}$ is given by%
\begin{equation}
p_{t}(y)=\int_{0}^{\infty }q_{s}(y)\,d\sigma ^{t}(s),\quad \text{where\ \ }%
q_{s}(y)=\frac{1}{\mu _{p}\left( B_{s}(0)\right) }\,\mathbf{1}%
_{B_{s}(0)}(y)\,.  \label{p_t-sigma}
\end{equation}%
To find the Fourier transform $\widehat{p_{t}}(\xi )$, we argue as follows.
The ball $B_{s}(0)$, $p^{k}\leq s<p^{k+1}$, is the compact subgroup $p^{-k}%
\mathbb{Z}_{p}$ of $\mathbb{Q}_{p}$, whence the measure $\omega
_{s}=q_{s}\,\mu _{p}$ coincides with the normed Haar measure of that compact
subgroup. Since for any locally compact Abelian group, the Fourier transform
of the normed Haar measure of any compact subgroup is the indicator of its
annihilator group and, in our particular case, the annihilator of the group $%
p^{-k}\mathbb{Z}_{p}$ is the group $p^{k}\mathbb{Z}_{p}\,$, we obtain 
\begin{equation*}
\widehat{\omega _{s}}(\xi )=\mathbf{1}_{p^{k}\mathbb{Z}_{p}}(\xi )=\mathbf{1}%
_{[0,p^{-k}]}(\left\Vert \xi \right\Vert _{p}),\quad \text{where\ \ }%
p^{k}\leq s<p^{k+1}.
\end{equation*}%
It follows that when $\left\Vert \xi \right\Vert _{p}=p^{-l}$,%
\begin{equation*}
\widehat{p_{t}}(\xi )=\sum_{k:k\leq l}\left( \sigma ^{t}(p^{k+1})-\sigma
^{t}(p^{k})^{t}\right) =\sigma ^{t}(p^{l+1})=\exp \left( -t\,\psi
(\left\Vert \xi \right\Vert _{p})\right) ,
\end{equation*}%
whence%
\begin{equation*}
\psi (p^{-l})=\log \frac{1}{\sigma (p^{l+1})}.
\end{equation*}%
According to (\ref{sigma}), the sequence $\sigma (p^{l})$ is assumed to be
strictly increasing and to tend to zero as $l\rightarrow -\infty \,$. Thus, $%
\psi (p^{m})$ is as claimed in $(ii)$. Conversely, if a strictly increasing
sequence $\psi (p^{m})$ as in (ii) is given, we define the strictly
increasing sequence 
\begin{equation*}
\sigma (p^{m})=\exp \left( -\psi (p^{-m+1})\right) .
\end{equation*}%
Let $\sigma :[0\,,\,\infty )\rightarrow \lbrack 0\,,\,1)$ be any increasing
bijection which takes the values $\sigma (p^{m})$ at the points $p^{m}$. We
define the function $p_{t}(y)$ by the equation (\ref{p_t-sigma}). As $\sigma
(+\infty )=1,$ this is a probability density with respect to $\mu _{p}\,$.
It is straightforward that $\{p_{t}\}_{t>0}$ gives rise to a weakly
continuous convolution semigroup of probability measures on $\mathbb{Q}%
_{p}\, $. Moreover, each $p_{t}$ is rotation invariant by construction.
Thus, the semigroup $P_{t}:f\mapsto f\ast p_{t}$ is as desired.
\end{proof}

\begin{remark}
\RM\label{AlbeverioKarwowski} In \cite{AK1}, Albeverio and Karwowski started
with a sequence $\{a(m)\}_{m\in \mathbb{Z}}$ as in (\ref{A-sequence}) and
used the classical approach of backward and forward Kolmogorov equations to
construct a Markov semigroup $\{P_{t}\}$ on the ultra-metric measure space $(%
\mathbb{Q}_{p},d_{p},\mu _{p})$. In particular, they showed in \cite[Theorem
3.2 9]{AK1} that the Laplacian $\mathcal{L}$ of that semigroup has a pure
point spectrum $\{\lambda (m)\}$ as in (\ref{Lambda-sequence}), and the $%
\lambda (m)$-eigenspace is spanned by the functions $f_{B}\,$, where $B$
runs over all balls of radius $p^{m-1}$. Our Theorem \ref{Isotropic
Laplacian} shows that in fact the class of Markov semigroups constructed in 
\cite{AK1} coincides with the class of rotation invariant Markov semigroups.
\end{remark}

\subsection{Product spaces}

\label{SecProduct}Let $\{(X_{i},d_{i})\}_{i=1}^{n}$ be a finite sequence of
ultra-metric spaces; we assume that all $(X_{i},d_{i})$ are separable and
that all balls are compact. Let $(X,d)$ be their Cartesian product: $%
X=X_{1}\times ...\times X_{n}$ and, for $x=(x_{i})\in X$ and $y=(y_{i})\in Y$%
, we set 
\begin{equation*}
d(x,y)=\max \left\{ d_{i}(x_{i},y_{i}):i=1,2,...,n\right\} .
\end{equation*}%
Thus $(X,d)$ is a separable ultra-metric space, all balls in $\left(
X,d\right) $ are compact, and, moreover, each $d$-ball $B_{r}(a)$ in $X$ is
a product of $d_{i}$-balls $B_{r}^{i}(a_{i})$ in $X_{i}$ of the same radius.

Given a Radon measure $\mu _{i}$ on each %
$(X_{i},d_{i})$ we define $\mu =\bigotimes \mu _{i}$ on $(X,d)$. Let $%
\mathcal{V}_{c}$ be the set of all compactly supported locally constant
functions on $(X,d)$.

Consider the ultra-metric measure space $(X,d,\mu ).$ According to the
previous sections, there exists a rich class of isotropic Markov semigroups
and corresponding Laplacians on $(X,d,\mu )$ as constructed in (\ref%
{Averager}) -- (\ref{convexcomb-Pt}). Thanks to the product structure of $%
(X,d,\mu )$ one can define in a natural way a non-trivial and interesting
class of Markov semigroups and Laplacians which are \emph{not} isotropic.
Namely, choosing on each $(X_{i},d_{i},\mu _{i})$ an isotropic Markov
semigroup $\{P_{i}^{t}\}$, we define a Markov semigroup $\{P_{t}\}$ on $%
(X,d,\mu )$ as the tensor product of the $\{P_{i}^{t}\}$, 
\begin{equation*}
P_{t}=\bigotimes_{i=1}^{n}P_{i}^{t}.
\end{equation*}%
The semigroup $\{P_{t}\}$ has the following heat kernel:%
\begin{equation*}
p(t,x,y)=\prod\limits_{i=1}^{n}p_{i}(t,x_{i},y_{i}),
\end{equation*}%
where $p_{i}$ is the heat kernel of $\left\{ P_{i}^{t}\right\} $.

The generator $\mathcal{L}$ of $P_{t}$ can be described as follows: $%
\mathcal{V}_{c}\subset \func{dom}_{\mathcal{L}}$ and for any $f\in \mathcal{V%
}_{c}$ we have%
\begin{equation}
\mathcal{L}f(x)=\sum_{i=1}^{n}\mathcal{L}_{i}f(x)  \label{lap-sum}
\end{equation}%
where $x=\left( x_{1},...,x_{n}\right) $ and $\mathcal{L}_{i}$ acts on $%
x_{i} $. It follows that 
\begin{equation*}
\mathcal{L}f(x)=\int_{X}\left( f(x)-f(y)\right) \,J(x,dy)
\end{equation*}%
where%
\begin{equation*}
J(x,dy)=\sum_{i=1}^{n}J_{i}(x_{i},y_{i})\,d\mu _{i}(y_{i})\,,
\end{equation*}%
and $J_{i}(x_{i},y_{i})$ is the jump kernel of $\mathcal{L}_{i}.$

In particular, we see that for each $x\in X$ the measures $J(x,dy)$ and $\mu
(dy)$ are not necessarily mutually absolutely continuous (in the case when
at least one of $X_{l}$ is perfect, $J(x,dy)$ is singular with respect to $%
\mu $), which implies that the semigroup $\left\{ P_{t}\right\} $ is not
necessarily an isotropic Markov semigroup.

In this paper we do not intend to develop a general theory on product
spaces. Our aim is to study in detail two specific examples related to $p$%
-adic analysis.

In the first example we consider the Vladimirov Laplacian that matches well
the above general construction. In the second example we consider the
Taibleson Laplacian defined in terms of the multidimensional Riesz kernels,
see Taibleson~\cite{TAI75} and Rodriguez-Vega and Zuniga-Galindo~\cite{ZuGa2}%
. We show that the Taibleson Laplacian is isotropic. This will allow us to
improve the heat kernel bounds from \cite{ZuGa2} and to obtain some new
results (transience/recurrence, independence on $1\leq p<\infty $ of the $%
L^{p}$-spectrum, precise bounds of the moments of the corresponding Markov
process etc.)

Consider the linear space $\mathbb{Q}_{p}^{n}=\mathbb{Q}_{p}\times ...\times 
\mathbb{Q}_{p}\ $over the field $\mathbb{Q}_{p}$ and define in $\mathbb{Q}%
_{p}^{n}$ a norm 
\begin{equation}
\left\Vert z\right\Vert _{p}=\max \left\{ \left\Vert z_{i}\right\Vert
_{p}:i=1,2,...,n\right\} .  \label{norm}
\end{equation}%
It clearly satisfies the ultra-metric triangle inequality (\ref{um}) and is
homogeneous in the following sense:%
\begin{equation*}
\left\Vert az\right\Vert _{p}=\left\Vert a\right\Vert _{p}\left\Vert
z\right\Vert _{p},\quad \text{for all}\quad \text{ }a\in \mathbb{Q}_{p},\
z\in \mathbb{Q}_{p}^{n}.
\end{equation*}%
Set 
\begin{equation*}
d_{p}\left( x,y\right) =\left\Vert x-y\right\Vert _{p}
\end{equation*}
so that $(\mathbb{Q}_{p}^{n},d_{p})$ is an ultra-metric space.

Let $\mu _{p}=\bigotimes \mu _{p,i}$ be the additive Haar measure on the
Abelian group $\mathbb{Q}_{p}^{n}.$ As before, let $\mathcal{V}_{c}$ be the
set of all compactly supported locally constant functions on the
ultra-metric space $(\mathbb{Q}_{p}^{n},d_{p})$. Recall that $\mathcal{V}%
_{c} $ is a dense subset in $L^{2}=L^{2}(\mathbb{Q}_{p}^{n},\mu _{p}).$

\subsubsection{The Vladimirov Laplacian}

\label{Vladimirov}For any given $n$-tuple $\alpha =(\alpha _{1}\,,\dots
,\alpha _{n})$ with entries $\alpha _{i}>0$ we define the ultra-metric%
\begin{equation*}
d_{p,\alpha }(x,y)=\max \left\{ \left\Vert x_{i}-y_{i}\right\Vert
_{p}^{\alpha _{i}}:i=1,2,...,n\right\} .
\end{equation*}%
In particular, the ultra-metric $d_{p}\left( x,y\right) $ defined above
corresponds to the case $\alpha =\left( 1,...,1\right) $. The identity map%
\begin{equation*}
\left( \mathbb{Q}_{p}^{n},d_{p,\alpha }\right) \rightarrow \left( \mathbb{Q}%
_{p}^{n},d_{p}\right)
\end{equation*}%
is a homeomorphism, but not bi-Lipschitz, unless $\alpha _{i}=1$ for all $i$%
. This fact plays an essential role in the study of the class of Laplacians
introduced next as a special instance of (\ref{lap-sum}).

\begin{definition}
\RM\label{alpha-Vladimirov} Let $\alpha =(\alpha _{1},\dots ,\alpha _{n})$.
For any function $f\in \mathcal{V}_{c}$ we define the operator%
\begin{equation*}
\mathfrak{V}^{\alpha }f(x)=\sum_{i=1}^{n}\mathfrak{D}_{x_{i}}^{\alpha
_{i}}f(x).
\end{equation*}%
\ where $x=(x_{1},x_{2},...,x_{n})$ and $\mathfrak{D}_{x_{i}}^{\alpha _{i}}$
is the $p$-adic fractional derivative of order $\alpha _{i}$ acting on $%
x_{i} $.
\end{definition}

The operator $\mathfrak{V}^{\alpha }$ on $\mathbb{Q}_{p}^{3}$ with $\alpha
=(2,2,2)$, was introduced by Vladimirov~\cite{Vladimirov} as an analogue of
the classical Laplace operator in $\mathbb{R}^{3}.$ This operator, which we
denote briefly by $\mathfrak{V}^{2}$, is translation invariant and
homogeneous, that is,%
\begin{equation*}
\mathfrak{V}^{2}\tau _{y}(f)=\tau _{y}(\mathfrak{V}^{2}f),\quad \text{where}%
\quad \tau _{y}f(x)=f(x+y).
\end{equation*}%
and%
\begin{equation*}
\mathfrak{V}^{2}\theta _{a}(f)=\left\Vert a\right\Vert _{p}^{2}\theta _{a}(%
\mathfrak{V}^{2}f),\quad \text{where}\quad \theta
_{a}f(x)=f(ax_{1},ax_{2},ax_{3}).
\end{equation*}%
It follows that the Green function $g(x,y)$ of the operator $\mathfrak{V}%
^{2} $ on $\mathbb{Q}_{p}^{3}$ is also translation invariant and homogeneous:%
\begin{equation*}
g(x,y)=g(x-z,y-z)\quad \text{and\ \ }g(ax,ay)=g(x,y)/\left\Vert a\right\Vert
_{p}\,,\;a\in \mathbb{Q}_{p}\,.
\end{equation*}%
In particular, setting $\mathfrak{E}(x)=g(x,0)$ , we obtain for all non-zero 
$a\in \mathbb{Q}_{p}$ the identity 
\begin{equation*}
\mathfrak{E}(a,a,a)=\frac{\mathfrak{E}(1,1,1)}{\left\Vert a\right\Vert _{p}}.
\end{equation*}%
This identity was observed in \cite{Vladimirov}. It gives an idea of how the
Green function of the operator $\mathfrak{V}^{2}$ (in Vladimirov's
terminology, the fundamental solution of the equation $\mathfrak{V}^{2}\,%
\mathfrak{E}=\delta $) behaves at infinity/at zero. Below, in Proposition %
\ref{homogeneous Laplacian}, we will prove that, for all non-zero $%
a=(a_{1},a_{2},a_{3})\in \mathbb{Q}_{p}^{3}$, 
\begin{equation}
\mathfrak{E}(a_{1},a_{2},a_{3})\simeq \frac{1}{\left\Vert a\right\Vert _{p}}.
\label{Asymptotic-Vladimirov}
\end{equation}

In fact, we shall prove similar estimate for more general operators $%
\mathfrak{V}^{\alpha }\ $without the homogeneity property. We start by
listing some properties of the operator $(\mathfrak{V}^{\alpha },\mathcal{V}%
_{c}\mathcal{)}$ from Definition \ref{alpha-Vladimirov} which follow
directly from the corresponding properties of the \textquotedblleft
one-dimensional Laplacians\textquotedblright\ $\mathfrak{D}^{\alpha _{i}}.$

\begin{enumerate}
\item $(\mathfrak{V}^{\alpha },\mathcal{V}_{c}\mathcal{)}$ is a non-negative
definite symmetric operator.

\item $(\mathfrak{V}^{\alpha },\mathcal{V}_{c}\mathcal{)}$ admits a complete
system of compactly supported eigenfunctions. In particular, the operator $(%
\mathfrak{V}^{\alpha },\mathcal{V}_{c})$ is essentially self-adjoint.

\item The semigroup $\exp (-t\,\mathfrak{V}^{\alpha })$ is symmetric and
Markovian. It admits the heat kernel $p_{\alpha }(t,x,y)$ which has the
following form 
\begin{equation*}
p_{\alpha }(t,x,y)=\prod_{i=1}^{n}p_{\alpha _{i}}(t,x_{i},y_{i}).
\end{equation*}

\item The semigroup $\exp (-t\,\mathfrak{V}^{\alpha })$ is transient if and
only if $A>1$.

\item For all $f\in \mathcal{V}_{c}$ 
\begin{equation*}
\mathfrak{V}^{\alpha }f(x)=\int_{\mathbb{Q}_{p}^{n}}\left( f(x)-f(y)\right)
\,J_{\alpha }(x,dy)
\end{equation*}%
where 
\begin{equation*}
J_{\alpha }(x,dy)=\sum_{i_{1}}^{n}J_{\alpha _{i}}(x_{i}-y_{i})\,d\mu
_{p,i}(y_{i})
\end{equation*}%
and%
\begin{equation*}
J_{\alpha _{i}}(x_{i}-y_{i})=\frac{p^{\alpha _{i}}-1}{1-p^{-\alpha _{i}-1}}%
\frac{1}{\left\Vert x_{i}-y_{i}\right\Vert _{p}^{1+\alpha _{i}}}\,.
\end{equation*}
\end{enumerate}

In particular, the semigroup $\exp (-t\,\mathfrak{V}^{\alpha })$ is in
general not an isotropic Markov semigroup.

Observe that thanks to the group structure of $\,\mathbb{Q}_{p}^{n}$, the
functions $(x,y)\mapsto p_{\alpha }(t,x,y)$ and $(x,y)\mapsto g_{\alpha
}(x,y)$ are translation invariant. Hence, setting%
\begin{equation*}
p_{\alpha }(t,z)=p_{\alpha }(t,z,0)\quad \text{and\ \ }g_{\alpha
}(z)=g_{\alpha }(z,0)\,,
\end{equation*}%
we obtain%
\begin{equation*}
p_{\alpha }(t,x,y)=p_{\alpha }(t,x-y)\quad \text{and\ \ }g_{\alpha
}(x,y)=g_{\alpha }(x-y).
\end{equation*}

\begin{proposition}
\label{product-heat kernel}Set%
\begin{equation*}
A=\sum_{i=1}^{n}\frac{1}{\alpha _{i}}.
\end{equation*}%
Then the heat kernel satisfies the following estimate 
\begin{equation}
p_{\alpha }(t,z)\simeq t^{-A}\prod_{i=1}^{n}\min \left\{ 1,\frac{%
t^{1+1/\alpha _{i}}}{\left\Vert z_{i}\right\Vert _{p}^{1+\alpha _{i}}}%
\right\}  \label{pproduct}
\end{equation}%
uniformly for all $t>0$ and $z\in \mathbb{Q}_{p}^{n}\,$. 
In particular, for all $t>\left\Vert z\right\Vert _{p,\alpha }\,$,%
\begin{equation}
p_{\alpha }(t,z)\simeq \,t^{-A}  \label{products-heat kernel equivalence}
\end{equation}
\end{proposition}

\begin{proof}
By Theorem \ref{Vladimirov heat kernel} we have%
\begin{equation*}
p_{\alpha _{i}}(t,z_{i})\simeq \frac{t}{\left( t^{1/\alpha _{i}}+\left\Vert
z_{i}\right\Vert _{p}\right) ^{1+\alpha _{i}}}\simeq \frac{1}{t^{1/\alpha
_{i}}}\,\min \left\{ 1,\frac{t^{1+1/\alpha _{i}}}{\left\Vert
z_{i}\right\Vert _{p}^{1+\alpha _{i}}}\right\} ,
\end{equation*}%
whence the claim follows.
\end{proof}

\begin{proposition}
\label{product-Green function} The semigroup $\exp (-t\,\mathfrak{V}^{\alpha
})$ is transient if and only if $A>1.$ If $A>1$ then, for all $z\in \mathbb{Q%
}_{p}^{n}$ and some $C_{1}>0,$%
\begin{equation*}
g_{\alpha }(z)\geq C_{1}\left( \frac{1}{\left\Vert z\right\Vert _{p,\alpha }}%
\right) ^{A-1}\!\!.
\end{equation*}%
For any $\kappa >0$, we define the set 
\begin{equation*}
\Omega (\kappa )=\left\{ x\in \mathbb{Q}_{p}^{n}:\max_{i}\left\{ \left\Vert
x_{i}\right\Vert _{p}^{\alpha _{i}}\right\} \leq \kappa \,\min_{i}\left\{
\left\Vert x_{i}\right\Vert _{p}^{\alpha _{i}}\right\} \right\} .
\end{equation*}%
Then, for all $z\in \Omega (\kappa )$ and some constant $C_{2}>0$ which
depends on $\kappa ,$ 
\begin{equation*}
g_{\alpha }(z)\leq C_{2}\left( \frac{1}{\left\Vert z\right\Vert _{p,\alpha }}%
\right) ^{A-1}\!\!.
\end{equation*}
\end{proposition}

\begin{proof}
The transience criterion $A>1$ follows from $p_{a}\left( t,x,x\right) \simeq
t^{-A}$. To prove the lower bound, we use (\ref{products-heat kernel
equivalence}) and write%
\begin{equation*}
g_{\alpha }(z)=\int_{0}^{\infty }p_{\alpha }(t,z)\,dt\geq \int_{\left\Vert
z\right\Vert _{p,\alpha }}^{\infty }p_{\alpha }(t,z)\,dt\geq
C_{1}\int_{\left\Vert z\right\Vert _{p,\alpha }}^{\infty
}t^{-A}\,dt=c_{1}\left( \frac{1}{\left\Vert z\right\Vert _{p,\alpha }}%
\right) ^{A-1}\!\!.
\end{equation*}%
On the other hand we have%
\begin{equation*}
g_{\alpha }(z)=\left( \int_{0}^{\left\Vert z\right\Vert _{p,\alpha
}}+\int_{\left\Vert z\right\Vert _{p,\alpha }}^{\infty }\right) p_{\alpha
}(t,z)\,dt=:\mathit{I}+\mathit{II}.
\end{equation*}%
To estimate the second term $\mathit{II}$, we use again (\ref{products-heat
kernel equivalence}):%
\begin{equation*}
\mathit{II}\simeq \int_{\left\Vert z\right\Vert _{p,\alpha }}^{\infty
}t^{-A}\,dt\simeq \left( \frac{1}{\left\Vert z\right\Vert _{p,\alpha }}%
\right) ^{A-1}\!\!.
\end{equation*}%
To estimate the first term we use (\ref{pproduct}): 
\begin{eqnarray*}
\mathit{I} &\leq &c\int_{0}^{\left\Vert z\right\Vert _{p,\alpha
}}t^{-A}\prod_{i=1}^{n}\frac{t^{1+1/\alpha _{i}}}{\left\Vert
z_{i}\right\Vert _{p}^{1+\alpha _{i}}}\,dt \\
&=&c\int_{0}^{\left\Vert z\right\Vert _{p,\alpha }}\prod_{i=1}^{n}\frac{1}{%
\left\Vert z_{i}\right\Vert _{p}^{1+\alpha _{i}}}t^{n}\,dt=c^{\prime
}\prod_{i=1}^{n}\frac{1}{\left\Vert z_{i}\right\Vert _{p}^{1+\alpha _{i}}}%
\left\Vert z\right\Vert _{p,\alpha }^{n+1}.
\end{eqnarray*}%
When $z\in \Omega (\kappa ),$ we obtain 
\begin{eqnarray*}
\mathit{I} &\leq &c^{\prime \prime }\prod_{i=1}^{n}\frac{1}{\left\Vert
z_{i}\right\Vert _{p}^{1+\alpha _{i}}}\left( \min \left\{ \left\Vert
z_{i}\right\Vert _{p}^{\alpha _{i}}\right\} \right) ^{n+1}\leq c^{\prime
\prime }\min \left\{ \left\Vert z_{i}\right\Vert _{p}^{\alpha _{i}}\right\}
\prod_{i=1}^{n}\frac{1}{\left\Vert z_{i}\right\Vert _{p}} \\
&=&c^{\prime \prime }\min \left\{ \left\Vert z_{i}\right\Vert _{p}^{\alpha
_{i}}\right\} \prod_{i=1}^{n}\frac{1}{(\left\Vert z_{i}\right\Vert
_{p}^{\alpha _{i}})^{1/\alpha _{i}}}.
\end{eqnarray*}%
Next, 
\begin{equation*}
\prod_{i=1}^{n}\frac{1}{(\left\Vert z_{i}\right\Vert _{p}^{\alpha
_{i}})^{1/\alpha _{i}}}\leq \prod_{i=1}^{n}\frac{1}{\left( \min \left\{
\left\Vert z_{j}\right\Vert _{p}^{\alpha _{j}}\right\} \right) ^{1/\alpha
_{i}}}=\left( \frac{1}{\min \left\{ \left\Vert z_{j}\right\Vert _{p}^{\alpha
_{j}}\right\} }\right) ^{A},
\end{equation*}%
whence%
\begin{equation*}
\mathit{I}\leq c^{\prime \prime }\left( \frac{1}{\min \left\{ \left\Vert
z_{j}\right\Vert _{p}^{\alpha _{j}}\right\} }\right) ^{A-1}\!\!.
\end{equation*}%
Again using the fact that $z\in \Omega (\kappa )$, we write 
\begin{equation*}
\left( \frac{1}{\min \left\{ \left\Vert z_{j}\right\Vert _{p}^{\alpha
_{j}}\right\} }\right) ^{A-1}\leq \left( \frac{\kappa }{\max \left\{
\left\Vert z_{j}\right\Vert _{p}^{\alpha _{j}}\right\} }\right)
^{A-1}=c(\kappa )\left( \frac{1}{\left\Vert z\right\Vert _{p,\alpha }}%
\right) ^{A-1}.
\end{equation*}%
The obtained upper bounds on the integrals $\mathit{I}$ and $\mathit{II}$
imply the desired upper bound for $g_{\alpha }(z).$~
\end{proof}

\begin{proposition}
\label{homogeneous Laplacian} Let $\alpha =(\alpha _{1}\,,\dots ,\alpha
_{n})=(\beta ,\dots ,\beta )$ be an $n$-tuple having all entries equal to $%
\beta .$ Assume that $(n-1)/2<\beta <n.$ Then the semigroup $\exp (-t\,%
\mathfrak{V}^{\alpha })$ is transient and the Green function $g_{\alpha }(z)$
satisfies the estimates%
\begin{equation}
g_{\alpha }(z)\simeq \left( \frac{1}{\left\Vert z\right\Vert _{p,\alpha }}%
\right) ^{A-1},  \label{homogeneous Green-Vladimirov-1}
\end{equation}%
for all $z\in \mathbb{Q}_{p}^{n}$ and some $c_{1},c_{2}>0.$
\end{proposition}

Since $A=\frac{n}{\beta }$ and $\left\Vert z\right\Vert _{p,\alpha
}=\left\Vert z\right\Vert _{p}^{\beta }$, the estimate (\ref{homogeneous
Green-Vladimirov-1}) is equivalent to 
\begin{equation}
g_{\alpha }(z)\simeq \left( \frac{1}{\left\Vert z\right\Vert _{p}}\right)
^{n-\beta }.  \label{homogeneous Green-Vladimirov-2}
\end{equation}

\begin{proof}
Transience follows from Proposition \ref{product-Green function} because $%
A=n/\beta >1.$ The same Proposition yields the desired lower bound of the
Green function. To prove the upper bound, we observe that the Laplacian $%
\mathfrak{V}^{\alpha }$ is homogeneous, that is%
\begin{equation*}
\mathfrak{V}^{\alpha }\circ \theta _{a}=\left\Vert a\right\Vert _{p}^{\beta
}\cdot \theta _{a}\circ \mathfrak{V}^{\alpha },
\end{equation*}%
for all $a\in \mathbb{Q}_{p}.$ This implies that also the Green function $%
g_{\alpha }(z)$ is homogeneous, that is%
\begin{equation*}
g_{\alpha }(az)=\left\Vert a\right\Vert _{p}^{n-\beta }g_{\alpha }(z),
\end{equation*}%
for all $a\in \mathbb{Q}_{p}$ and $z\in \mathbb{Q}_{p}^{n}.$

Without loss of generality assume that $\left\Vert z\right\Vert _{p,\alpha
}=\left\Vert z_{1}\right\Vert _{p}^{\beta }>0.$ Then%
\begin{eqnarray*}
g_{\alpha }(z) &=&g_{\alpha }\left(
z_{1}(1,z_{2}/z_{1},...,z_{n}/z_{1})\right) =\left\Vert z_{1}\right\Vert
_{p}^{n-\beta }g_{\alpha }(1,z_{2}/z_{1},...,z_{n}/z_{1}) \\
&=&\left( \frac{1}{\left\Vert z\right\Vert _{p,\alpha }}\right)
^{A-1}g_{\alpha }(1,z_{2}/z_{1},...,z_{n}/z_{1}) \\
&\leq &\left( \frac{1}{\left\Vert z\right\Vert _{p,\alpha }}\right)
^{A-1}\sup \left\{ g_{\alpha }(1,x_{2},...,x_{n}):x_{i}\in \mathbb{Z}%
_{p}\right\} .
\end{eqnarray*}%
Next we apply our assumption $\beta >(n-1)/2$ and obtain from (\ref{pproduct}%
)%
\begin{eqnarray*}
g_{\alpha }(1,x_{2},...,x_{n}) &=&\int_{0}^{\infty }p_{\alpha
}(t,(1,x_{2},...,x_{n})\,dt \\
&=&\left( \int_{0}^{1}+\int_{1}^{\infty }\right) p_{\alpha
}(t,(1,x_{2},...,x_{n})\,dt \\
&\leq &c\int_{0}^{1}t^{-\frac{n}{\beta }}\,t^{1+\frac{1}{\beta }%
}\,dt+c^{\prime }\int_{1}^{\infty }t^{-\frac{n}{\beta }}\,dt=c_{2}<\infty ,
\end{eqnarray*}%
which implies the desired upper bound.
\end{proof}

\subsubsection{The Taibleson Laplacian}

The Fourier transform $\mathcal{F}:f\mapsto \widehat{f}$ of a function $f$
on the locally compact Abelian group $\mathbb{Q}_{p}^{n}$ is defined by

\begin{equation*}
\widehat{f}(\theta )=\int_{Q_{p}^{n}}\left\langle x,\theta \right\rangle
f(x)d\mu _{p}^{n}(x),
\end{equation*}%
where $x=\left( x_{1},...,x_{n}\right) \in \mathbb{Q}_{p}^{n}$, $\theta
=\left( \theta _{1},...,\theta _{n}\right) \in (\mathbb{Q}_{p}^{n})^{\ast }=%
\mathbb{Q}_{p}^{n},$ 
\begin{equation*}
\left\langle x,\theta \right\rangle =\prod_{k=1}^{n}\left\langle
x_{k},\theta _{k}\right\rangle ,
\end{equation*}%
and $d\mu _{p}^{n}(x)=d\mu _{p}(x_{1})...d\mu _{p}(x_{n})$ is the Haar
measure on $\mathbb{Q}_{p}^{n}$. It is known that $\mathcal{F}$ is a linear
isomorphism from $\mathcal{V}_{c}$ onto itself, which justifies the
following Definition (compare with Definition \ref{def-frac-derivative}).

\begin{definition}
\label{Taibleson operator}The Taibleson operator $\mathfrak{T}^{\alpha }$
for $\alpha >0$ is defined on functions $f\in \mathcal{V}_{c}$ by%
\begin{equation*}
\widehat{\mathfrak{T}^{\alpha }f}(\zeta )=\left\Vert \zeta \right\Vert
_{p}^{\alpha }\widehat{f}(\zeta ),\quad \zeta \in \mathbb{Q}_{p}^{n}.
\end{equation*}
\end{definition}

It follows that $(\mathfrak{T}^{\alpha },\mathcal{V}_{c})$ is an essentially
self-adjoint and non-negative definite operator in $L^{2}.$ This operator
was introduced by Taibleson~\cite{TAI75}, and the associated semigroup $\exp
(-t\,\mathfrak{T}^{\alpha })$ was studied by Rodriguez-Vega and
Zuniga-Galindo~\cite{ZuGa2}. In particular, it was shown that%
\begin{equation}
\mathfrak{T}^{\alpha }f(x)=\frac{p^{\alpha }-1}{1-p^{-\alpha -n}}\int_{%
\mathbb{Q}_{p}^{n}}\frac{f(x)-f(y)}{\left\Vert x-y\right\Vert _{p}^{\alpha
+n}}\,d\mu _{p}^{n}(y).  \label{Taibleson Laplacian}
\end{equation}%
The equation (\ref{Taibleson Laplacian}) implies that the operator $(-%
\mathfrak{T}^{\alpha },\mathcal{V}_{c})$ satisfies the $\max $-principle,
whence its semigroup is Markovian. Our aim is to show that $\exp (-t\,%
\mathfrak{T}^{\alpha })$ is an isotropic Markov semigroup on the
ultra-metric measure space $(\mathbb{Q}_{p}^{n},d_{p},\mu _{p}^{n})$.

Our first observation is that the spectrum of the symmetric operator $(%
\mathfrak{T}^{\alpha },\mathcal{V}_{c})$ coincides with the range of the
function $\zeta \mapsto \left\Vert \zeta \right\Vert _{p}^{\alpha },$ 
\begin{equation*}
\func{spec}\mathfrak{T}^{\alpha }=\{p^{k\alpha }:k\in \mathbb{Z}\}\cup \{0\}.
\end{equation*}%
The eigenspace $\mathcal{H}(\lambda )$ of the operator $(\mathfrak{T}%
^{\alpha },\mathcal{V}_{c})$ corresponding to the eigenvalue $\lambda
=p^{k\alpha },$ is spanned by the functions 
\begin{equation*}
f_{k}=\frac{1}{\mu _{p}^{n}(p^{k}\mathbb{Z}_{p}^{n})}\mathbf{1}_{p^{k}%
\mathbb{Z}_{p}^{n}}-\frac{1}{\mu _{p}^{n}(p^{k-1}\mathbb{Z}_{p}^{n})}\mathbf{%
1}_{p^{k-1}\mathbb{Z}_{p}^{n}}
\end{equation*}%
and\ all its\ shifts\ $f_{k}(\cdot +a)$ with $a\in \left. \mathbb{Q}%
_{p}^{n}\right/ p^{k}\mathbb{Z}_{p}^{n}$. Indeed, computing the Fourier
transform of the function $f_{k},$%
\begin{equation*}
\widehat{f_{k}}(\zeta )=\mathbf{1}_{\{\left\Vert \zeta \right\Vert _{p}\leq
p^{k}\}}-\mathbf{1}_{\{\left\Vert \zeta \right\Vert _{p}\leq p^{k-1}\}}=%
\mathbf{1}_{\{\left\Vert \zeta \right\Vert _{p}=p^{k}\}},
\end{equation*}%
we obtain%
\begin{equation*}
\widehat{\mathfrak{T}^{\alpha }f_{k}}(\zeta )=\left\Vert \zeta \right\Vert
_{p}^{\alpha }\widehat{f_{k}}(\zeta )=p^{k\alpha }\widehat{f_{k}}(\zeta ).
\end{equation*}%
All the above shows that the operator $\mathfrak{T}^{\alpha }$ coincides
with an isotropic Laplacian $\mathcal{L}_{\alpha }$ on $(\mathbb{Q}%
_{p}^{n},d_{p},\mu _{p}^{n})$ associated with the distance distribution
function 
\begin{equation*}
\sigma _{\alpha }\left( r\right) =\exp \left( -\left( \frac{p}{r}\right)
^{\alpha }\right) ,
\end{equation*}%
and the semigroup $\exp (-t\,\mathfrak{T}^{\alpha })$ coincides with the
isotropic semigroup $\left\{ P_{\alpha }^{t}\right\} $.

Observe that the associated intrinsic ultra-metric is%
\begin{equation*}
d_{p\ast }(x,y)=\left( \frac{\left\Vert x-y\right\Vert _{p}}{p}\right)
^{\alpha }.
\end{equation*}%
The spectral distribution function $N_{\alpha }(x,\tau )=N_{\alpha }(\tau )$
is the non-decreasing, left-continuous staircase function which has jumps at
the points $\tau _{k}=p^{k\alpha }$, $k\in \mathbb{Z},$ and takes values $%
N_{\alpha }(\tau _{k})=p^{(k-1)n}$ at these points. It follows that%
\begin{equation*}
N_{\alpha }(\tau )\simeq \tau ^{n/\alpha }.
\end{equation*}%
In particular, $\tau \mapsto N_{\alpha }(\tau )$ is a doubling function, and
Theorem \ref{heat-k-doubling} implies the following result.

\begin{theorem}
\label{Taibleson heat kernel} The semigroup $\exp (-t\,\mathfrak{T}^{\alpha
})$ on $\mathbb{Q}_{p}^{n}$ admits a continuous heat kernel $p_{\alpha
}(t,x,y)$ that satisfies the estimate%
\begin{equation}
p_{\alpha }(t,x,y)\simeq \frac{\,t}{\left( t^{1/\alpha }+\left\Vert
x-y\right\Vert _{p}\right) ^{n+\alpha }}\,,  \label{pat}
\end{equation}%
In particular, the semigroup $\exp (-t\,\mathfrak{T}^{\alpha })$ is
transient if and only if $\alpha <n.$ In the transient case, the Green
function (=Taibleson's Riesz kernel) satisfies the identity%
\begin{equation*}
g_{\alpha }(x,y)=\frac{1-p^{-\alpha }}{1-p^{\alpha -n}}\,\frac{1}{\left\Vert
x-y\right\Vert _{p}^{n-\alpha }}.
\end{equation*}
\end{theorem}

Note that the upper bound in (\ref{pat}) was proved in \cite{ZuGa2}.

Definition \ref{Isotropic semigroup} of a rotation invariant Laplacian on $%
\mathbb{Q}_{p}$ can be carried over to $\mathbb{Q}_{p}^{n}$. The Taibleson
operator $\mathfrak{T}^{\alpha }$ is an example of a rotation invariant
Laplacian. Theorem \ref{Isotropic Laplacian}, Corollary \ref%
{Isotropic-Monotone Class} and Corollary \ref{Our Class} and their proofs
remain valid also for $\mathbb{Q}_{p}^{n}$. Here we provide a short proof of
a slightly weaker result that is of significance for us. Set $\mathfrak{T}=%
\mathfrak{T}^{1}$.

\begin{theorem}
\label{Multidimensional Polya}The equation $(\mathcal{L},\mathcal{V}%
_{c})=(\psi (\mathfrak{T}),\mathcal{V}_{c})$, where $\psi $ is an arbitrary
increasing bijection $[0\,,\,\infty )\rightarrow \lbrack 0\,,\,\infty )$,
gives a complete description of the class of isotropic Laplacians on the
ultra-metric measure space $(\mathbb{Q}_{p}^{n},d_{p},\mu _{p}^{n})$.
\end{theorem}

\begin{proof}
Let $\psi :[0,\infty )\mapsto \lbrack 0,\infty )$ be an increasing
bijection. By Theorem \ref{subord}, the operator $\ (\psi (\mathfrak{T}),%
\mathcal{V}_{c})$ is an isotropic Laplacian.

Conversely, let $(\mathcal{L},\mathcal{V}_{c})$ be an isotropic Laplacian on 
$(\mathbb{Q}_{p}^{n},d_{p},\mu _{p}^{n})$. Let $d_{p\ast }$ be the intrinsic
distance associated with $\mathcal{L}$. By construction, $d_{p\ast }$ is an
increasing function of $d_{p}\,$, see (\ref{d*}). Since the range of $d_{p}$
is the set $\{p^{k}:k\in \mathbb{Z}\}\cup \{0\}$, one can choose an
increasing bijection $\varphi :[0\,,\,\infty )\rightarrow \lbrack
0\,,\,\infty )$ such that $d_{p\ast }=\phi (d_{p}).$ Let $\lambda (B)$ and $%
\tau (B)$ be the eigenvalues of $(\mathcal{L},\mathcal{V}_{c})$ and $(%
\mathfrak{T},\mathcal{V}_{c})$, respectively, corresponding to the ball $%
B\subset \mathbb{Q}_{p}^{n}\,$. Since the intrinsic distance associated with 
$\mathfrak{T}$ is $p^{-1}d_{p}\,$, we get 
\begin{eqnarray*}
\lambda (B) &=&\frac{1}{\mathrm{diam}_{p\ast }(B)}=\frac{1}{\varphi \left( 
\mathrm{diam}_{p}(B)\right) } \\
&=&\frac{1}{\varphi \left( p/\tau (B)\right) }=:\psi \left( \tau (B)\right) ,
\end{eqnarray*}%
where $\psi (s)=1/\phi (p/s)$, an increasing bijection of $[0\,,\,\infty )$
onto itself.

Since both $(\mathcal{L},\mathcal{V}_{c})$ and $(\psi (\mathfrak{T}),%
\mathcal{V}_{c})$ are isotropic Laplacians defined on the ultra-metric
measure space $(\mathbb{Q}_{p}^{n},d_{p},\mu _{p}^{n})$ whose sets of
eigenvalues coincide, we get 
\begin{equation*}
(\mathcal{L},\mathcal{V}_{c})=(\psi (\mathfrak{T}),\mathcal{V}_{c}),
\end{equation*}%
or equivalently, in terms of the Fourier transform,%
\begin{equation*}
\widehat{\mathcal{L}f}(\zeta )=\psi (\left\Vert \zeta \right\Vert _{p})\,%
\widehat{f}(\zeta ),\text{ }
\end{equation*}%
for all $f\in \mathcal{V}_{c}$ and $\zeta \in \mathbb{Q}_{p}^{n}\,$, which
finishes the proof.
\end{proof}

\section{Random walks on a tree and jump processes on its boundary}

\label{trees}\setcounter{equation}{0}

\subsection{Rooted trees and their boundaries}

A tree is a connected graph $T$ without cycles (closed paths of length $\geq
3$). We tacitly identify $T$ with its vertex set, which is assumed to be
infinite. We write $u\sim v$ if $u,v\in T$ are neighbours. For any pair of
vertices $u,v\in T$, there is a unique shortest path, called \emph{geodesic
segment} 
\begin{equation*}
\pi (u,v)=[u=v_{0}\,,v_{1}\,,\dots ,v_{k}=v]
\end{equation*}%
such that $v_{i-1}\sim v_{i}$ and all $v_{i}$ are disctinct. If $u=v$ then
this is the \emph{empty} or \emph{trivial} path. The number $k$ is the \emph{%
length} of the path (the graph distance between $u$ and $v$). In $T$ we
choose and fix a \emph{root vertex} $o$. We write $|v|$ for the length of $%
\pi (o,v)$. The choice of the root induces a partial order on $T$, where $%
u\leq v$ when $u\in \pi (o,v)$. Every $v\in T\setminus \{o\}$ has a unique 
\emph{predecessor} $v^{-}=v_{o}^{-}$ with respect to $o$, which is the
unique neighbour of $v$ on $\pi (o,v)$. Thus, the set of all (unoriented)
edges of $T$ is 
\begin{equation*}
E(T)=\{[v^{-},v]:v\in T\,,\;v\neq o\}\,.
\end{equation*}%
For $u\in T$, the elements of the set 
\begin{equation*}
\{v\in T:v^{-}=u\}
\end{equation*}%
are the \emph{successors} of $u$, and its cardinality $\deg ^{+}(u)$ is the 
\emph{forward degree} of $u$.

In this and the next section, we assume that 
\begin{equation}  \label{degrees}
2 \le \deg^+(u) < \infty \quad \text{for every }\; u \in T\,.
\end{equation}

$$
\beginpicture 

\setcoordinatesystem units <.8mm,1.2mm> 

\setplotarea x from -10 to 70, y from -6 to 35


\setlinear

\plot 2 2  32 32 /

\plot 32 32 62 2 /

 \plot 16 16 30 2 /

 \plot 48 16 34 2 /

 \plot 8 8 14 2 /

 \plot 24 8 18 2 /

 \plot 40 8 46 2 /

 \plot 56 8 50 2 /

 \plot 4 4 6 2 /

 \plot 12 4 10 2 /

 \plot 20 4 22 2 /

 \plot 28 4 26 2 /

 \plot 36 4 38 2 /

 \plot 44 4 42 2 /

 \plot 52 4 54 2 /

 \plot 60 4 58 2 /



\setdashes <2pt> \linethickness=.5pt
\putrule from -2 -4 to 66 -4

\put {$\vdots$} at 32 0

\put {$\partial T$} [l] at -12 -4

\put {$\scriptstyle\bullet$} at 32 32
\put {$o$} [b] at 32 33.5

\put {$\scriptstyle\bullet$} at 4 4
\put {$u$} [r] at 2.5 4.5

\put {$\scriptstyle\bullet$} at 24 8
\put {$v$} [l] at 25.5 8.5

\put {$\scriptstyle\bullet$} at 16 16
\put {$u \wedge v$} [r] at 14.5 16.5

\put{\rm Figure 4} at 30 -10

\endpicture
$$

A \emph{(geodesic) ray} in $T$ is a one-sided infinite path $\pi =
[v_0\,,v_1\,, v_2\,, \dots]$ such that $v_{n-1} \sim v_n$ and all $v_n$ are
disctinct. Two rays are \emph{equivalent} if their symmetric difference (as
sets of vertices) is finite. An \emph{end} of $T$ is an equivalence class of
rays. We shall typically use letters $x$, $y$, $z$ to denote ends (and
letters $u$, $v$, $w$ for vertices). The set of all ends of $T$ is denoted $%
\partial T$. This is the \emph{boundary} at infinity of the tree. For any $u
\in T$ and $x \in \partial T$, there is a unique ray $\pi(u,x)$ which is a
representative of the end (equivalence class) $x$ and starts at $u$. We
write 
\begin{equation*}
\widehat T = T \cup \partial T.
\end{equation*}
For $u \in T$, the \emph{branch of $T$ rooted at $u$} is the subtree $T_u$
that we identify with its set of vertices 
\begin{equation}  \label{eq:branch}
T_u = \{ v \in T : u \le v\}\,,
\end{equation}
so that $T_o=T$. We write $\partial T_u$ for the set of all ends of $T$
which have a representative path contained in $T_u$, and $\widehat T_u = T_u
\cup \partial T_u\,$.

For $w, z \in \widehat T$, we define their \emph{confluent} $w \wedge z = w
\wedge_o z$ with respect to the root $o$ by the relation 
\begin{equation*}
\pi(o,w \wedge z) = \pi(o,w) \cap \pi(o,z)\,.
\end{equation*}
It is the last common element on the geodesics $\pi(o,w)$ and $\pi(o,z)$, a
vertex of $T$ unless $w=z \in \partial T$. See Figure~4.

One of the most common ways to define an ultra-metric on $\widehat{T}$ is 
\begin{equation}
d_{e}(z,w)=%
\begin{cases}
0\,, & \text{if}\;z=w\,, \\ 
e^{-|z\wedge w|}\,, & \text{if}\;z\neq w\,.%
\end{cases}
\label{eq:tree-metr}
\end{equation}%
Then $\widehat{T}$ is compact, and $T$ is open and dense. We are mostly
interested in the compact ultra-metric space $\partial T$. In the metric $%
d_{e}$ of (\ref{eq:tree-metr}), each $d_{e}$-ball with centre $x \in
\partial T$ is of the form $\partial T_{u}$ for some $x\in \pi (o,x)$.
Indeed 
\begin{equation*}
\partial T_{u}=B_{e^{-|u|}}(x) \quad \text{for every}\;o\in \pi(o,x)\,,\quad 
\text{and}\quad \Lambda _{d_{e}}(x)=\{e^{-|u|}:u\in \pi(o,u)\}\,.
\end{equation*}

\medskip

Conversely, we now start with a compact ultra-metric space $(X,d)$ that does
not possess isolated points, and construct a tree $T$ as follows: The vertex
set of $T$ is the collection 
\begin{equation*}
\mathcal{B}=\left\{ B_{r}(x):x\in X\,,\;r>0\right\}
\end{equation*}%
of all closed balls in $(X,d)$, already encountered in \S \ref{generator}.
Here, we may assume (if we wish) that $r\in \Lambda _{d}(x)$.

We now consider any ball $v=B\in \mathcal{B}$ as a vertex of a tree $T$. We
choose our root vertex as $o=X$, which belongs to $\mathcal{B}$ by
compactness. Neighborhood is given by the predecessor relation of balls, as
given by Definition \ref{predecessor}. That is, if $v=B$ then $u=B^{\prime }$
is the predecessor vertex $v^{-}$ of $v$ in the tree $T$. By compactness,
each $x$ has only finitely many successors, and since there are no isolated
points in $X$, every vertex has at least 2 successors, so that (\ref{degrees}%
) holds.

This defines the tree structure. For any $x \in X$, the collection of all
balls $B_r(x)$, $r \in \Lambda_d(x)$, ordered decreasingly, forms the set of
vertices of a ray in $T$ that starts at $o$. Via a straightforward exercise,
the mapping that associates to $x$ the end of $T$ represented by that ray is
a homeomorphism from $X$ onto $\partial T$. Thus, we can identify $X$ and $%
\partial T$ as ultra-metric spaces.

In this identification, if originally a vertex $u$ was interpreted as a ball 
$B_r(x)$, $r \in \Lambda_d(x)$, then the set $\partial T_u$ of ends of the
branch $T_u$ just coincides with the ball $B_r(x)$. That is, we are
identifying each vertex $u$ of $T$ with the set $\partial T_u$.

If we start with an arbitrary locally finite tree and take its space of ends
as the ultra-metric space $X$, then the above construction does not recover
vertices with forward degree $1$, so that in general we do not get back the
tree we started with. However, via the above construction, the
correspondence between compact ultra-metric spaces without isolated points
(perfect ultra-metric spaces) and locally finite rooted trees with forward
degrees $\geq 2$ is bijective (cf. \cite{Hughes04}).

It is well known that any ultra-metric space $X$ which is both compact and
perfect is homeomorphic to the ternary Cantor set $C\subset \lbrack 0,1]$.
When $X$ is not compact but still perfect we have a homeomorphism $X\simeq C$
$\backslash $ $\{p\},$ where $p\in C$ is any fixed point.

For the rest of this and the next section, we shall abandon the notation $X$
for compact and perfect ultra-metric space.

\smallskip

\emph{We consider $X$ as the boundary $\partial T$ of a locally finite,
rooted tree with forward degrees $\ge 2$.}

\smallskip

At the end, we shall comment on how one can handle the presence of vertices
with forward degree $1$, as well as the non-compact case.

\medskip

There are many ways to equip $\partial T$ with an ultra-metric that has the
same topology and the same compact-open balls $\partial T_x\,$, $x \in T$,
possibly with different radii than in the standard metric (\ref{eq:tree-metr}%
). The following is a kind of ultra-metric analogue of a length element.

\begin{definition}
\RM\label{def:element} Let $T$ be a locally finite, rooted tree $T$ with $%
\deg ^{+}(x)\geq 2$ for all $x$. An \emph{ultra-metric element} is a
function $\phi :T\rightarrow (0\,,\infty )$ with 
\begin{equation*}
\begin{split}
& \text{\RM(i)}\quad \phi (v^{-})>\phi (v)\quad \text{for every}\;v\in
T\setminus \{o\}\,, \\
& \text{\RM(ii)}\quad \lim \phi (v_{n})=0\quad \text{along every geodesic ray%
}\;\pi =[v_{0}\,,v_{1}\,,v_{2}\,,\dots ]\,.
\end{split}%
\end{equation*}%
It induces the ultra-metric $d_{\phi }$ on $\partial T$ given by 
\begin{equation*}
d_{\phi }(x,y)=%
\begin{cases}
0\,, & \text{if}\;x=y\,, \\ 
\phi (x\wedge y)\,, & \text{if}\;x\neq y\,.%
\end{cases}%
\end{equation*}
\end{definition}

The balls in this ultra-metric are again the sets 
\begin{equation*}
\partial T_u = B_{\phi(u)}(x)\,,\ \ x \in \partial T_u\,.
\end{equation*}
Note that condition (ii) in the definition is needed for having that each
end of $T$ is non-isolated in the metric $d_{\phi}\,$. The metric $d_e$ of (%
\ref{eq:tree-metr}) is of course induced by $\phi(x) = e^{-|x|}$.

\begin{lemma}
\label{lem:dphi} For a tree as in Definition \ref{def:element}, every
ultra-metric on $\partial T$ whose closed balls are the sets $\partial T_u\,$%
, $u \in T$, is induced by an ultra-metric element on $T$.
\end{lemma}

\begin{proof}
Given an ultra-metric $d$ as stated, we set $\phi (v)=\mathrm{diam}(\partial
T_{v})$, the diameter with respect to the metric $d$. Since $%
\deg^{+}(v^{-})\geq 2$ for any $v\in T\setminus \{o\}$, the ball $\partial
T_{v^{-}}$ is the disjoint union of at least two balls $\partial T_{u}$ with 
$u^{-}=v^{-}$. Therefore we must have $\mathrm{diam}(\partial T_{v})<\mathrm{%
diam}(\partial T_{v^{-}})$, and property (i) holds. Since no end is
isolated, $\phi $ satisfies (ii). It is now straightforward that $d_{\phi}=d 
$.
\end{proof}

In view of this correspondence, in the sequel we shall replace the subscript 
$d$ referring to the metric $d = d_{\phi}$ by the subscript $\phi$ referring
to the ultra-metric element. We note that 
\begin{equation}  \label{eq:range}
\mathrm{diam}_{\phi}(\partial T) = \phi(o)\,,\quad \Lambda_{\phi}(x) = \{
\phi(u) : u \in \pi(o,x) \} \quad\text{and}\quad \Lambda_{\phi} = \{\phi(v)
: v \in T \}.
\end{equation}
We also note here that for any $x \in \partial T$ and $v \in \pi(o,x)$, the
balls with respect to $d_{\phi}$ are 
\begin{equation}  \label{eq:balls}
B_r(x) = B^{\phi}_r(x) = 
\begin{cases}
\partial T_v\, & \text{for}\; \phi(v) \le r < \phi(v^-)\,,\; \text{if}\; v
\ne o \\ 
\partial T & \text{for}\; r \ge \phi(o)\,,\; \text{if}\; v = o\,.%
\end{cases}%
\end{equation}

\subsection{Isotropic jump processes on the boundary of a tree}

In view of the explanations given above, we can consider the isotropic jump
processes of (\ref{Averager})--(\ref{convexcomb-Pt}) on $X=\partial T$.
Since this space is compact, we may assume that the reference measure $\mu $
is a probability measure on $\partial T$. Given $\mu $, a distance
distribution $\sigma $ with properties (\ref{sigma}), and an ultra-metric
element $\phi $ on $T$, we can now refer to the $(d_{\phi },\mu ,\sigma )$%
-process simply as the $(\phi ,\mu ,\sigma )$-process on $\partial T$. We
can write the semigroup and its transition probabilities in detail as
follows. For $x\in \partial T$ and $\pi
(0,x)=[o=v_{0}\,,v_{1}\,,v_{2}\,,\dots ]$, using (\ref{eq:balls}), 
\begin{equation*}
\begin{split}
P^{t}f(x)& =\sum_{n=0}^{\infty }c_{n}^{t}\,\mathrm{Q}_{\phi (v_{n})}f(x)\,,
\\
\text{where\ \ }c_{0}^{t}& =1-\sigma ^{t}\left( \phi (v_{0})\right) \quad 
\text{and\ \ }c_{n}^{t}=\sigma ^{t}\left( \phi (v_{n-1})\right) -\sigma
^{t}\left( \phi (v_{n})\right) \;\text{ for}\;n\geq 1\,.
\end{split}%
\end{equation*}%
Thus, for arbitrary $u\in T$ and $x\in \partial T$ as above 
\begin{equation}
\mathbb{P}[X_{t}\in \partial T_{u}\mid X_{0}=x]=\sum_{n=0}^{\infty
}c_{n}^{t}\,\frac{\mu (\partial T_{v_{n}}\cap \partial T_{u})}{\mu (\partial
T_{v_{n}})}.  \label{eq:Pt}
\end{equation}

We know that we have some freedom in the choice of the measure $\sigma$: any
two measures whose distribution functions coincide on the value set $%
\Lambda_{\phi}$ of $\phi$ give rise to the same process. Recall the
Definition \ref{def:standard} of the standard $(d,\mu)$-process, now to be
re-named the standard $(\phi,\mu)$-process.

\subsection{Nearest neighbour random walks on a tree}

On a tree as a discrete structure, there are other, very well studied
stochastic processes, namely \emph{random walks.} Our aim is to analyze how
they are related with isotropic jump processes on the boundary of the tree.
A good part of the material outlined next is taken from the book of \textrm{%
Woess} \cite{W-Markov}. An older, recommended reference is the seminal paper
of \textrm{Cartier}~\cite{Ca}.

A \emph{nearest neighbour random walk} on the locally finite, infinite tree $%
T$ is induced by its stochastic transition matrix $\mathcal{P}=\left(
p(u,v)\right) _{u,v\in T}\,$ with the property that $p(u,v)>0$ if and only
if $u\sim v$. The resulting discrete-time Markov chain (random walk) is
written $(Z_{n})_{n\geq 0}\,$. Its $n$-step transition probabilities 
\begin{equation*}
p^{(n)}(u,v)=\mathbb{P}_{u}[Z_{n}=v],\ \ u,v\in T,
\end{equation*}%
are the elements of the $n^{\mathrm{th}}$ power of the matrix $\mathcal{P}$.
The notation $\mathbb{P}_{u}$ refers to the probability measure on the
trajectory space that governs the random walk starting at $u$. We assume
that the random walk is \emph{transient,} i.e., with probability $1$ it
visits any finite set only finitely often. Thus, $0<G(u,v)<\infty $ for all $%
u,v\in T$, where 
\begin{equation*}
G(u,v)=\sum_{n=0}^{\infty }p^{(n)}(u,v)
\end{equation*}%
is the \emph{Green kernel} of the random walk. In addition, we shall also
make crucial use of the quantities 
\begin{equation*}
F(u,v)=\mathbb{P}_{u}[Z_{n}=v\;\text{for some}\;n\geq 0]\quad \text{and\ \ }
U(v,v)=\mathbb{P}_{v}[Z_{n}=v\;\text{for some}\;n\geq 1]\,.
\end{equation*}%
We shall need several identities relating them and start with a few of them,
valid for all $u,v\in T$. 
\begin{align}
G(u,v)& =F(u,v)G(v,v)  \label{eq:FG} \\
G(v,v)& =\frac{1}{1-U(v,v)}  \label{eq:GU} \\
U(v,v)& =\sum_{u}p(v,u)F(u,v)  \label{eq:UF} \\
F(u,v)& =F(u,w)F(w,v)\quad \text{whenever }\;w\in \pi (u,v)  \label{eq:FF}
\end{align}%
The first three hold for arbitrary denumerable Markov chains, while (\ref%
{eq:FF}) is specific for trees (resp., a bit more generally, when $w$ is a
so-called cut point between $u$ and $v$). The identities show that those
quantities are completely determined just by all the $F(u,v)$, where $u\sim
v $. More identities, as to be found in \cite[Chapter 9]{W-Markov}, will be
displayed and used later on. By transience, the random walk $Z_{n}$ must
converge to a random end, a simple and well-known fact. See e.g. \cite{Ca}
or \cite[Theorem 9.18]{W-Markov}.

\begin{lemma}
\label{lem:conv} There is a $\partial T$-valued random variable $Z_{\infty}$
such that for every starting point $u \in T$, 
\begin{equation*}
\mathbb{P}_u[Z_n \to Z_{\infty} \;\text{in the topology of}\;\ \widehat T]
=1.
\end{equation*}
\end{lemma}

In brief, the argument is as follows: by transience, random walk
trajectories must accumulate at $\partial T$ almost surely. If such a
trajectory had two distinct accumulation points, say $x$ and $y$, then by
the nearest neighbour property, the trajectory would visit the vertex $x
\wedge_u y$ infinitely often, which can occur only with probability $0$.

We can consider the family of \emph{limit distributions} $\nu _{u}\,$, $u\in
T$, where for any Borel set $B\subset \partial T$, 
\begin{equation*}
\nu _{u}(B)=\mathbb{P}_{u}[Z_{\infty }\in B]\,.
\end{equation*}%
The sets $\partial T_{u}\,$, $u\in T$ (plus the empty set), form a
semi-algebra that generates the Borel $\sigma $-algebra of $\partial T$.
Thus, each $\nu _{u}$ is determined by the values of those sets. There is an
explicit formula, compare with \cite{Ca} or \cite[Proposition 9.23]{W-Markov}%
. For $v\neq o$, 
\begin{equation}
\nu _{u}(\partial T_{v})=%
\begin{cases}
F(u,v)\dfrac{1-F(v,v^{-})}{1-F(v^{-},v)F(v,v^{-})}\,, & \text{if}\;u\in
\{v\}\cup (T\setminus T_{v})\,, \\[12pt] 
1-F(u,v)\dfrac{F(v,v^{-})-F(v^{-},v)F(v,v^{-})}{1-F(v^{-},\nu )F(v,v^{-})}\,,
& \text{if}\;u\in T_{v}\,.%
\end{cases}
\label{eq:nu}
\end{equation}%
A \emph{harmonic function} is a function $h:T\rightarrow \mathbb{R}$ with $%
\mathcal{P}h=h$, where 
\begin{equation*}
Ph(u)=\sum_{v}p(u,v)h(v)\,.
\end{equation*}%
For any Borel set $B\subset \partial T$, the function $u\mapsto \nu _{u}(B)$
is a bounded harmonic function. One deduces that all $\nu _{u}$ are
comparable: $p^{(k)}(u,v)\,\nu _{u}\leq \nu _{v}\,$, where $k$ is the length
of $\pi (u,v)$. Thus, for any function $\varphi \in L^{1}(\partial T,\nu
_{o})$, the function $h_{\varphi }$ defined by 
\begin{equation*}
h_{\varphi }(u)=\int_{\partial T}\varphi \,d\nu _{u}
\end{equation*}%
is finite and harmonic on $T$. It is often called the \emph{Poisson transform%
} of $\varphi $.

We next define a measure $\mathsf{m}$ on $T$ via its atoms: $\mathsf{m}(o)=1$%
, and for $v\in T\setminus \{o\}$ with $\pi (o,v)=[o=v_{0}\,,v_{1}\,,\dots
,v_{k}=v]$, 
\begin{equation}
\mathsf{m}(v)=\frac{p(v_{0},v_{1})p(v_{1},v_{2})\cdots p(v_{k-1},v_{k})}{%
p(v_{1},v_{0})p(v_{2},v_{1})\cdots p(v_{k},v_{k-1})}\,.  \label{eq:treerev}
\end{equation}%
Then for all $u,v\in T$, 
\begin{equation}
\mathsf{m}(u)p(u,v)=\mathsf{m}(v)p(v,u)\,,\quad \text{and consequently}\quad 
\mathsf{m}(u)G(u,v)=\mathsf{m}(v)G(v,u)\,;  \label{eq:reversible}
\end{equation}%
the random walk is \emph{reversible}. This would allow us to use the \emph{%
electrical network} interpretation of $(T,\mathcal{P},\mathsf{m})$, for
which there are various references: see e.g. Yamasaki~\cite{Ya}, Soardi~\cite%
{So}, or -- with notation as used here -- \cite[Chapter 4]{W-Markov}. We do
not go into its details here; each edge $e=[v^{-},v]\in E(T)$ is thought of
as an electric conductor with \emph{conductance} 
\begin{equation*}
a(v^{-},v)=\mathsf{m}(v)p(v,v^{-}).
\end{equation*}%
We get the Dirichlet form $\mathcal{E}_{T}=\mathcal{E}_{T,\mathcal{P}}$ for
functions $f,g:T\rightarrow \mathbb{R}$, defined by 
\begin{equation}
\mathcal{E}_{T}(f,g)=\sum_{[v^{-},v]\in E(T)}\left( f(v)-f(v^{-})\right)
\left( g(v)-g(v^{-})\right) \,a(v^{-},v)\,.  \label{eq:Dir-T}
\end{equation}%
It is well defined for $f,g$ in the space 
\begin{equation}
\mathcal{D}(T)=\mathcal{D}(T,\mathcal{P})=\{f:T\rightarrow \mathbb{R}\mid 
\mathcal{E}_{T}(f,f)<\infty \}.  \label{eq:D(T)}
\end{equation}

\subsection{Harmonic functions of finite energy and their boundary values}

We are interested in the subspace 
\begin{equation*}
\mathcal{HD}(T)=\mathcal{HD}(T,\mathcal{P})=\{h\in \mathcal{D}(T,\mathcal{P}%
):\mathcal{P}h=h\}
\end{equation*}%
of harmonic functions with \emph{finite power.} The terminology comes from
the interpretation of such a function as the potential of an electric flow
(or current), and then $\mathcal{E}_{T}(h,h)$ is the power of that flow.%
\footnote{%
In the mathematical literature, mostly the expression \textquotedblleft
energy\textquotedblright\ is used for $\mathcal{E}_{T}(h,h)$, but it seems
that \textquotedblleft power\textquotedblright\ is the more appropriate
terminology from Physics.}

Every function in $\mathcal{HD}(T,\mathcal{P})$ is the Poisson transform of
some function $\varphi \in L^{2}(\partial T,\nu _{o})$. This is valid not
only for trees, but for general finite range reversible Markov chains, and
follows from the following facts.

\begin{enumerate}
\item Every function in $\mathcal{HD}$ is the difference of two non-negative
functions in $\mathcal{HD}$.

\item Every non-negative function in $\mathcal{HD}$ can be approximated,
monotonically from below, by a sequence of non-negative bounded functions in 
$\mathcal{HD}$.

\item Every bounded harmonic function (not necessarily with finite power) is
the Poisson transform of a bounded function on the boundary.
\end{enumerate}

In the general setting, the latter is the (active part of) the Martin
boundary, with $\nu _{u}$ being the limit distribution of the Markov chain,
starting from $u$, on that boundary. (1) and (2) are contained in \cite{Ya}
and \cite{So}, while (3) is part of general Martin boundary theory, see e.g. 
\cite[Theorem 7.61]{W-Markov}.

Thus, we can introduce a form $\mathcal{E}_{\mathcal{HD}}$ on $\partial T$
by setting 
\begin{equation}
\begin{split}
\mathcal{D}(\partial T,\mathcal{P})& =\{\varphi \in L^{1}(\partial T,\nu
_{o}):\mathcal{E}_{T}(h_{\varphi },h_{\varphi })<\infty \}\,, \\
\mathcal{E}_{\mathcal{HD}}(\varphi ,\psi )& =\mathcal{E}_{T}(h_{\varphi
},h_{\psi })\quad \text{for}\;\varphi ,\psi \in \mathcal{D}(\partial T,%
\mathcal{P}).
\end{split}
\label{eq:DirHD}
\end{equation}

\subsection{Jump processes on the boundary of a tree}

\label{SecDN}Kigami~\cite{Ki} elaborates an expression for the form $%
\mathcal{E}_{\mathcal{HD}}(\varphi ,\psi )$ of (\ref{eq:DirHD}) by
considerable effort, shows its regularity properties and then studies the
jump process on $\partial T$ induced by this Dirichlet form. We call this
the \emph{boundary process} associated with the random walk.

Now, there is a rather simple expression for $\mathcal{E}_{\mathcal{HD}}\,.$
We define the \emph{Na\"{\i}m kernel} on $\partial T\times \partial T$ by 
\begin{equation}
\Theta _{o}(x,y)=%
\begin{cases}
\dfrac{\mathsf{m}(o)}{G(o,o)F(o,x \wedge y)F(x \wedge y ,o)}\,, & \text{if}%
\;x\neq y \,, \\ 
+\infty \,, & \text{if}\;x=y\,.%
\end{cases}
\label{eq:Naim}
\end{equation}%
In our case, $\mathsf{m}(o)=1$, but we might want to change the base point,
or normalize the measure $\mathsf{m}$ in a different way.

\begin{theorem}
\label{thm:doob-naim} For any transient nearest neighbour random walk on the
tree $T$ with root $o$, and all functions $\varphi $, $\psi $ in $\mathcal{D}%
(\partial T,\mathcal{P})$, 
\begin{equation*}
\mathcal{E}_{\mathcal{HD}}(\varphi ,\psi )=\frac{1}{2}\int_{\partial
T}\int_{\partial T}\left(\varphi (x)-\varphi (y)\right)\left(\psi (x)-\psi
(y)\right)\Theta _{o}(x,y)\,d\nu _{o}(x)\,d\nu _{o}(y)\,.
\end{equation*}
\end{theorem}

There is a general definition of the Na\"{\i}m kernel \cite{Na} that
involves the Martin boundary, which in the present case is $\partial T$. A
proof of Theorem \ref{thm:doob-naim} is given in \cite{Do} in a setting of
abstract potential theory on Green spaces, which are locally Euclidean. The
definition of \cite{Na} refers to the same type of setting. Now, infinite
networks, even when seen as metric graphs, are not locally Euclidean. In
this sense, so far the definition of the kernel and a proof of Theorem \ref%
{thm:doob-naim} for transient, reversible random walks have not been well
accessible in the literature. In a forthcoming paper, Georgakopoulos and
Kaimanovich~\cite{GK} will provide those \textquotedblleft missing
links\textquotedblright . We give a direct and simple proof of Theorem \ref%
{thm:doob-naim} for the specific case of trees. We start with the following
observation.

\begin{lemma}
\label{lem:invariant} The measure $\Theta_o(x,y)\,d\nu_o(x)\,d\nu_o(y)$ on $%
\partial T \times \partial T$ is invariant with respect to changing the base
point (root) $o$.
\end{lemma}

\begin{proof}
We want to replace the base point $o$ with some other $u\in T$. We may
assume that $u\sim o$. Indeed, then we may step by step replace the current
base point by one of its neighbours to obtain the result for arbitrary $u$.

Recall that the confluent that appears in the definition (\ref{eq:Naim})) of 
$\Theta_o$ depends on the root $o$, while for $\Theta_x$ it becomes the one
with respect to $x$ as the new root. It is a well-known fact that 
\begin{equation*}
\frac{d\nu_u}{d\nu_o}(x) = K(u,x) = \frac{G(u,u\wedge_o x)}{G(o,u\wedge_o x)}%
\,,
\end{equation*}
the \emph{Martin kernel.} Thus, we have to show that for all $x,y \in
\partial T$ ($x \ne y$) 
\begin{equation*}
\dfrac{\mathsf{m}(o)}{G(o,o)F(o,x\wedge_o y) F(x\wedge_o y,o)} = \dfrac{%
\mathsf{m}(u)K(u,x) K(u,y)}{G(u,u)F(u,x\wedge_u y) F(x\wedge_u y,u)} \,.
\end{equation*}

\medskip

\emph{Case 1.} $x ,y \in \partial T_{u}$. Then $x \wedge_{o}y=x \wedge_{u}y
=:v\in T_{u}$, and $u\wedge_{o}x =u\wedge_{o}y=u $. Thus, using (\ref{eq:FG}%
), (\ref{eq:FF}) and the fact that by (\ref{eq:reversible}) 
\begin{equation*}
\mathsf{m}(u)/G(o,u)=\mathsf{m}(o)/G(u,o),
\end{equation*}
we obtain 
\begin{eqnarray*}
\frac{\mathsf{m}(u)K(u,x)K(u,y)}{G(u,u)F(u,x \wedge_{u}y) F(x\wedge_{u}y,u)}
&=&\frac{\mathsf{m}(u)}{G(u,u)F(u,v)F(v,u)} \left( \frac{G(u,u)}{G(o,u)}%
\right) ^{2} \\
&=&\frac{\mathsf{m}(o)G(u,u)}{F(u,v)F(v,u)G(o,u)G(u,o)} \\
&=&\frac{\mathsf{m}(o)}{F(u,v)F(v,u)F(o,u)F(u,o)G(o,o)} \\
&=&\frac{\mathsf{m}(o)}{F(o,v)F(v,o)G(o,o)}\,,
\end{eqnarray*}%
as required. There are 3 more cases.

\smallskip

\emph{Case 2.} $x ,y \in \partial T\setminus \partial T_{u}\,$. Then 
\begin{equation*}
x \wedge_{o}y =x \wedge _{u}y =:w\in T\setminus T_{u},\quad\text{and\ \ }
u\wedge_{o}x =u\wedge _{o}y =o.
\end{equation*}

\smallskip

\emph{Case 3.} $x \in \partial T_{u}\,$, $y \in \partial T\setminus\partial
T_{u}\,$. Then 
\begin{equation*}
x \wedge _{o}y =o,\ \ x \wedge_{u}y =u,\ \ u\wedge _{o}x =u \quad\text{and\
\ } u\wedge _{o}y =o.
\end{equation*}

\smallskip

\emph{Case 4.} $x \in \partial T\setminus \partial T_{u}\,$, $y \in \partial
T_{u}\,$. This is like Case 3, exchanging the roles of $x $ and $y$.

\smallskip

In all cases 2--4, the computation is done very similarly to Case 1, a
straightforward exercise.~
\end{proof}

For proving Theorem \ref{thm:doob-naim}, we need a few more facts related
with the network setting; compare e.g. with \cite[\S 4.D]{W-Markov}.

The space $\mathcal{D}(T)$ of (\ref{eq:D(T)}) is a Hilbert space when
equipped with the inner product 
\begin{equation*}
(f,g) = \mathcal{E}_T(f,g) + f(o)g(o)\,.
\end{equation*}
The subspace $\mathcal{D}_0(T)$ is defined as the closure of the space of
finitely supported functions in $\mathcal{D}(T)$. It is a proper subspace if
and only if the random walk is transient, and then the function $G_v(u) =
G(u,v)$ is in $\mathcal{D}_0(T)$ for any $v \in T$ \cite{Ya}, \cite{So}. We
need the formula 
\begin{equation}  \label{eq:Green}
\mathcal{E}_T(f,G_v) = \mathsf{m}(v) f(v) \quad \text{for every}\; f \in 
\mathcal{D}_0(T)\,.
\end{equation}
Given a branch $T_w$ of $T$ ($w \in T \setminus \{ o\}$), we can consider it
as a subnetwork equipped with the same conductances $a(u,v)$ for $[u,v] \in
E(T_w)$. The associated measure on $T_w$ is 
\begin{equation*}
\mathsf{m}_{T_w}(u) = \sum_{v \in T_z: v \sim u} a(u,v) = 
\begin{cases}
\mathsf{m}(u) & \text{if}\; u \in T_w \setminus \{w\}\,, \\ 
\mathsf{m}(w)-a(w,w^-) & \text{if}\; u =w\,.%
\end{cases}%
\end{equation*}
The resulting random walk on $T_w$ has transition probabilities 
\begin{equation*}
p_{T_w}(u,v) = \frac{a(v,w)}{\mathsf{m}_{T_w}(u)} = 
\begin{cases}
p(u,v) & \text{if}\; u \in T_w \setminus \{w\}\,,\;v \sim u\,, \\[9pt] 
\dfrac{p(w,v)}{1-p(w,w^-)} & \text{if}\; u =w\,,\;v \sim u\,.%
\end{cases}%
\end{equation*}
We have $F_{T_w}(u,u^-) = F(u,u^-)$ and thus also $F_{T_w}(u,w) = F(u,w)$
for every $u \in T_w \setminus \{w\}$, because before its first visit to $w$%
, the random walk on $T_w$ obeys the same transition probabilities as the
original random walk on $T$. It is then easy to see \cite[p. 241]{W-Markov}
that the random walk on $T_w$ is transient if and only if for the original
random walk, $F(w,w^-) <1$, which in turn holds if and only if $%
\nu_o(\partial T_w) > 0$. (In other parts of this and the preceding two
sections, this is always assumed, but for the proof of Theorem \ref%
{thm:doob-naim}, we just assume the random walk on the whole of $T$ to be
transient.) Conversely, if $F(w,w^-)=1$ then $F(u,w)=1$ for all $u \in T_w\,$%
.

Below, we shall need the following formula for the limit distributions.

\begin{lemma}
\label{lem:nu} For $u \in T \setminus \{o\}$, 
\begin{equation*}
\nu_u(\partial T_u) = 1 - p(u,u^-)\left(G(u,u) - G(u^-,u)\right).
\end{equation*}
\end{lemma}

\begin{proof}
By (\ref{eq:formula1}), 
\begin{equation*}
G(u,u)p(u,u^-) = \frac{F(u,u^-)}{1-F(u,u^-)F(u^-,u)}
\end{equation*}
Thus, 
\begin{equation*}
p(u,u^-)\left(G(u,u) - G(u^-,u)\right) = \left(1-F(u^-,u)\right)
G(u,u)p(u,u^-) = 1 -\nu_u(\partial T_u)
\end{equation*}
after a short computation using (\ref{eq:nu})
\end{proof}

\medskip

\begin{proof}[Proof of Theorem \protect\ref{thm:doob-naim}]
We first prove the Doob-Na\"{\i}m formula (shortly, D-N-formula) for the
case when $\varphi =\mathbf{1}_{\partial T_{v}}$ and $\psi =\mathbf{1}%
_{\partial T_{w}}$ for two proper branches $T_{v}$ and $T_{w}$ of $T$. They
are either disjoint, or one of them contains the other.

\smallskip

\emph{Case 1. } $T_w \subset T_v\,.\ $ (The case $T_v \subset T_w\,$ is
analogous by symmetry.)

This means that $w\in T_{v}\,.$ For $x,y\in \partial T$ we have 
\begin{equation*}
\left(\varphi (x)-\varphi (y)\right)\left(\psi (x)-\psi (y)\right)=1
\end{equation*}%
if $\;x\in \partial T_{w}$ and $y\in \partial T\setminus \partial T_{v}$ or
conversely, and $=0$ otherwise. By Lemma \ref{lem:invariant}, we may choose $%
v$ as the base point. Thus, the right hand side of the identity is 
\begin{equation*}
\int_{\partial T\setminus \partial T_{v}}\int_{\partial T_{w}}\Theta
_{v}(x,y)\,d\nu _{v}(x)\,d\nu _{v}(y)=\dfrac{\mathsf{m}(v)}{G(v,v)}\,\nu
_{v}(\partial T\setminus \partial T_{v})\nu _{v}(\partial T_{w})\,,
\end{equation*}%
since $x\wedge _{v}y=v$ and $F(v,v)=1$.

Let us now turn to the left hand side of the D-N-formula. The Poisson
transforms of $\varphi $ and $\psi $ are 
\begin{equation*}
h_{\varphi }(u)=\nu _{u}(\partial T_{v})\quad \text{and\ \ }h_{\psi }(u)=\nu
_{u}(\partial T_{w}).
\end{equation*}%
By (\ref{eq:nu}), 
\begin{eqnarray*}
h_{\varphi }(u) &=&F(u,v)\nu _{v}(\partial T_{v})\,,\ \ u\in \{v\}\cup
(T\setminus T_{v}) \\
1-h_{\varphi }(u) &=&F(u,v)\nu _{v}(\partial T\setminus \partial T_{v})\,,\
\ u\in T_{v}\,.
\end{eqnarray*}%
We set $F_{v}(u)=F(u,v)$ and write 
\begin{equation*}
h_{\varphi }(u)-h_{\varphi }(u^{-})=\left( 1-h_{\varphi }(u^{-})\right)
-\left( 1-h_{\varphi }(u)\right)
\end{equation*}%
whenever this is convenient, and analogously for $h_{\psi }\,$. Then we get 
\begin{equation*}
\begin{split}
& \mathcal{E}_{T}(h_{\varphi }\,,h_{\psi }) \\
& =\!\!\!\sum_{[u,u^{-}]\in E(T)\setminus E(T_{v})}\!\!\!a(u,u^{-})\left(
F(u,v)-F(u^{-},v)\right) \nu _{v}(\partial T_{v})\,\left(
F(u,v)-F(u^{-},v)\right) \nu _{v}(\partial T_{w}) \\
& \quad -\!\!\!\!\!\!\sum_{[u,u^{-}]\in E(T_{v})\setminus
E(T_{w})}\!\!\!\!\!\!a(u,u^{-})\left( F(u,v)-F(u^{-},v)\right) \nu
_{v}(\partial T\setminus \partial T_{v})\,\left( F(u,v)-F(u^{-},v)\right)
\nu _{v}(\partial T_{w}) \\
& \quad +\!\!\!\!\!\!\sum_{[u,u^{-}]\in
E(T_{w})}\!\!\!\!\!\!a(u,u^{-})\left( F(u,v)-F(u^{-},v)\right) \nu
_{v}(\partial T\setminus \partial T_{v})\,\left( F(u,v)-F(u^{-},v)\right)
\nu _{v}(\partial T\setminus \partial T_{w}) \\
& =\mathcal{E}_{T}(F_{v}\,,F_{w})\nu _{v}(\partial T_{v})\nu _{v}(\partial
T_{w})-\mathcal{E}_{T_{v}}(F_{v}\,,F_{w})\nu _{v}(\partial T_{w})+\mathcal{E}%
_{T_{w}}(F_{v}\,,F_{w})\nu _{v}(\partial T\setminus \partial T_{v})\,,
\end{split}%
\end{equation*}%
where of course $\mathcal{E}_{T_{v}}$ is the Dirichlet form of the random
walk on the branch $T_{v}\,$, as discussed above, and analogously for $%
\mathcal{E}_{T_{w}}\,$. Now $F_{v}=G_{v}/G(v,v)$ by (\ref{eq:FG}), whence (%
\ref{eq:Green}) yields 
\begin{equation}
\mathcal{E}_{T}(F_{v},F_{w})=\frac{\mathcal{E}_{T}(G_{v},F_{w})}{G(v,v)}=%
\frac{\mathsf{m}(v)F(v,w)}{G(v,v)}\,.  \label{eq:DirT}
\end{equation}%
Recall that for the random walk on $T_{v}\,$, we have $F_{T_{v}}(u,v)=F(u,v)$
for every $u\in T_{v}\,$. Also, 
\begin{equation*}
\mathsf{m}_{T_{v}}(v)=\mathsf{m}(v)-a(v,v^{-})=\mathsf{m}(v)\left(
1-p(v,v^{-})\right) .
\end{equation*}%
We apply (\ref{eq:DirT}) to that random walk and get 
\begin{equation*}
\mathcal{E}_{T_{v}}(F_{v},F_{w})=\frac{\mathsf{m}(v)\left(
1-p(v,v^{-})\right) F(v,w)}{G_{T_{v}}(v,v)}\,.
\end{equation*}%
We now apply (\ref{eq:GU}) and (\ref{eq:UF}), recalling in addition that 
\begin{equation*}
p_{T_{v}}(v,u)=\frac{p(v,u)}{1-p(v,v^{-})}
\end{equation*}%
for $u\in T_{v}\,$, and obtain%
\begin{eqnarray*}
\frac{1-p(v,v^{-})}{G_{T_{v}}(v,v)} &=&1-p(v,v^{-})-\left(
1-p(v,v^{-})\right) U_{T_{v}}(v,v) \\
&=&1-p(v,v^{-})-\sum_{u:u^{-}=v}p(v,u)F(u,v) \\
&=&1-p(v,v^{-})-\left( U(v,v)-p(v,v^{-})F(v^{-},v)\right) \\
&=&\frac{1}{G(v,v)}-p(v,v^{-})\left( 1-F(v^{-},v)\right) =\frac{\nu
_{v}(\partial T_{v})}{G(v,v)},
\end{eqnarray*}%
where in the last step we have used Lemma \ref{lem:nu}. We have obtained 
\begin{equation*}
\mathcal{E}_{T_{v}}(F_{v},F_{w})=\frac{\mathsf{m}(v)F(v,w)}{G(v,v)}\,\nu
_{v}(\partial T_{v}).
\end{equation*}%
In the same way, exchanging roles between $T_{w}$ and $T_{v}\,$ and using
reversibility (\ref{eq:reversible}), 
\begin{equation*}
\mathcal{E}_{T_{w}}(F_{v},F_{w})=\frac{\mathsf{m}(w)F(w,v)}{G(w,w)}\,\nu
_{w}(\partial T_{w})=\frac{\mathsf{m}(v)F(v,w)}{G(v,v)}\,\nu _{w}(\partial
T_{w})=\frac{\mathsf{m}(v)}{G(v,v)}\,\nu _{v}(\partial T_{w})
\end{equation*}%
Putting things together, we get that 
\begin{equation*}
\mathcal{E}_{T}(h_{\varphi }\,,h_{\psi })=\mathcal{E}_{T_{w}}(F_{v}\,,F_{w})%
\nu _{v}(\partial T\setminus \partial T_{v})=\frac{\mathsf{m}(v)}{G(v,v)}%
\,\nu _{v}(\partial T_{w})\nu _{v}(\partial T\setminus \partial T_{v}),
\end{equation*}%
as proposed.

\smallskip

\emph{Case 2. } $T_w \cap T_v = \emptyset\,.$

In view of Lemma \ref{lem:invariant}, both sides of the D-N-formula are
independent of the root $o$. Thus we may declare our root to be one of the
neighbours of $v$ that is not on $\pi (v,w)$. Also, let $\bar{v}$ be the
neighbour of $v$ on $\pi (w,v)$. Then, with our chosen new root, the
complement of the \textquotedblleft old\textquotedblright\ $T_{v}$ is $T_{%
\bar{v}}$, which contains $T_{w}$ (The latter remains the same with respect
to the new root).

Thus, we can apply the result of case 1 to $T_{\bar{v}}$ and $T_{w}$. This
means that we have to replace the functions $\varphi $ and $h_{\varphi }$
with $1-\varphi $ and $1-h_{\varphi }$, respectively, which just means that
we change the sign on both sides of the identity. We are re-conducted to
Case 1 without further computations.

\smallskip

We deduce from what we have done so far, and from linearity of the Poisson
transform as well of bilinearity of the forms on both sides of the
D-N-formula , that it holds for linear combinations of indicator functions
of sets $\partial T_{v}\,$. Those indicator functions are dense in the space 
$C(\partial T)$ with respect to the $\max $-norm. Thus, the D-N-formula
holds for all continuous functions on $\partial T$. The extension to all of $%
\mathcal{D}(\partial T,\mathcal{P})$ is by standard approximation.
\end{proof}

\section{The duality of random walks on trees and isotropic processes 
\newline
on their boundaries}

\label{duality}\setcounter{equation}{0}When looking at our isotropic
processes and at the boundary process of Kigami \cite{Ki}, it is natural to
ask the following two questions.

\medskip

\textbf{Question I.} Given a transient random walk on $T$ associated with
the Dirichlet form $\mathcal{E}_{T}$ of (\ref{eq:Dir-T}), does the boundary
process on $\partial T$ induced by the form $\mathcal{E}_{\mathcal{HD}}$ of (%
\ref{eq:DirHD}) arise as one of the isotropic processes (\ref{convexcomb-Pt}%
) on $\partial T$ with transition probabilities (\ref{eq:Pt}), with respect
to the measure $\mu =\nu _{o}$ on $\partial T$, some ultra-metric element $%
\phi $ on $T$ and a suitable distance distribution function $\sigma $ on $%
[0\,,\,\infty )\,$?

\medskip

\textbf{Question II.} Conversely, given data $\mu $, $\phi $ and $\sigma $,
is there a random walk on $T$ with limit distribution $\nu _{o}=\mu $ such
that the isotropic process induced by $\mu $, $\phi $ and $\sigma $ is the
boundary process with Dirichlet form $\mathcal{E}_{\mathcal{HD}}$ ?

\medskip

Before answering both questions, we need to specify the assumptions more
precisely. When starting with $(\phi ,\mu ,\sigma )$, we always assume as
before that $\mu $ is supported by the whole of $\partial T$.

Thus, on the side of the random walk, we also want that $\limfunc{supp}(\nu
_{o})=\partial T$. This is equivalent with the requirement that $\nu
_{o}(\partial T_{v})>0$ for every $v\in T$. By (\ref{eq:nu}) this is in turn
equivalent with 
\begin{equation}
F(v,v^{-})<1\quad \text{for every}\;v\in T\setminus \{o\}.  \label{eq:F<1}
\end{equation}%
Indeed, we shall see that we need a bit more, namely that 
\begin{equation}
\lim_{v\rightarrow \infty }G(v,o)=0\,,  \label{eq:Green-0}
\end{equation}%
that is, for every $\varepsilon >0$ there is a finite set $A\subset T$ such
that $G(v,o)<\varepsilon $ for all $v\in T\setminus A$. This condition is
necessary and sufficient for solvability of the \emph{Dirichlet problem:}
for any $\varphi \in C(\partial X)$, its Poisson transform $h_{\varphi }$
provides the unique continuous extension of $\varphi $ to $\widehat{T}$
which is harmonic in $T$. See e.g. \cite[Corollary 9.44]{W-Markov}.

We shall restrict attention to random walks with properties (\ref{eq:F<1})
and (\ref{eq:Green-0}) on a rooted tree with forward degrees $\ge 2$.

\subsection{Answer to Question I}

We start with a random walk that fulfills the above requirements. We know
from \S \ref{Intro} that each $(\mu ,\phi ,\sigma )$-process arises as the
standard process of Definition \ref{def:standard} with respect to the
intrinsic metric (cf. Theorem \ref{p-laplace}): given $\phi $ and $\sigma $,
the intrinsic metric is induced by the ultra-metric element 
\begin{equation}
\phi _{\ast }(u)=-1\big/\log \sigma \left( \phi (u)\right) \,.
\label{intrinsic-element}
\end{equation}%
Thus, we can eliminate $\sigma $ from our considerations by just looking for
an ultra-metric element $\phi $ such that the boundary process is the
standard process on $\partial T$ associated with $(\nu _{0}\,,\phi )$.

Since the processes are determined by the Dirichlet forms, we infer from
Theorems \ref{Thm-Dirichlet form/Laplacian} and \ref{thm:doob-naim} that we
are looking for $\phi $ such that $J(x,y)=\Theta _{o}(x,y)$ for all $x,y\in
\partial T$ with $x\neq y$, where $J(x,y)$ is given by (\ref{Jxy}).
Rewriting $J(x,y)$ in terms of $\phi $, $\nu _{o}$ and the tree structure,
this becomes 
\begin{equation}
\frac{1}{\phi (o)}+\int_{1/\phi (o)}^{1/\phi (x\wedge y)}\frac{dt}{\nu
_{o}\left( B_{1/t}^{\phi }(x)\right) }=\frac{\mathsf{m}(o)}{%
G(o,o)F(o,x\wedge y)F(x\wedge y,o)}.  \label{eq:tosolve1}
\end{equation}%
In our case, $\mathsf{m}(o)=1$, but we keep track of what happens when one
changes the root or the normalisation of $\mathsf{m}$. First of all, since $%
\deg ^{+}(o)\geq 2$, there are $x,y\in \partial T$ such that $x\wedge y=o$.
We insert these two boundary points in (\ref{eq:tosolve1}). Since $F(o,o)=1$%
, we see that we must have 
\begin{equation*}
\phi (o)=G(o,o)/\mathsf{m}(o)\,.
\end{equation*}%
Now take $v\in T\setminus \{o\}$. Since forward degrees are $\geq 2$, there
are $x,y,y^{\prime }\in \partial T$ such that $x\wedge y=v$ and $x\wedge
y^{\prime }=v^{-}$. We write (\ref{eq:tosolve1}) first for $(x,y^{\prime })$
and then for $(x,y)$ and then take the difference, leading to the equation 
\begin{equation}
\int_{1/\phi (v^{-})}^{1/\phi (v)}\frac{dt}{\nu _{o}\left( B_{1/t}^{\phi
}(x)\right) }=\frac{\mathsf{m}(o)}{G(o,o)F(o,v)F(v,o)}-\frac{\mathsf{m}(o)}{%
G(o,o)F(o,v^{-})F(v^{-},o)}.  \label{eq:tosolve2}
\end{equation}%
By (\ref{eq:balls}), within the range of the last integral we must have $%
B_{1/t}^{\phi }(v)=\partial T_{v}\,$, whence that integral reduces to 
\begin{equation*}
\left( \frac{1}{\phi (v)}-\frac{1}{\phi (v^{-})}\right) \frac{1}{\nu
_{o}(\partial T_{v})}
\end{equation*}%
We multiply equation (\ref{eq:tosolve2}) by $\nu _{o}(\partial T_{v})$ and
simplify the resulting right hand side 
\begin{equation*}
\left( \frac{\mathsf{m}(o)}{G(o,o)F(o,v)F(v,o)}-\frac{\mathsf{m}(o)}{%
G(o,o)F(o,v^{-})F(v^{-},o)}\right) \nu _{o}(\partial T_{v})
\end{equation*}%
by use of the identities (\ref{eq:FG}) -- (\ref{eq:FF})) and the first of
the two formulas of (\ref{eq:nu}) (for $\nu _{o}$). We obtain that the
ultra-metric element that we are looking for should satisfy 
\begin{equation}
\frac{1}{\phi (v)}-\frac{1}{\phi (v^{-})}=\frac{\mathsf{m}(o)}{G(v,o)}-\frac{%
\mathsf{m}(o)}{G(v^{-},o)}\quad \text{for every}\;v\in T\setminus \{o\}\,.
\label{eq:solution1}
\end{equation}%
This determines $1/\phi (v)$ recursively, and with $\mathsf{m}(o)=1$, we get 
\begin{equation*}
\phi (v)=G(v,o)\,.
\end{equation*}%
Since by (\ref{eq:FG}) and (\ref{eq:FF}) 
\begin{equation*}
G(v,o)=F(v,v^{-})G(v^{-}o),
\end{equation*}%
the assumptions (\ref{eq:F<1}) and (\ref{eq:Green-0}) yield that $\phi $ is
an ultra-metric element. Tracing back the last computations, we find that
with this choice of $\phi $, we have indeed that $J(x,y)=\Theta _{o}(x,y)$
for all $x,y\in \partial T$ with $x\neq y$. We have proved the following.

\begin{theorem}
\label{thm:I} Let $T$ be a locally finite, rooted tree with forward degrees $%
\geq 2$. Consider a transient nearest neighbour random walk on $T$ that
satisfies (\ref{eq:FG}) and (\ref{eq:FF}). Then the boundary process on $%
\partial T$ induced by the Dirichlet form (\ref{eq:DirHD}) coincides with
the standard process associated with ultra-metric element $\phi =G(\cdot ,o)$
and the limit distribution $\nu _{o}$ of the random walk.
\end{theorem}

Let $\mathcal{L}$ be the Laplacian associated with the boundary process of
Theorem \ref{thm:I}. $\mathcal{L}$ acts on locally constant functions $f$ by%
\begin{equation*}
\mathcal{L}f(x)=\int_{\partial T}\left( f(x)-f(y)\right) \Theta
_{o}(x,y)\,d\nu _{o}(y).
\end{equation*}%
In view of the identification of balls in $\partial T$ with vertices of $T,$
the functions of (\ref{fC}) now become%
\begin{equation*}
f_{v}=\frac{\mathbf{1}_{\partial T_{v}}}{\nu _{o}(\partial T_{v})}-\frac{%
\mathbf{1}_{\partial T_{v^{-}}}}{\nu _{o}(\partial T_{v^{-}})},\text{ }v\in
T\backslash \{o\}.
\end{equation*}%
In addition, we set $f_{o}=\mathbf{1}$ and note that it is an eigenfunction
of $\mathcal{L}$ with eigenvalue $0.$ Applying Theorem \ref{Eigenvalues-thm}
we obtain

\begin{corollary}
\label{eigenvalues -boundary process}For $v\in T\setminus \{o\},$ we have $%
\mathcal{L}f_{v}=G(v^{-},o)^{-1}f_{v}$ and the set of eigenfunctions $%
\left\{ f_{v}\right\} _{v\in T}$ is complete. In particular, we have 
\begin{equation*}
\func{spec}\mathcal{L}=\overline{\{G(v,o)^{-1}:v\in T\}}\cup \{0\}.
\end{equation*}
\end{corollary}

\begin{remark}
\RM\label{eigenspace-boundary process} For any two vertices $v$ and $w$ in $%
T\setminus \{o\}$ such that $v^{-}=w^{-}=u$ the functions $f_{v}$ and $f_{w}$
are eigenfunctions of $\mathcal{L}$ corresponding to the eigenvalue $\lambda
=$ $1/G(u,o).$ Hence the eigenspace $\mathcal{H(}u\mathcal{)}$ corresponding
to the vertex $u$ is spanned by functions $\{f_{v}:v^{-}=u\}.$ Since the
rank of the system $\{f_{v}:v^{-}=u\}$ is $\deg ^{+}(u)-1,$ where $\deg
^{+}(u)\geq 2$ is the forward degree of the vertex $u$, we obtain%
\begin{equation*}
\dim \mathcal{H(}u\mathcal{)}=\deg ^{+}(u)-1
\end{equation*}%
(cf. (\ref{HB})).
\end{remark}

\begin{remark}
\RM\label{rmk:sigma} Given the random walk on $T$ and the associated
boundary process on $\partial T$, we might want to realize it as the $(\nu
_{o},\phi ,\sigma )$-process for an ultra-metric element $\phi $ different
from $G(\cdot ,o)$. This means that we have to look for a suitable distance
distribution $\sigma $ on $[0\,,\,\infty )$, different from the inverse
exponential distribution (\ref{invexp}). In view of (\ref{intrinsic-element}%
), we are looking for $\sigma $ such that for our given generic $\phi $, 
\begin{equation*}
\sigma \left(\phi (v)\right)=e^{-1/G(v,o)}.
\end{equation*}%
For this it is necessary that $\phi (u)=\phi (v)$ whenever $G(u,o)=G(v,o)$:
we need $\phi $ to be constant on equipotential sets. In that case, the
distribution function $\sigma (r)$ is determined by the above equation for $%
r $ in the value set $\Lambda _{\phi }$ of the ultra-metric $d_{\phi }\,$.
We can \textquotedblleft interpolate\textquotedblright\ that function in an
arbitrary way (monotone increasing, left continuous) and get a feasible
measure $\sigma $.
\end{remark}

\subsection{Answer to Question II}

Answering Question II means that we start with $\phi $ and $\mu $ and then
look for a random walk with limit distribution $\nu _{o}=\mu $ such that the
standard $(\phi ,\mu )$-process is the boundary process associated with the
random walk. We know from Theorem \ref{thm:I} that in this case, we should
have $\phi (v)=G(v,o)$, whence in particular, $\phi (o)>1$. Thus we cannot
expect that every $\phi $ is suitable. The most natural choice is to replace 
$\phi $ by $C\cdot \phi $ for some constant $C>0$. For the standard
processes associated with $\phi $ and $C\cdot \phi $, respectively, this
just gives rise of a linear time change: if the old process is $%
\{X_{t}\}_{t>0}$, then the new one is $\{X_{t/C}\}_{t>0}\,$.

\begin{theorem}
\label{thm:II} Let $T$ be a locally finite, rooted tree with forward degrees 
$\geq 2$. Consider an ultra-metric element $\phi $ on $T$ and a fully
supported probability measure $\mu $ on $\partial T$. Then there are a
unique constant $C>0$ and a unique transient nearest neighbour random walk
on $T$ that satisfies (\ref{eq:FG}) and (\ref{eq:FF}) with the following
properties:
\end{theorem}

\begin{enumerate}
\item $\mu =\nu _{o}\,$ is the limit distribution of the random walk.

\item The associated boundary process coincides with the standard process on 
$\partial T$ induced by the ultra-metric element $C\cdot \phi $ and the
given measure~$\mu $.
\end{enumerate}

For the proof, we shall need three more formulas. The first two are taken
from \cite[Lemma 9.35]{W-Markov}, while the third is immediate from (\ref%
{eq:nu}) and (\ref{eq:FF}) 
\begin{align}
G(u,u)\,p(u,v)& =\frac{F(u,v)}{1-F(u,v)F(v,u)}\,\;\;\text{if}\;u\sim
v\,,\quad \text{and}\quad  \label{eq:formula1} \\
G(u,u)& =1+\sum_{v:v\sim u}\frac{F(u,v)F(v,u)}{1-F(u,v)F(v,u)}
\label{eq:formula2} \\
F(v^{-},v)& =\frac{\nu _{o}(\partial T_{v})/F(o,v^{-})}{%
1-F(v,v^{-})+F(v,v^{-})\,\,\nu _{o}(\partial T_{v})/F(o,v^{-})}.
\label{eq:formula3}
\end{align}

\begin{proof}[Proof of Theorem \protect\ref{thm:II}]
We proceed as follows: we start with $\phi$ and $\mu$ and replace $\phi$ by
a new ultra-metric element $C\cdot \phi$, with $C$ to be determined, and $%
\mu $ being the candidate for the limit distribution of the random walk that
we are looking for.

Using the various formulas at our disposal, we first construct in the only
possible way the quantities $F(u,v)\,$, $u,v,\in T$, in particular when $%
u\sim v$. In turn, they lead to the Green kernel $G(u,v)$. So far, these
will be only \textquotedblleft would-be\textquotedblright\ quantities whose
feasibility will have to be verified. Until that verification, we shall
denote them $\widetilde{F}(u,v)$ and $\widetilde{G}(u,v)$. Via (\ref%
{eq:formula1}), they will lead to definitions of transition probabilities $%
p(u,v)$. Stochasticity of the resulting transition matrix $\mathcal{P}$ will
also have to be verified.

Only then, we will use a potential theoretic argument to show that $%
\widetilde G(u,v)$ really is the Green kernel associated with $\mathcal{P}$,
so that the question mark that is implicit in the ``$\,\,\widetilde{\;}\,\,$%
'' symbol can be removed.

\smallskip

First of all, in view of Theorem \ref{thm:I}, we must have 
\begin{equation*}
C \cdot \phi(v) = \widetilde G(v,o)\,,
\end{equation*}
whence by (\ref{eq:FG}) and (\ref{eq:FF}) 
\begin{equation}  \label{eq:solveF1}
\widetilde F(v,v^-) = \phi(v)/\phi(v^-) \quad \text{for}\; v \in T
\setminus\{o\}\,,
\end{equation}
and more generally 
\begin{equation*}
\widetilde F(v,u) = \phi(v)/\phi(u)\quad \text{when} \; u \le v\,.
\end{equation*}
We note immediately that $0 < \widetilde F(v,u) < 1$ when $u < v$, and that $%
\widetilde F(u,u)=1$.

Next, we use (\ref{eq:formula3}) to construct recursively $\widetilde{F}%
(v^{-},v)$ and $\widetilde{F}(o,v)$. We start with $\widetilde{F}(o,o)=1$.
If $v\neq o$ and $\widetilde{F}(o,v^{-})$ is already given, with 
\begin{equation*}
\mu (\partial T_{v^{-}})\leq \widetilde{F}(o,v^{-})\leq 1
\end{equation*}
(the lower bound is required by (\ref{eq:nu})), then we have to set 
\begin{equation}
\widetilde{F}(v^{-},v)=\frac{\mu (\partial T_{v})/\widetilde{F}(o,v^{-})} {1-%
\widetilde{F}(v,v^{-})+\widetilde{F}(v,v^{-})\,\,\mu (\partial T_{v})/%
\widetilde{F}(o,v^{-})}  \label{eq:solveF2}
\end{equation}%
and 
\begin{equation*}
\widetilde{F}(o,v)=\widetilde{F}(o,v^{-})\widetilde{F}(v^{-},v)\,.
\end{equation*}
Since 
\begin{equation*}
\widetilde{F}(o,v^{-})\geq \mu (\partial T_{v^{-}})\geq \mu (\partial T_{v}),
\end{equation*}
we see that 
\begin{equation*}
0<\widetilde{F}(v^{-},v)\leq 1.
\end{equation*}
We set -- as imposed by (\ref{eq:FF}) -- 
\begin{equation*}
\widetilde{F}(o,v)=\widetilde{F}(o,v^{-})\widetilde{F}(v^{-},v).
\end{equation*}%
Formula (\ref{eq:solveF2}) (re-)transforms into 
\begin{equation}
\mu (\partial T_{v})=\widetilde{F}(o,v^{-})\widetilde{F}(v^{-},v) \frac{1-%
\widetilde{F}(v,v^{-})}{1-\widetilde{F}(v,v^{-})\widetilde{F}(v^{-},v)}\leq 
\widetilde{F}(o,v)\leq 1\,,  \label{eq:re-mu}
\end{equation}%
as needed for our recursive construction. At this point, we have all $%
\widetilde{F}(u,v) $, initially for $u\sim v$, and consequently for all $u,v$
by taking products along geodesic paths.

We now can compute the constant $C$: (\ref{eq:formula2}), combined with (\ref%
{eq:solveF1}) and (\ref{eq:re-mu}) for $u\sim o$ forces 
\begin{eqnarray*}
C\phi (o) &=&\widetilde{G}(o,o)=1+\sum_{u:u\sim o}\frac{\widetilde{F}(o,u)%
\widetilde{F}(u,o)}{1-\widetilde{F}(o,u)\widetilde{F}(u,o)} \\
&=&1+\sum_{u:u\sim o} \frac{\widetilde{F}(u,o)}{1-\widetilde{F}(u,o)}%
\mu(\partial T_{u}) \\
&=&1+\sum_{u:u\sim o}\frac{\phi (u)/\phi (o)}{1-\phi (u)/\phi (o)}
\mu(\partial T_{u})
\end{eqnarray*}%
Therefore 
\begin{equation}
C=\frac{1}{\phi (o)}+\sum_{u:u\sim o} \frac{\phi (u)/\phi (o)}{\phi
(o)-\phi(u)}\mu (\partial T_{u})\,.  \label{eq:C}
\end{equation}%
We now construct $\widetilde{G}(u,u)$ via (\ref{eq:formula2}): 
\begin{equation}
\widetilde{G}(u,u)=1+\sum_{v:v\sim u} \frac{\widetilde{F}(u,v)\widetilde{F}%
(v,u)}{1-\widetilde{F}(u,v)\widetilde{F}(v,u)}.  \label{eq:Gxx}
\end{equation}%
For $u=o$, we know that this is compatible with our choice of $C$. At last,
our only choice for the Green kernel is 
\begin{equation*}
\widetilde{G}(u,v)=\widetilde{F}(u,v)\widetilde{G}(v,v)\,,\ \ u,v\in T\,.
\end{equation*}%
Now we finally arrive at the only way how to define the transition
probabilities, via (\ref{eq:formula1}): 
\begin{equation}
p(u,v)=\frac{1}{\widetilde{G}(u,u)} \frac{\widetilde{F}(u,v)}{1-\widetilde{F}%
(u,v)\widetilde{F}(v,u)}.  \label{eq:pxy}
\end{equation}%
\emph{Claim 1.}\quad $\mathcal{P}$ is stochastic.

\smallskip

\noindent \emph{Proof of Claim 1.} Combining (\ref{eq:pxy}) with (\ref%
{eq:Gxx}), we deduce that we have to verify that for every $u\in T$, 
\begin{equation}
\sum_{v:v\sim u}\frac{\widetilde{F}(u,v)\left(1-\widetilde{F}(v,u)\right)}{1-%
\widetilde{F}(u,v)\widetilde{F}(v,u)}=1\,.  \label{eq:verify}
\end{equation}%
If $u=o$, then by (\ref{eq:re-mu}) this is just 
\begin{equation*}
\sum_{v:v\sim o}\mu (\partial T_{v})=1.
\end{equation*}%
If $u\neq o$ then, again by (\ref{eq:re-mu}), the left hand side of (\ref%
{eq:verify}) is 
\begin{equation*}
\begin{split}
\sum_{v:v^{-}=u}& \frac{\widetilde{F}(u,v)\left(1-\widetilde{F}(v,u)\right)}{%
1-\widetilde{F}(u,v)\widetilde{F}(v,u)}+\frac{\widetilde{F}(u,u^{-})\left(1-%
\widetilde{F}(u^{-},u)\right)}{1-\widetilde{F}(u,u^{-})\widetilde{F}(u^{-},u)%
} \\
& =\sum_{v:v^{-}=u}\frac{\mu (\partial T_{v})}{\widetilde{F}(o,u)}+1-\frac{1-%
\widetilde{F}(u^{-},u)}{1-\widetilde{F}(u,u^{-})\widetilde{F}(u^{-},u)}=1.
\end{split}%
\end{equation*}%
This proves Claim 1.

\smallskip

\noindent\emph{Claim 2.} For any $u_0\in T$, the function $\tilde g_{u_0}(u)
= \widetilde G(u,u_0)$ satisfies $\mathcal{P}\tilde g_{u_0} = \tilde g_{u_0}
-\mathbf{1}_{u_0}\,$.

\smallskip

\noindent \emph{Proof of Claim 2.} First, we combine (\ref{eq:Gxx}) with (%
\ref{eq:pxy}) to get 
\begin{equation*}
P\tilde{g}_{u_{0}}(u_{0})=\sum_{v:v\sim u_{0}}p(u_{0},v) \widetilde{F}%
(v,u_{0})\widetilde{G}(u_{0},u_{0}) =\sum_{v:v\sim u_{0}}\frac{\widetilde{F}%
(u_{0},v)\widetilde{F}(v,u_{0})} {1-\widetilde{F}(u_{0},v)\widetilde{F}%
(v,u_{0})}=\tilde{g}_{u_{0}}(u_{0})-1\,,
\end{equation*}%
and Claim 2 is true at $u=u_{0}\,$. Second, for $u\neq u_{0}$, let $w$ be
the neighbour of $u$ on $\pi (u,u_{0})$. Then 
\begin{eqnarray*}
P\tilde{g}_{u_{0}}(u) &=&\sum_{v:v\sim u,v\neq w}p(u,v)\widetilde{F}(v,u)%
\widetilde{G}(u,u_{0})+p(u,w)\widetilde{G}(w,u_{0}) \\
&=&\underset{\widetilde{G}(u,u)-1}{\underbrace{\sum_{v:v\sim u} \frac{%
\widetilde{F}(u,v)\widetilde{F}(v,u)}{1-\widetilde{F}(u,v)\widetilde{F}(v,u)}%
}} \frac{\widetilde{G}(u,u_{0})}{\widetilde{G}(u,u)}-p(u,w)\widetilde{F}(w,u)%
\widetilde{G}(u,u_{0})+p(u,w)\widetilde{G}(w,u_{0}) \\
&=&G(u,u_{0})\left( 1-\frac{1}{\widetilde{G}(u,u)}-p(u,w)\widetilde{F}(w,u)
+p(u,w)\frac{1}{\widetilde{F}(u,w)}\right) =\tilde{g}_{u_{0}}(u)
\end{eqnarray*}%
since 
\begin{equation*}
p(u,w)/\widetilde{F}(u,w)-p(u,w)\widetilde{F}(w,u)=1/\widetilde{G}(u,u)
\end{equation*}
by (\ref{eq:pxy}). This completes the proof of Claim 2.

\smallskip

Now we can conclude: the function $\tilde{g}_{u_{0}}$ is non-constant,
positive and superharmonic. Therefore the random walk with transition matrix 
$\mathcal{P}$ given by \ref{eq:pxy} is transient and does posses a Green
function $G(u,v)$. Furthermore, by the Riesz decomposition theorem, we have 
\begin{equation*}
\tilde{g}_{u_{0}}=Gf+h\,,
\end{equation*}%
where $h$ is a non-negative harmonic function and the charge $f$ of the
potential 
\begin{equation*}
Gf(u)=\sum_{v}G(u,v)f(v)
\end{equation*}
is $f=\tilde{g}_{u_{0}}-P\tilde{g}_{u_{0}}=\mathbf{1}_{u_{0}}\,$. That is, 
\begin{equation*}
\widetilde{G}(u,u_{0})=G(u,u_{0})+h(x)\quad \text{for all}\;u \in T.
\end{equation*}%
Now let $x \in \partial T$ and $v=u_{0}\wedge x$. If $u\in T_{v}$ then by
our construction 
\begin{equation*}
\widetilde{G}(u,u_{0})=\widetilde{G}(u,o) \frac{\widetilde{G}(v,u_{0})}{%
\phi(v)}\phi (u)\rightarrow 0\quad \text{as}\;u\rightarrow x \,.
\end{equation*}%
Therefore $\widetilde{G}(\cdot ,u_{0})$ vanishes at infinity, and the same
must hold for $h$. By the Maximum Principle, $h\equiv 0$.

We conclude that $\widetilde G(u,v) = G(u,v)$ for all $u,v\in T$. But then,
by our construction, also $\widetilde F(u,v) =F(u,v)$, the ``first hitting''
kernel associated with $\mathcal{P}$. Comparing (\ref{eq:re-mu}) with (\ref%
{eq:nu}), we see that $\mu = \nu_o\,$. This completes the proof.
\end{proof}

\subsection{The non-compact case}

Our general approach in the present work is not restricted to compact
spaces. In case of a non-compact, locally compact ultra-metric space without
isolated points, one constructs the tree in the same way: the vertex set
corresponds to the collection of all closed balls, and neighbourhood in the
resulting tree is defined as above: if a vertex $v$ corresponds to a ball $B$%
, then the predecessor $v^-$ is the vertex corresponding to the ball $%
B^{\prime }$ (see Definition \ref{predecessor}), and there is the edge $%
[v^-,v]$. Now \emph{every} vertex has a predecessor (while in the compact
case, the root vertex has none), and the tree has its root at infinity,
i.e., the ultra-metric space becomes $\partial^* T = \partial T
\setminus\{\varpi\}$, where $\varpi$ is a fixed reference end of $T$. See
Figure 5 below.

We now start with this situation: given a tree $T$ and a reference end $%
\varpi \in \partial T$, the predecessor $v^- = v^-_{\varpi}$ of a vertex $v$
with respect to $\varpi$ is the neighbour of $v$ on the geodesic $%
\pi(v,\varpi)$. Given two elements $w, z \in \widehat T \setminus \{\varpi\}$%
, their confluent $w \curlywedge z$ with respect to $\varpi$ is again
defined as the last common element on the geodesics $\pi(\varpi,w)$ and $%
\pi(\varpi,z)$, a vertex, unless $v=w \in \partial^* T$ (Figure~5). Again,
it is natural to assume that each vertex has at least two forward neighbours.

In this situation, for the Definition \ref{def:element} of an ultra-metric
element $\phi:T \to (0\,,\infty)$, we need besides monotonicity [$\phi(v) <
\phi(v^-)$] that $\phi$ tends to $\infty$ along $\pi(o,\varpi)$, while it
has to tend to $0$ along any geodesic going to $\partial^* T$. The
associated ultra-metric on $\partial^* T$ is then given in the same way as
before: 
\begin{equation*}
d_{\phi }(x,y)=%
\begin{cases}
0\,, & \text{if}\;x=y\,, \\ 
\phi (x\curlywedge y)\,, & \text{if}\;x\neq y\,.%
\end{cases}%
\end{equation*}
Let us note here that also when $\phi$ does not tend to $\infty$ along $%
\pi(x,\varpi)$, this does define an ultra-metric, but then $(\partial^* T,
d_{\phi})$ will not be complete. Also, if the inequality $\phi(v) \le
\phi(v^-)$ is not strict, one gets an ultra-metric, but then the above
construction of the tree of closed balls does not recover the original tree
from $(\partial^* T, d_{\phi})$. Finally, if $\phi$ does not tend to $0$
along some geodesic $\pi(o,x)$, $x \in \partial^* T$, then $x$ will be an
isolated point in $(\partial^* T, d_{\phi})$. (The last two observations are
also true in the compact case, for a tree with a root vertex.)

\smallskip

Returning to our setting, the reference measure $\mu$ of a $%
(\phi,\mu,\sigma) $-process may have infinite mass: a Radon measure
supported on the whole of $\partial^* T$. Again, we know that it is
sufficient to study the standard $(\phi,\mu)$-process. We give a brief
outline of the duality of such processes with random walks on $T$. This
should be compared with the final part of Kigami's second paper \cite%
{Kigami2} (whose preprint became available when the largest part of this
work had been done, and in particular, the preliminary version \cite{W-arxiv}
of the present random walk sections had first been circulated.).

With respect to $\varpi$, the branch of $T$ rooted at $u \in T$ is now 
\begin{equation*}
T_u = T_{\varpi,u} = \{ v \in T : u \in \pi(v,\varpi) \}.
\end{equation*}
Then $\partial T_u$ is a compact subset of $\partial^* T$, a ball with $%
d_{\phi}$-diameter $\phi(u)$. Here, it will be good to write $T_{o,u}$ for
the branch with respect to a root vertex $o \in T$, as defined in (\ref%
{eq:branch}). We note that $T_{\varpi,u} = T_{o,u}$ iff $u \notin
\pi(o,\varpi)$. In addition to the reference end $\varpi$, we choose such a
root $o$ and write $o_n$ for its $n$-th predecessor, that is, the vertex on $%
\pi(o,\varpi)$ at graph distance $n$ from $o$.

$$
\beginpicture 

\setcoordinatesystem units <.8mm,1.3mm>

\setplotarea x from -10 to 104, y from -14 to 46

\arrow <6pt> [.2,.67] from 2 2 to 40 40

\plot 32 32 62 2 /

 \plot 16 16 30 2 /

 \plot 48 16 34 2 /

 \plot 8 8 14 2 /

 \plot 24 8 18 2 /

 \plot 40 8 46 2 /

 \plot 56 8 50 2 /

 \plot 4 4 6 2 /

 \plot 12 4 10 2 /

 \plot 20 4 22 2 /

 \plot 28 4 26 2 /

 \plot 36 4 38 2 /

 \plot 44 4 42 2 /

 \plot 52 4 54 2 /

 \plot 60 4 58 2 /



\arrow <6pt> [.2,.67] from 99 29 to 88 40

 \plot 66 2 96 32 /

 \plot 70 2 68 4 /

 \plot 74 2 76 4 /

 \plot 78 2 72 8 /

 \plot 82 2 88 8 /

 \plot 86 2 84 4 /

 \plot 90 2 92 4 /

 \plot 94 2 80 16 /


\setdots <3pt>
\putrule from -4.8 4 to 102 4
\putrule from -4.5 8 to 102 8
\putrule from -2 16 to 102 16
\putrule from -1.7 32 to 102 32

\setdashes <2pt> \linethickness=.5pt
\putrule from -2 -5 to 102 -5

\put {$\vdots$} at 32 0
\put {$\vdots$} at 64 0

\put {$\dots$} [l] at 103 6
\put {$\dots$} [l] at 103 24

\put {$H_{-2}$} [l] at -14 32
\put {$H_{-1}$} [l] at -14 16
\put {$H_0$} [l] at -14 8
\put {$H_1$} [l] at -14 4
\put {$\partial^* T$} [l] at -14 -5
\put {$\vdots$} at -10 0
\put {$\vdots$} [B] at -10 36

\put {$\scriptstyle\bullet$} at 8 8
\put {$o$} [rb] at 7.2 8.8
\put {$\scriptstyle\bullet$} at 16 16
\put {$o_1$} [rb] at 15.2 16.8
\put {$\scriptstyle\bullet$} at 32 32
\put {$o_2$} [rb] at 31.2 32.8
\put {$\scriptstyle\bullet$} at 36 4 
\put {$u$} [rb] at 35.2 4.8
\put {$\scriptstyle\bullet$} at 56 8 
\put {$v$} [lb] at 56.8 8.8
\put {$\scriptstyle\bullet$} at 48 16 
\put {$u \curlywedge v$} [lb] at 48.8 16.8

\put {$\varpi$} at 42 42
\put {$\varpi$} at 86 42
\put{\rm Figure 5} at 47 -11

\endpicture
$$

Now let $\mathcal{P}=\left( p(u,v)\right) _{u,v\in T}$ be the transition
matrix of a transient nearest neighbour random walk on $T$. We assume once
more that (\ref{eq:F<1}) holds: $F(v,v^{-})<1$ for every $v\in T$, but now
predecessors refer to $\varpi $. (Indeed, this implies (\ref{eq:F<1}) with
respect to any choice of the root vertex.) We now consider the Dirichlet
form $\mathcal{E}_{\mathcal{HD}}$ and look at the formula of Theorem \ref%
{thm:doob-naim}. We would like to move $o$ to $\varpi $ in that formula. We
know from Lemma \ref{lem:invariant} that the measures $\Theta
_{o_{n}}(x,y)\,d\nu _{o_{n}}(x)\,d\nu _{o_{n}}(y)$ are the same for all $n$.
However, the measures $\nu _{o_{n}}$ restricted to $\partial ^{\ast }T$ will
typically converge vaguely to $0$. Thus, we normalise by defining 
\begin{equation*}
\mu _{n}=\frac{1}{\nu _{o_{n}}(\partial T_{o})}\,\nu _{o_{n}}\quad \text{%
and\ \ }J_{n}(x,y)=\Theta _{o_{n}}(x,y)\,\left( \nu _{o_{n}}(\partial
T_{o})\right) ^{2}\,.
\end{equation*}%
For the following, recall that $T_{u}=T_{\varpi ,u}\,$, and note that $%
u\curlywedge o=o_{k}$ for some $k\geq 0$.

\begin{lemma}
\label{lem:lim} Let $A \subset \partial^*T$ be compact, so that there is a
vertex $u$ such that $A \subset \partial T_u\,$. If $u \curlywedge o = o_k$
then for all $n \ge k$ and for all $x,y \in \partial T_{o_k}\,$, 
\begin{equation*}
\mu_n(A) = \mu_k(A) =: \mu(A)\,,\quad \text{and\ \ } J_n(x,y) = J_k(x,y) =:
J(x,y).
\end{equation*}
We have 
\begin{equation*}
J(x,y) = \mathfrak{j}(x \curlywedge y) \quad \text{with} \quad \mathfrak{j}%
(v) = \frac{\vartheta^2}{K(v,\varpi)^2}\,\frac{G(v,v)}{\mathsf{m}(v)}\,,\; v
\in T\,,
\end{equation*}
where 
\begin{equation*}
\vartheta = \frac{\mathsf{m}(o)\nu_o(\partial T_o)}{G(o,o)}\,, \quad \text{%
and\ \ } K(v,\varpi) = \frac{F(v, v \curlywedge o)}{F(o, v \curlywedge o)} = 
\frac{F(v, v \wedge_o \varpi)}{F(o, v \curlywedge_o \varpi)}
\end{equation*}
is the Martin kernel at $\varpi$.
\end{lemma}

\begin{proof}
Since $\partial T_{o_k} $ contains both $\partial T_o$ and $A$, we have for $%
n \ge k$ 
\begin{equation*}
\mu_n(A) = \frac{\nu_{o_n}(A)}{\nu_{o_n}(\partial T_o)} = \frac{%
F(o_n,o_k)\nu_{o_k}(A)}{F(o_n,o_k)\nu_{o_k}(\partial T_o)} = \mu_k(A).
\end{equation*}
Analogously, Let $x,y \in \partial T_{o_k}$ and $x \curlywedge y = v$, an
element of $T_{o_k}\,$. We use the identity $\mathsf{m}(v)G(v,w) = \mathsf{m}%
(w)G(w,v)$, which implies $\mathsf{m}(o_n) F(o_n,o) = \mathsf{m}(o_n)
G(o_n,o)/G(o,o) = \mathsf{m}(o) G(o,o_n)/G(o,o)$, and compute for $n \ge k$ 
\begin{eqnarray*}
J_n(x,y) &=& \frac{\nu_{o_n}(\partial T_o)^2\,\mathsf{m}(o_n)}{%
F(o_n,v)G(v,o_n)} = \nu_{o}(\partial T_o)^2\, \frac{\mathsf{m}(o_n)^2
F(o_n,o)^2}{\mathsf{m}(o_n)F(o_n,v)G(v,o_n)} \\
&=& \frac{\nu_{o}(\partial T_o)^2\mathsf{m}(o)^2}{ G(o,o)^2} \, \frac{%
G(o,o_n)^2}{G(v,o_n)^2} \, \frac{G(v,v)}{\mathsf{m}(v)}\,,
\end{eqnarray*}
which yields the proposed formula, since $G(o,o_n) = F(o,o\curlywedge
v)G(o\curlywedge v, o_n)$ and $G(v,o_n) = F(v,o\curlywedge v)G(o\curlywedge
v, o_n)$.
\end{proof}

Now its is not hard to deduce the following.

\begin{theorem}
\label{thm:non-cp-RW} Let $T$ and its reference end $\varpi$ be as outlined
above. Consider a nearest neighbour random walk on $T$ that satisfies $%
F(v,v^-)<1$ for every $v \in T$. Let $\mu$ and $J$ be as in Lemma \ref%
{lem:lim}. Then for all compactly supported continuous functions $\varphi,
\psi$ on $\partial^* T$, the Dirichlet form (\ref{eq:DirHD}) can be written
as 
\begin{eqnarray*}
\mathcal{E}_{\mathcal{HD}}(\varphi,\psi) &=& \mathcal{E}_{J}(\varphi,\psi) +
\vartheta \cdot \nu_o(\{\varpi\}) \int_{\partial^* T}
\varphi(x)\psi(x)\,d\mu(x)\,, \quad\text{where} \\
\mathcal{E}_{J}(\varphi,\psi) &=& \frac12 \int_{\partial^* T}
\int_{\partial^* T}
\left(\varphi(x)-\varphi(y)\right)\left(\psi(x)-\psi(y)\right)\, J(x,y)\,
d\mu(x)\, d\mu(y)\,.
\end{eqnarray*}
When the random walk is Dirichlet regular (in which case $\nu_o(\{\varpi\})
= 0$), the form $\mathcal{E}_{J}= \mathcal{E}_{\mathcal{HD}}$ induces the
standard $(\mu,\phi)$-process, where the ultra-metric element $\phi$ with
respect to $\varpi$ is given by 
\begin{equation*}
\phi(v) = \frac{1}{\vartheta} \, K(v,\varpi)\,,
\end{equation*}
and $\vartheta$ and the Martin kernel $K(v,\varpi)$ are as defined in Lemma %
\ref{lem:lim}.

In particular, the $(\mu,\phi)$-process is the boundary process with a
time-change.
\end{theorem}

\begin{proof}
There is $k$ such that the compact supports of $\varphi $ and $\psi $ are
contained in $\partial T_{o_{k}}\,$. Let $n\geq k$. Using lemmas \ref%
{lem:invariant} and \ref{lem:lim}, 
\begin{eqnarray*}
\mathcal{E}_{\mathcal{HD}}(\varphi ,\psi ) &=&\frac{1}{2}\int_{\partial
T}\int_{\partial T}\left( \varphi (x)-\varphi (y)\right) \left( \varphi
(x)-\varphi (y)\right) \,J_{n}(x,y)\,d\mu _{n}(x)\,d\mu _{n}(y) \\
&=&\frac{1}{2}\int_{\partial T_{o_{n}}}\int_{\partial T_{o_{n}}}\left(
\varphi (x)-\varphi (y)\right) \left( \varphi (x)-\varphi (y)\right)
\,J(x,y)\,d\mu (x)\,d\mu (y) \\
&&+\int_{\partial T_{o_{n}}}\varphi (x)\psi (x)\underset{=:f_{n}\left(
x\right) }{\underbrace{\int_{\partial T\setminus \partial
T_{o_{n}}}J_{n}(x,y)\,d\mu _{n}(y)}}\,d\mu (x)
\end{eqnarray*}%
As $n\rightarrow \infty $, 
\begin{equation*}
\frac{1}{2}\int_{\partial T_{o_{n}}}\int_{\partial T_{o_{n}}}\left( \varphi
(x)-\varphi (y)\right) \left( \varphi (x)-\varphi (y)\right) \,J(x,y)\,d\mu
(x)\,d\mu (y)\rightarrow \mathcal{E}_{J}(\varphi ,\psi )\,.
\end{equation*}%
Let us look at the second term. We have 
\begin{equation*}
f_{n}(x)=\int_{\partial T\setminus \partial T_{o_{n}}}\Theta
_{o_{n}}(x,y)\,\nu _{o_{n}}(\partial T_{o})\,d\nu _{o_{n}}(y).
\end{equation*}%
For $x\in \partial T_{o_{n}}$ and $y\in \partial T\setminus \partial
T_{o_{n}}\,$, their confluent with respect to $o_{n}$ is $o_{n}$ itself.
Therefore, using (\ref{eq:Naim}) and (\ref{eq:nu}) 
\begin{eqnarray*}
\Theta _{o_{n}}(x,y)\,\nu _{o_{n}}(\partial T_{o}) &=&\frac{\mathsf{m}(o_{n})%
}{G(o_{n},o_{n})}\,F(o_{n},o)\,\nu _{o}(\partial T_{o})=\frac{\mathsf{m}%
(o_{n})G(o_{n},o)}{G(o_{n},o_{n})G(o,o)}\,\nu _{o}(\partial T_{o}) \\
&=&\frac{\mathsf{m}(o)G(o,o_{n})}{G(o_{n},o_{n})G(o,o)}\,\nu _{o}(\partial
T_{o})=\frac{\mathsf{m}(o)}{G(o,o)}\,\nu _{o}(\partial
T_{o})\,F(o,o_{n})=\vartheta \,F(o,o_{n})\,.
\end{eqnarray*}%
Now note that for $y\in \partial T\setminus \partial T_{o_{n}}\,$, we have $%
F(o,o_{n})\,d\nu _{o_{n}}(y)=d\nu _{o}(y)$. Therefore 
\begin{equation*}
f_{n}(x)=\vartheta \int_{\partial T\setminus \partial
T_{o_{n}}}F(o,o_{n})\,d\nu _{o_{n}}(y)=\vartheta \cdot \nu _{o}(\partial
T\setminus \partial T_{o_{n}})\rightarrow \vartheta \cdot \nu _{o}(\{\varpi
\})\,,
\end{equation*}%
and as $n\rightarrow \infty \,$, we can use dominated convergence to get
that 
\begin{eqnarray*}
\int_{\partial T_{o_{n}}}\varphi (x)\psi (x)\int_{\partial T\setminus
\partial T_{o_{n}}}J_{n}(x,y)\,d\mu _{n}(y)\,d\mu (x) &=&\int_{\partial
^{\ast }T}\varphi (x)\psi (x)\,f_{n}(x)\,d\mu (x) \\
&\rightarrow &\vartheta \cdot \nu _{o}(\{\varpi \})\int_{\partial ^{\ast
}T}\varphi (x)\psi (x)\,d\mu (x)\,,
\end{eqnarray*}%
as proposed. To prove the formula for the associated ultra-metric element,
we proceed as in the proof of Theorem \ref{thm:I}, see (\ref{eq:tosolve1})
and the subsequent lines. We find that the ultra-metric element must satisfy 
\begin{equation*}
\frac{1}{\phi (v)}-\frac{1}{\phi (v^{-})}=\left( \mathfrak{j}(v)-\mathfrak{j}%
(v^{-})\right) \mu (\partial T_{v})\,.
\end{equation*}%
The right hand side of this equation can be computed: we have $%
v^{-}\curlywedge o=o_{k}$ for some $k\geq 0$, and combining the arguments
after (\ref{eq:tosolve1}) with those of the proof of Lemma \ref{lem:lim}, 
\begin{eqnarray*}
\left( \mathfrak{j}(v)-\mathfrak{j}(v^{-})\right) \mu (\partial T_{v})
&=&\left( \frac{\mathsf{m}(o_{k})}{F(o_{k},v)G(v,o_{k})}-\frac{\mathsf{m}%
(o_{k})}{F(o_{k},v^{-})G(v^{-},o_{k})}\right) \nu _{o_{k}}(\partial
T_{v})\,\nu _{o_{k}}(\partial T_{o}) \\
&=&\left( \frac{\mathsf{m}(o_{k})}{G(v,o_{k})}-\frac{\mathsf{m}(o_{k})}{%
G(v^{-},o_{k})}\right) F(o_{k},o)\,\nu _{o}(\partial T_{o}) \\
&=&\left( \frac{G(o,o_{k})}{G(v,o_{k})}-\frac{G(o,o_{k})}{G(v^{-},o_{k})}%
\right) \frac{\mathsf{m}(o)\nu _{o}(\partial T_{o})}{G(o,o)}=\frac{\vartheta 
}{K(v,\varpi )}-\frac{\vartheta }{K(v^{-},\varpi )}
\end{eqnarray*}%
We infer that $1/\phi (\cdot )-\vartheta /K(\cdot ,\varpi )$ must be
constant. By Dirichlet regularity of the random walk, $K(v,\varpi
)\rightarrow \infty $ as $v\rightarrow \varpi $. On the other hand, also $%
\phi (o_{n})$ must tend to infinity. Thus, the constant is $0$, and $\phi $
has the proposed form.
\end{proof}

Lemma \ref{lem:lim} and Theorem \ref{thm:non-cp-RW} lead to clearer insight
and simpler proofs concerning the material on random walks in \cite[\S 10 -- 
\S 11]{Kigami2}, in particular \cite[Theorem 11.3]{Kigami2}. Namely, our
limit measure $\mu$ coincides with the $\nu_{\ast}$ of \cite{Kigami2}. We
also note here, that there are examples where $\mu(\partial^* T) = \infty$,
as well as examples where $\mu(\partial^* T) < \infty$, even though the
ultra-metric space is non-compact. In the present work, we have always
assumed that the reference measure has infinite mass in the non-compact
case, but this is not crucial for our approach.

\medskip

\begin{remark}
\RM\label{rmk:mixed} In the present sections \ref{trees} -- \ref{duality},
we have always assumed that the ultra-metric space has no isolated points,
which for the tree means that $\deg ^{+}\geq 2$. Theme of \cite{BGP1} is the
opposite situation, where all points are isolated, i.e., the space is
discrete. In that case the ultra-metric space is also the boundary of a
tree, which does not consist of ends, but of terminal vertices, that is,
vertices with only one neighbour.

From the point of view of the present section, the mixed situation works
equally well. If we start with a locally compact ultra-metric space having
both isolated and non-isolated points, we can construct the tree in the same
way. The vertex set is the collection of all closed balls. The isolated
points will then become \emph{terminal vertices} of the tree, which have no
neighbour besides the predecessor, as for example the vertices $x$ and $y$
in Figure~6. All interior (non terminal) vertices will have forward degree $%
\ge 2$.

$$
\beginpicture 

\setcoordinatesystem units <.8mm,1.3mm>

\setplotarea x from -10 to 104, y from -12 to 46

\arrow <6pt> [.2,.67] from 2 2 to 40 40

\plot 32 32 48 16 /

 \plot 16 16 30 2 /


 \plot 8 8 14 2 /

 \plot 24 8 20 4 /



 \plot 4 4 6 2 /

 \plot 12 4 10 2 /


 \plot 28 4 26 2 /







\arrow <6pt> [.2,.67] from 99 29 to 88 40

 \plot 66 2 96 32 /

 \plot 70 2 68 4 /

 \plot 74 2 76 4 /

 \plot 78 2 72 8 /

 \plot 82 2 88 8 /

 \plot 86 2 84 4 /

 \plot 90 2 92 4 /

 \plot 94 2 80 16 /




\put {$\vdots$} at 8 0
\put {$\vdots$} at 88 0

\put {$\dots$} [l] at 103 6
\put {$\dots$} [l] at 103 24


\put {$\scriptstyle\bullet$} at 8 8
\put {$o$} [rb] at 7.2 8.8
\put {$\scriptstyle\bullet$} at 48 16 

\put {$\varpi$} at 42 42
\put {$\varpi$} at 86 42
\put {$\scriptstyle\bullet$} at 20 4 
\put {$x$} [rb] at 19.2 4.8
\put {$y$} [rb] at 50 17.4
\put{\rm Figure 6} at 47 -7

\endpicture
$$

In the compact case, the boundary $\partial T$ of that tree consists of the
terminal vertices together with the space of ends. In the non-compact case,
we will again have a reference end $\varpi$ as above, and $\partial^* T$
consists of all ends except $\varpi$, plus the terminal vertices. The
definition of an ultra-metric element remains the same, but we only need to
define it on interior vertices. In this general setting, the construction of 
$(\phi,\mu,\sigma)$-processes remains unchanged.

Even in presence of isolated points, the duality between $(\phi,\mu,\sigma)$%
-processes and random walks on the associated tree remains as explained
here. The random walk should then be such that the terminal vertices are
absorbing, and that the Green kernel tends to $0$ at infinity. The
Doob-Na\"\i m formula extends readily to that setting.
\end{remark}

\begin{remark}
\RM\label{rmk:rwstart} Let us again consider the general situation when we
start with a transient random walk on a locally finite, rooted tree~$T$.

The limit distribution $\nu _{o}$ will in general not be supported by the
whole of $\partial T$. The boundary process can of course still be
constructed, see \cite{Ki}, but will evolve naturally on $\limfunc{supp}(\nu
_{o})$ only. Thus, we can consider our ultra-metric space to be just $%
\limfunc{supp}(\nu _{o})$. The tree associated with this ultra-metric space
will in general not be the tree we started with, nor its \emph{transient
skeleton} as defined in \cite[(9.27)]{W-Markov} (the subtree induced by $o$
and all $v\in T\setminus \{o\}$ with $F(v,v^{-})<1$, where $v^- = v^-_o$).

The reasons are twofold. First, the construction of the tree associated with 
$\limfunc{supp}(\nu_{o})$ will never give back vertices with forward degree $%
1$. Second, some end contained in $\limfunc{supp}(\nu_{o})$ may be isolated
within that set, while not being isolated in $\partial T$. But then this
element will become a terminal vertex in the tree associated with the
ultra-metric (sub)space $\limfunc{supp}(\nu_{o})$. This occurs precisely
when the transient skeleton has isolated ends.

Thus, one should work with a modified ``reduced'' tree plus random walk in
order to maintain the duality between random walks and isotropic jump
processes. The same observations apply to the non-compact case, with a
reference end in the place of the root and the measure $\mu$ of Lemma \ref%
{lem:lim} in the place of $\nu_o\,$.
\end{remark}

\begin{remark}
\RM\label{rmk:regularity} Given a transient random walk on the rooted tree $%
T $, \cite{Ki} also recovers an \emph{intrinsic metric} of the boundary
process on $\partial T$ (compact case !) in terms of what is called an
ultra-metric element in the present paper. This is of course $\phi
(x)=G(x,o) $, denoted $D_{x}$ in \cite{Ki}, where it is shown that for $\nu
_{o}$-almost every $\xi \in \partial T$, $D_{x}\rightarrow 0$ along the
geodesic ray $\pi (o,\xi )$. This has the following potential theoretic
interpretation.

A point $x\in \partial T$ is called \emph{regular for the Dirichlet problem,}
if for every $\varphi \in C(\partial T)$, its Poisson transform $h_{\varphi
} $ satisfies 
\begin{equation*}
\lim_{v\rightarrow x}h_{\varphi }(v)=\varphi (x).
\end{equation*}%
It is known from Cartwright, Soardi and Woess~\cite[Remark 2]{CSW} that $x$
is regular if and only if $\lim_{u\rightarrow x}G(u,o)=0$ (as long as $T$
has at least $2$ ends), see also \cite[Theorem 9.43]{W-Markov}. By the
latter theorem, the set of regular points has $\nu _{o}$-measure 1. That is,
the Green kernel vanishes at $\nu _{o}$-almost every boundary point.
\end{remark}

\begin{remark}
\RM\label{rmk:vondra} In the proof of Theorem \ref{thm:II}, we have
reconstructed random walk transition probabilities from $C\cdot\phi(u) =
G(u,o)$ and $\mu = \nu_o\,$.

A similar (a bit simpler) question was addressed by Vondra{\v{c}}ek~\cite{Vo}%
: how to reconstruct the transition probabilities from \emph{all} limit
distributions $\nu_u\,$, $u \in T$, on the boundary. This, as well as our
method, basically come from (\ref{eq:nu}) and (\ref{eq:formula1}) + (\ref%
{eq:formula2}), which can be traced back to Cartier~\cite{Ca}.
\end{remark}

\section{Random walk associated with $p$-adic fractional derivative}

\setcounter{equation}{0}\label{SecFrac}

In this section we consider a two-fold specific example which unites the
approaches of Section \ref{SecQp} and Sections \ref{trees}--\ref{duality}.
We start with the compact case.

\subsection{The $p$-adic fractional derivative on $\mathbb{Z}_{p}$}

Let $\mathbb{Z}_{p}\subset \mathbb{Q}_{p}$ be the group of $p$-adic
integers. As a counterpart of the operator $\mathfrak{D}^{\alpha }$ we
introduce the operator $\mathbb{D}^{\alpha }$ of fractional derivative on $%
\mathbb{Z}_{p}\,$. We show that it is the Laplacian of an appropriate
isotropic Markov semigroup. Then we construct a random walk associated with $%
\mathbb{D}^{\alpha }$ in the sense of Sections \ref{trees}--\ref{duality}.

Since $\mathbb{Z}_{p}$ is a compact Abelian group, its dual $\widehat{%
\mathbb{Z}_{p}}$ is a discrete Abelian group. It is known that the group $%
\widehat{\mathbb{Z}_{p}}$ can be identified with the group%
\begin{equation*}
Z(p^{\infty })=\{p^{-n}m:0\leq m<p^{n},n=1,2,...\}
\end{equation*}%
equipped with addition of numbers mod $1$ as the group operation. As sets
(but not as groups) $Z(p^{\infty })\subset \mathbb{Q}_{p}\,$, whence the
function $\xi \mapsto \left\Vert \xi \right\Vert_{p}$ is well-defined on the
group $Z(p^{\infty }).$

\begin{definition}
\label{def-frac-Z-p}\RM The operator $(\mathbb{D}^{\alpha },\mathcal{V}_c),$ 
$\alpha >0,$ is defined via the Fourier transform on the compact Abelian
group $\mathbb{Z}_{p}$ by 
\begin{equation*}
\widehat{\mathbb{D}^{\alpha }f}(\xi )=\left\Vert \xi \right\Vert
_{p}^{\alpha }\widehat{f}(\xi ),\text{ }\xi \in Z(p^{\infty })\,,
\end{equation*}%
where $\mathcal{V}_c$ is the space of locally constant functions on $\mathbb{%
Z}_{p}\,$.
\end{definition}

Compare with the Definition \ref{def-frac-derivative} of the operator $%
\mathfrak{D}^{\alpha }.$

An immediate consequence 
is that the operator $\mathbb{D}^{\alpha }$ is a non-negative definite
self-adjoint operator whose spectrum coincides with the range of the function%
\begin{equation*}
\xi \mapsto \left\Vert \xi \right\Vert _{p}^{\alpha }:Z(p^{\infty
})\rightarrow \mathbb{R}_{+},
\end{equation*}%
that is, 
\begin{equation*}
\func{spec}\mathbb{D}^{\alpha }=\{0,p^{\alpha },p^{2\alpha },...\}.
\end{equation*}%
The eigenspace $\mathcal{H}(\lambda )$ of the operator $\mathbb{D}^{\alpha }$
corresponding to the eigenvalue $\lambda =p^{k\alpha },$ $k\geq 1,$ is
spanned by the function 
\begin{equation*}
f_{k}=\frac{1}{\mu _{p}(p^{k}\mathbb{Z}_{p})}\mathbf{1}_{p^{k}\mathbb{Z}%
_{p}}-\frac{1}{\mu _{p}(p^{k-1}\mathbb{Z}_{p})}\mathbf{1}_{p^{k-1}\mathbb{Z}%
_{p}}
\end{equation*}%
and\ its\ shifts\ $f_{k}(\cdot +a)$ with any $a\in \mathbb{Z}_{p}/p^{k}%
\mathbb{Z}_{p}$.

Indeed, computing the Fourier transform of the function $f_{k},$%
\begin{equation*}
\widehat{f_{k}}(\xi )=\mathbf{1}_{\{\left\Vert \xi \right\Vert _{p}\leq
p^{k}\}}-\mathbf{1}_{\{\left\Vert \xi \right\Vert _{p}\leq p^{k-1}\}}=%
\mathbf{1}_{\{\left\Vert \xi \right\Vert _{p}=p^{k}\}},
\end{equation*}%
we obtain%
\begin{equation*}
\widehat{\mathbb{D}^{\alpha }f_{k}}(\xi )=\left\Vert \xi \right\Vert
_{p}^{\alpha }\widehat{f_{k}}(\xi )=p^{k\alpha }\widehat{f_{k}}(\xi ).
\end{equation*}%
The maximal number of linearly independent functions in the set $%
\{f_{k}(\cdot +a):a\in \mathbb{Z}_{p}/p^{k}\mathbb{Z}_{p}\}$ is $p^{k-1}($ $%
p-1)$, whence 
\begin{equation*}
\dim \mathcal{H}(\lambda )=p^{k-1}(p-1).
\end{equation*}%
All the above shows that $\mathbb{D}^{\alpha }$ coincides with the Laplacian
of some isotropic Markov semigroup $(\mathbb{P}^{t}_{\alpha })_{t>0}$ on the
ultra-metric measure space $(\mathbb{Z}_{p},d_{p},\mu _{p})$. In particular,
using the complete description of $\func{spec}\mathbb{D}^{\alpha }$ we
compute the intrinsic distance, call it $d_{p,\alpha }(x,y),$ 
\begin{equation*}
d_{p,\alpha }(x,y)=\left( \frac{\left\Vert x-y\right\Vert _{p}}{p}\right)
^{\alpha }.
\end{equation*}%
It is now straightforward to compute the spectral distribution function $%
\mathbb{N}_{\alpha }(x,\tau )\equiv \mathbb{N}_{\alpha }(\tau )$ and then
the jump-kernel $\mathbb{J}_{\alpha }(x,y)\equiv $ $\mathbb{J}_{\alpha
}(x-y) $ of the operator $\mathbb{D}^{\alpha }.$ We claim that%
\begin{equation}
\mathbb{J}_{\alpha }(x,y)=\frac{p^{\alpha }-1}{1-p^{-\alpha -1}}\left( \frac{%
p^{-\alpha }-p^{-\alpha -1}}{1-p^{-\alpha }}+\frac{1}{\left\Vert
x-y\right\Vert _{p}^{1+\alpha }}\right) .  \label{jump-kernel on Z-p}
\end{equation}%
Recall for comparison that according to (\ref{jump-kernel-Q_p}) the
jump-kernel $J_{\alpha }(x,y)$ of the operator $\mathfrak{D}^{\alpha }$ is
given by%
\begin{equation*}
J_{\alpha }(x,y)=\frac{p^{\alpha }-1}{1-p^{-\alpha -1}}\frac{1}{\left\Vert
x-y\right\Vert _{p}^{1+\alpha }}.
\end{equation*}%
To prove (\ref{jump-kernel on Z-p}), we compute $\mathbb{J}_{\alpha }(z).$
Let $\left\Vert z\right\Vert _{p}=p^{-l},$ then $d_{p,\alpha
}(0,z)=p^{-(l+1)\alpha }$ and%
\begin{equation*}
\mathbb{J}_{\alpha }(z)=\int\limits_{0}^{1/d_{p,\alpha }(0,z)}\mathbb{N}%
_{\alpha }(\tau )d\tau =\int\limits_{0}^{p^{(l+1)\alpha }}\mathbb{N}_{\alpha
}(\tau )d\tau .
\end{equation*}%
The function $\mathbb{N}_{\alpha }(\tau )$ is a non-decreasing,
left-continuous staircase function having jumps at the points $\tau
_{k}=p^{k\alpha },$ $k=1,2,...,$ and taking values at these points $\mathbb{N%
}_{\alpha }(\tau _{k})=p^{k-1},$ whence%
\begin{eqnarray*}
\mathbb{J}_{\alpha }(z) &=&1\cdot p^{\alpha }+p(p^{2\alpha }-p^{\alpha
})+p^{2}(p^{3\alpha }-p^{2\alpha })+...+p^{l}(p^{(l+1)\alpha }-p^{l\alpha })
\\
&=&\frac{1-p^{-1}}{1-p^{-\alpha -1}}+\frac{p^{\alpha }-1}{1-p^{-\alpha -1}}%
p^{l(\alpha +1)} \\
&=&\frac{p^{\alpha }-1}{1-p^{-\alpha -1}}\left( \frac{p^{-\alpha
}-p^{-\alpha -1}}{1-p^{-\alpha }}+\frac{1}{\left\Vert z\right\Vert
_{p}^{1+\alpha }}\right)
\end{eqnarray*}%
as desired. Next, we apply Theorem \ref{Thm-Dirichlet form/Laplacian} and
obtain%
\begin{equation}
\mathbb{D}^{\alpha }f(x)=\int_{\mathbb{Z}_{p}}\left( f(x)-f(y)\right) 
\mathbb{J}_{\alpha }(x-y)\,d\mu _{p}(y).  \label{frac.deriv. on Z-p}
\end{equation}%
The equations (\ref{jump-kernel on Z-p})--(\ref{frac.deriv. on Z-p}) and (%
\ref{sing-int}) now yield the following result.

\begin{corollary}
\label{fractional derivatives} For any function $f$ defined on $\mathbb{Z}%
_{p}\subset \mathbb{Q}_{p}$ we set $\widetilde{f}=f$ on $\mathbb{Z}_{p}$ and 
$0$, otherwise. Then 
\begin{equation*}
f\in \func{dom}(\mathbb{D}^{\alpha })\Longrightarrow \widetilde{f}\in \text{ 
}\func{dom}(\mathfrak{D}^{\alpha }),\text{ }
\end{equation*}%
\begin{equation}
\mathbb{D}^{\alpha }f(x)=\mathfrak{D}^{\alpha }\widetilde{f}(x)\text{ \ \
and\ \ \ }(\mathbb{D}^{\alpha }f,f)=\left( \mathfrak{D}^{\alpha }\widetilde{f%
}\text{ },\widetilde{f}\text{ }\right)  \label{frac.derivetives Z-p/Q-p}
\end{equation}%
whenever $x\in \mathbb{Z}_{p},$ $f\in \func{dom}(\mathbb{D}^{\alpha })$ and $%
(\mathbf{1},f)=0$.
\end{corollary}

\subsection{Nearest neighbour random walk on the rooted tree $\mathbb{T}%
_{p}^{o}$}

As an illustration of Theorem \ref{thm:II} we construct a random walk on the
rooted tree associated with $\mathbb{Z}_p$ whose boundary process coincides
with the isotropic process driven by the operator $C\cdot \mathbb{D}^{\alpha
},$ where $C=p^{-\alpha }(1-p^{-\alpha })$.

The Abelian group $\mathbb{Z}_{p}$ can be identified with the boundary of
the tree $\mathbb{T}_{p}^o$ with root $o$ where every vertex $v$ has $p$
forward neighbours. In our identification, this is the tree of balls of the
ultra-metric space $(\mathbb{Z}_{p},d_{p})$ with root $o$ corresponding to
the whole of $\mathbb{Z}_{p}$ and the ultra-metric $d_{p}(x,y)=\left\Vert
x-y\right\Vert_{p}.$ See Figure~4 above, where $p=2$. We fix a constant $%
c\in (0,1)$ and consider the nearest neighbour random walk on $\mathbb{T}%
_{p}^o$ with%
\begin{equation}
p(v,v^{-})=1-c\text{ \ \ and\ \ }p(v^{-},v)=\left\{ 
\begin{array}{cc}
1/p & \text{if }v^{-}=o \\ 
c/p & \text{otherwise}%
\end{array}%
\right. \,.  \label{our r.w. on z_p}
\end{equation}%
Using \cite[Thm. 1.38 and Prop. 9.3 ]{W-Markov} one can compute precisely
the Green function $G(v,o)$, the hitting probability $F(v,o)$ and other
quantities associated with our random walk. In particular, choosing $%
c=(1+p^{-\alpha })^{-1},$ we obtain%
\begin{equation}
F(v,o)=p^{-\alpha \left\vert v\right\vert }\text{ \ \ and\ \ }G(v,o)=\frac{%
p^{-\alpha \left\vert v\right\vert }}{1-p^{-\alpha }},
\label{F(v,o) and G(v,o)}
\end{equation}%
where $\left\vert v\right\vert $ is the graph distance from $v$ to $o.$ We
see that the Green function vanishes at infinity, whence the random walk is
Dirichlet regular.

The transition probabilities are invariant with respect to all automorphisms
of the tree. Every such automorphism must fix $o$ and every level of the
tree. Let $\nu =\nu _{o}$ be the limit distribution on $\partial \mathbb{T}%
_{p}^o$ of the random walk starting at $o$. Then also $\nu$ is invariant
under the automorphism group of the tree (whose action extends to the
boundary). In particular it is invariant under the action of $\mathbb{Z}%
_{p}. $ Thus, under the identification of $\partial \mathbb{T}_{p}^o$ with $%
\mathbb{Z}_{p},$ we have that $\nu =\mu _{p}\,$, the normalized Haar measure
of $\mathbb{Z}_{p}\,$.

We now look at the boundary process induced by our random walk as a jump
process on $\mathbb{Z}_{p}\,$. By Theorem \ref{thm:I}, the boundary process
arises as an isotropic jump process with the reference measure $\mu _{p}.$
Let $\mathcal{L}$ be its Laplacian. By Corollary \ref{eigenvalues -boundary
process}, the set $\func{spec}\mathcal{L}$ coincides with the range of the
function $v\mapsto 1/G(v,o)$, $v\in \mathbb{T}_{p}^{o}\,$, together with $%
\{0\}$. In view of the above formula for $G(v,o)$ this means that 
\begin{equation*}
\func{spec}\mathcal{L}=\{0,(1-p^{-\alpha }),p^{\alpha }(1-p^{-\alpha
}),p^{2\alpha }(1-p^{-\alpha }),\dots \}.
\end{equation*}%
Remember that%
\begin{equation*}
\func{spec}\mathbb{D}^{\alpha }=\{0,p^{\alpha },p^{2\alpha },\dots \}=\frac{%
p^{\alpha }}{1-p^{-\alpha }}~\func{spec}\mathcal{L}.
\end{equation*}%
Since both $\mathbb{D}^{\alpha }$ and $\mathcal{L}$ have the same
orthonormal basis of eigenfunctions, we conclude that they are proportional,
that is,%
\begin{equation}
\mathbb{D}^{\alpha }=\frac{p^{\alpha }}{1-p^{-\alpha }}~\mathcal{L}.
\label{Laplacians}
\end{equation}%
Thus, finally we come to the following conclusion

\begin{proposition}
\label{bdry process for d_alpha}The boundary process $\{X_{t}\}_{t>0}$
associated with the random walk defined in \emph{(\ref{our r.w. on z_p})}
with parameter $c=(1+p^{-\alpha })^{-1}$ and the isotropic jump process $%
\{X_{t}^{\alpha }\}_{t>0}$ driven by the operator $\mathbb{D}^{\alpha }$ are
related by the linear time change $X_{t/C}=X_{t}^{\alpha }$, where $%
C=p^{-\alpha }(1-p^{-\alpha }).$
\end{proposition}

The equation (\ref{Laplacians}) implies that the jump kernels $\mathbb{J}%
_{\alpha }(x,y)$ and $\Theta _{o}(x,y)$ of the operators $\mathbb{D}^{\alpha
}$ and $\mathcal{L}$, respectively, are related by 
\begin{equation}
\mathbb{J}_{\alpha }(x,y)=\frac{p^{\alpha }}{1-p^{-\alpha }}~\Theta
_{o}(x,y).  \label{jump kernels}
\end{equation}%
We now show how to compute the Na\"{\i}m kernel 
\begin{equation*}
\Theta _{o}(x,y)=\frac{1}{G(o,o)F(o,v)F(v,o)},\text{ \ \ where }v=x\wedge y,
\end{equation*}%
directly, using the data of (\ref{F(v,o) and G(v,o)}). We do not yet have $%
F(o,v).$ We shall compute 
\begin{equation*}
N(v)=\frac{1}{F(o,v)F(v,o)}.
\end{equation*}%
Since it depends only on the level $k$ of $v,$ we consider an arbitrary
geodesic ray $[o=v_{0},v_{1},...]$ and set up a linear recursion for $%
N(v_{k}).$ Denoting by $w_{1}$ an arbitrary neighbour of $o$ different from $%
v_{1}$ and applying \cite[Prop. 9.3(b)]{W-Markov} and (\ref{F(v,o) and
G(v,o)}), we obtain 
\begin{equation*}
F(o,v_{1})=\frac{1}{p}+\frac{p-1}{p}F(w_{1},o)F(o,v_{1})=\frac{1}{p}+\frac{%
(p-1)p^{-\alpha }}{p}F(o,v_{1}),
\end{equation*}%
whence

\begin{equation*}
F(o,v_{1})=\frac{p^{\alpha }}{p^{\alpha +1}-p+1}.
\end{equation*}%
Thus, we get the initial values 
\begin{equation*}
N(v_{0})=1\quad \text{and}\quad N(v_{1})=p^{\alpha +1}-p+1.
\end{equation*}%
Next, for $k\geq 1,$ we let $w_{k+1}$ be a forward neighbour of $v_{k}$
different from $v_{k+1}.$ Applying once again \cite[Prop. 9.3(b)]{W-Markov}
and (\ref{F(v,o) and G(v,o)}), we obtain%
\begin{eqnarray*}
F(v_{k},v_{k+1}) &=&\frac{p^{\alpha }}{p(p^{\alpha }+1)}+\frac{%
(p-1)p^{\alpha }}{p(p^{\alpha }+1)}F(w_{k+1},v_{k})F(v_{k},v_{k+1}) \\
&&+\frac{1}{p^{\alpha }+1}F(v_{k-1},v_{k})F(v_{k},v_{k+1}).
\end{eqnarray*}%
We insert the value $F(w_{k+1},v_{k})=p^{-\alpha }$ and divide by 
\begin{equation*}
F(o,v_{k+1})=F(o,v_{k})F(v_{k},v_{k+1})=F(o,v_{k-1})F(v_{k-1},v_{k})F(v_{k},v_{k+1}).
\end{equation*}%
Then we get%
\begin{equation*}
\frac{1}{F(o,v_{k})}=\frac{p^{\alpha }}{p(p^{\alpha }+1)}\frac{1}{%
F(o,v_{k+1})}+\frac{p-1}{p(p^{\alpha }+1)}\frac{1}{F(o,v_{k})}+\frac{1}{%
p^{\alpha }+1}\frac{1}{F(o,v_{k-1})}.
\end{equation*}%
Now we multiply both sides with $1/F(v_{k},o)=p^{\alpha k}$ and get%
\begin{equation*}
N(v_{k})=\frac{1}{p(p^{\alpha }+1)}N(v_{k+1})+\frac{p-1}{p(p^{\alpha }+1)}%
N(v_{k})+\frac{p^{\alpha }}{p^{\alpha }+1}N(v_{k-1}).
\end{equation*}%
This is a homogeneous second order linear recursion with constant
coefficients. Its characteristic polynomial has roots $1$ and $p^{\alpha
+1}. $ Therefore 
\begin{equation*}
N(v_{k})=A+Bp^{(\alpha +1)k}.
\end{equation*}%
Inserting the initial values we easily find the values of $A$ and $B.$ In
order to get the Na\"{\i}m kernel, we have to multiply by $%
1/G(o,o)=1-p^{-\alpha }.$ Thus, we get%
\begin{equation*}
\Theta _{o}(x,y)=\frac{(1-p^{-\alpha })(p-1)}{p^{\alpha +1}-1}+\frac{%
(1-p^{-\alpha })(p^{\alpha +1}-p)}{p^{\alpha +1}-1}p^{(\alpha +1)k}=\frac{%
1-p^{-\alpha }}{p^{\alpha }}\,\mathbb{J}_{\alpha }(x,y)\,,
\end{equation*}%
as desired.

\subsection{The random walk corresponding to $\mathfrak{D}^{\protect\alpha }$
on $\mathbb{Q}_{p}\,$}

We can combine the preceding considerations with the material of Lemma \ref%
{lem:lim} and Theorem \ref{thm:non-cp-RW} concerning the duality with random
walks in the non-compact case. It is now easy to understand the random walk
corresponding to the fractional derivative on the whole of $\mathbb{Q}_p\,$.

The tree associated with $\mathbb{Q}_{p}$ is the homogenous tree $T=\mathbb{T%
}_{p}$ with degree $p+1$. We have to choose a reference end $\varpi $. Then
we can identify its lower boundary $\partial ^{\ast }\mathbb{T}_{p}$ with
the field of $p$-adic numbers. With respect to $\varpi $, every vertex $v$
has its predecessor $v^{-}$ and $p$ successors. Every subtree $%
T_{v}=T_{\varpi ,v}$ is isomorphic with the rooted tree $\mathbb{T}_{p}^{o}$
considered above in the compact case of the $p$-adic integers. In
particular, we choose the root vertex $o$ such that $\partial T_{o}=\mathbb{Z%
}_{p}\,$. See Figure~5 above, where $p=2$.

We now define the random walk on $\mathbb{T}_p$ as in (\ref{our r.w. on z_p}%
), but with predecessors referring to $\varpi\,$: 
\begin{equation}
p(v,v^{-})=1-c \quad\text{and\ \ } p(v^{-},v) = c/p\,, \quad \text{where\ \ }
c = (1+p^{-\alpha})^{-1}.  \label{our r.w. on q_p}
\end{equation}%
For the following quantities, see e.g. \cite[pp. 423-424]{W-lamp}. For all $%
v \in \mathbb{T}_p\,$, 
\begin{equation*}
F(v,v^-) = p^{-\alpha}\,,\ \ F(v^-,v) = p^{-1}\,,\ \ G(v,v) = \frac{%
1+p^{-\alpha}}{1-p^{-\alpha-1}}\quad \text{and}\quad \nu_v(\partial T_v) = 
\frac{1-p^{-\alpha}}{1-p^{-\alpha-1}}.
\end{equation*}
This yields that the reference measure $\mu$ of the boundary process with
respect to $\varpi$, as given by Lemma \ref{lem:lim}, is the standard Haar
measure of $\mathbb{Q}_p\,$.

We compute $\vartheta =(1-p^{-\alpha })/(1+p^{-\alpha }).$ Furthermore, let
us set $\mathfrak{h}(v)=d(v,v\curlywedge o)-d(o,v\curlywedge o)$ (where $d$
is the graph metric). This is the \emph{horocycle number} of $v$. That is,
the vertices with $\mathfrak{h}(v)=k$, $k\in \mathbb{Z}$, are the elements
in the $k$-th generation $H_{k}$ of the tree (see Figure~5), and $\partial
T_{v}$ corresponds to a ball with radius $p^{-k}$ in the standard
ultra-metric of $\mathbb{Q}_{p}\,$. Then 
\begin{equation*}
K(v,\varpi )=p^{\alpha \,\mathfrak{h}(v)}\quad \text{and}\quad \mathsf{m}%
(v)=p^{(\alpha -1)\mathfrak{h}(v)}.
\end{equation*}%
Putting things together, we get 
\begin{equation*}
\phi (v)=\frac{1+p^{-\alpha }}{1-p^{-\alpha }}\,p^{-\alpha \,\mathfrak{h}%
(v)}\quad \text{and}\quad \mathfrak{j}(v)=\frac{(1-p^{-\alpha })^{2}}{%
(1+p^{-\alpha })(1-p^{-\alpha -1})}\,p^{(\alpha +1)\,\mathfrak{h}(v)}
\end{equation*}%
Retranslating this into $p$-adic notation, we conclude that the intrinsic
metric and jump kernel of the boundary process with respect to $\varpi $ are
given by 
\begin{eqnarray*}
d_{\phi }(x,y) &=&\frac{1-p^{-\alpha }}{1+p^{-\alpha }}\,\left\Vert
x-y\right\Vert _{p}^{\alpha }\quad \text{and} \\
J(x,y) &=&\frac{(1-p^{-\alpha })^{2}}{(1+p^{-\alpha })(1-p^{-\alpha -1})}\,%
\frac{1}{\left\Vert x-y\right\Vert _{p}^{\alpha +1}}=\frac{1-p^{-\alpha }}{%
p^{\alpha }+1}\,J_{\alpha }(x,y)\,,
\end{eqnarray*}%
where $J_{\alpha }$ is the jump kernel associated with $\mathfrak{D}^{\alpha
}$. So at last, we get the following.

\begin{proposition}
\label{bdry process} The boundary process $\{X_{t}\}_{t>0}$ with respect to
the reference end $\varpi $ associated with the random walk \emph{(\ref{our
r.w. on q_p})} on $\mathbb{T}_{p}$ and the isotropic jump process $%
\{X_{t}^{\alpha }\}_{t>0}$ driven by the operator $\mathfrak{D}^{\alpha }$
on $\mathbb{Q}_{p}$ are related by the linear time change $X_{t/C^{\ast
}}=X_{t}^{\alpha }$, where $C^{\ast }=(1-p^{-\alpha })/(p^{\alpha }+1)$.
\end{proposition}


\baselineskip 12.5pt

\providecommand{\bysame}{\leavevmode\hbox to3em{\hrulefill}\thinspace}
\providecommand{\MR}{\relax\ifhmode\unskip\space\fi MR }
\providecommand{\MRhref}[2]{%
  \href{http://www.ams.org/mathscinet-getitem?mr=#1}{#2}
}
\providecommand{\href}[2]{#2}

\bigskip

\noindent A. Bendikov: Institute of Mathematics, Wroclaw University \newline
Pl. Grundwaldzki 2/4, 50-384 Wroc{\l }aw, Poland \newline
email: bendikov@math.uni.wroc.pl

\medskip

\noindent A. Grigor'yan: Department of Mathematics, University of Bielefeld, 
\newline
33501 Bielefeld, Germany \newline
email: grigor@math.uni-bielefeld.de

\medskip

\noindent Ch. Pittet: LATP, Universit\'{e} d'Aix-Marseille, 39 rue
Joliot-Curie, \newline
13453 Marseille Cedex 13, France \newline
email: pittet@cmi.univ-mrs.fr

\medskip

\noindent W. Woess: Institut f\"ur Mathematische Strukturtheorie, Technische
Universit\"at Graz \newline
Steyrergasse 30, A-8010 Graz, Austria \newline
email: woess@TUGraz.at

\end{document}